\documentclass[11pt,a4paper]{article}

\usepackage[utf8]{inputenc}
\usepackage[T1]{fontenc}
\usepackage[english]{babel}
\usepackage{amsmath,amssymb,amsthm,amsfonts}
\usepackage{mathrsfs}
\usepackage{mathtools}
\usepackage{graphicx}
\usepackage{geometry}
\usepackage{hyperref}
\usepackage{xcolor}
\usepackage{enumitem}
\usepackage{booktabs}

\hypersetup{
  colorlinks=true,
  linkcolor=blue!50!black,
  citecolor=blue!50!black,
  urlcolor=blue!50!black,
}

\theoremstyle{plain}
\newtheorem{theorem}{Theorem}[section]
\newtheorem{proposition}[theorem]{Proposition}
\newtheorem{lemma}[theorem]{Lemma}
\newtheorem{corollary}[theorem]{Corollary}
\newtheorem{conjecture}[theorem]{Conjecture}

\theoremstyle{definition}
\newtheorem{definition}[theorem]{Definition}
\newtheorem{example}[theorem]{Example}

\theoremstyle{remark}
\newtheorem{remark}[theorem]{Remark}
\newtheorem{question}[theorem]{Question}

\makeatletter
\newenvironment{proofsketch}[1][\textit{Proof sketch}]{%
  \par\pushQED{\qed}%
  \normalfont \topsep6\p@\@plus6\p@\relax
  \trivlist
  \item[\hskip\labelsep #1\@addpunct{.}]\ignorespaces
}{%
  \popQED\endtrivlist\@endpefalse
}
\makeatother

\newcommand{\R}{\mathbb{R}}
\newcommand{\C}{\mathbb{C}}
\newcommand{\Z}{\mathbb{Z}}
\newcommand{\HH}{\mathbb{H}}
\newcommand{\bO}{\mathbb{O}}
\newcommand{\sph}[1]{S^{#1}}
\newcommand{\Hol}{\operatorname{Hol}}

\newcommand{\tr}{\operatorname{tr}}

\newcommand{\dvol}{\,d\mathrm{vol}}
\newcommand{\YM}{\mathrm{YM}}
\newcommand{\YMH}{\mathrm{YM\text{-}H}}
\newcommand{\ClassFunc}{\mathcal{C}}        % the classification functor
\newcommand{\PathGpd}{\Pi_1}                % fundamental groupoid
\newcommand{\BGroup}[1]{B(#1)}              % delooping/classifying groupoid
\newcommand{\Curv}{F}
\newcommand{\Connection}{A}
\newcommand{\Lcx}{L_\xi^{\C}}              % complexified line bundle L_xi (x) C
\newcommand{\YMHmin}{\YMH_{\min}}          % minimum YMH energy
\newcommand{\YMmin}{\YM_{\min}}            % minimum Yang-Mills energy in topological sector
\newcommand{\adP}{\mathrm{ad}\,P}          % adjoint bundle of P
\newcommand{\Gr}{\operatorname{Gr}}        % Grassmannian
\newcommand{\Att}{\mathrm{Att}}            % attention bilinear form
\newcommand{\Sym}{\operatorname{Sym}}      % symmetric part
\newcommand{\Anti}{\operatorname{Anti}}    % antisymmetric part
\newcommand{\Op}{\mathcal{O}}              % Matérn operator
\newcommand{\inj}{\operatorname{inj}}      % injectivity radius

\title{Yang--Mills--Higgs Classification: \\
       A Geometric Theory of Binary Labels on Non-Contractible Spaces}
\author{Catalin Vasii\thanks{Ness Digital Engineering.
Email: \texttt{cvasii23@gmail.com}.}}
\date{\today}

\begin{document}
\maketitle

% =====================================================================
\begin{abstract}
We reformulate binary classification on a manifold $M$ as a
Yang--Mills--Higgs variational problem. Labelled data is encoded as a
functor $\ClassFunc\colon\PathGpd(M)\to\BGroup{\Z_2}$ from the
fundamental groupoid of $M$ to the one-object groupoid with
automorphism group $\Z_2$, with monodromy class
$[\ClassFunc]\in H^1(M,\Z_2)$ a topological obstruction to
realising the classifier by a sign function. The labels live in
$\Z_2$ and the classification obstruction is a $\Z_2$-bundle; the
structure group $G$ of the geometrically realising bundle is a
separate, richer choice that controls curvature, equal to $\Z_2$
in flat examples ($\sph{1}$ M\"obius, $T^2$ XOR), to $U(1)$ for
the $\sph{2}$ Dirac monopole, and to $SU(2)$ for the $\sph{4}$
BPST instanton. The classifier-section $\phi$ and the
connection $A$ jointly minimise the Yang--Mills--Higgs energy
$\int_M\|D_A\phi\|^2 + \int_M\|\Curv_A\|^2$ subject to hard
data conditions $\phi(x_i)=y_i$. The first sector carries the
classification content; the second is bounded below in each
topological class by the Bogomolny inequality and selects the
gauge background, recovering Paper~1's harmonic interpolation as
the contractible-base reduction.

Two structural payoffs follow. First, the bundle's curvature 2-form
$\Curv_A$ supplies a dictionary with transformer attention: by the
plaquette formula $\Curv_A(x)$ is a pairwise, antisymmetric,
algebra-valued form on tangent directions --- the structural role an
attention head occupies --- and we match it, as a stipulated
dictionary rather than a derived identity, to the antisymmetric
component of the attention bilinear at $x$, read against a reference
frame in the fibre; $\|\Curv_A(x)\|^2$ on 2-planes provides the
geometric content of softmax priority ranking, and the
abelian/non-abelian split of $\Curv_A$ corresponds to the
single-head/multi-head split of attention.
Second, the XOR problem on $T^2$ is realised by the covariantly
harmonic section of the double-M\"obius bundle, which the
variational selector picks out: solving the regularised capacitance
system on all four flat $\Z_2$-bundles gives a strict energy
ordering favouring the double-M\"obius class, robust across
regularisation parameters and data geometries. The same MLP trained
on the same data, on $4$ or on $1000$ points,
finds a structurally different boundary that ignores the toroidal
identifications. The framework gives the geometrically privileged
classifier; gradient descent gives one of many interpolants. We
work throughout with the connection either flat or pinned to the
Bogomolny moduli of an imposed sector, so the curvature is
data-decoupled by construction; whether freeing it makes curvature
respond to label proximity is posed as an open question.
\end{abstract}

% =====================================================================
\section{Introduction}\label{sec:intro}

A companion paper \cite{Vasii2026Paper1} proved that binary
classification, in the flat abelian regime, is classical potential
theory: classifiers are sections of a line bundle, training labels are
Dirichlet conditions, the SVM kernel is a Green's function, and
gradient descent is a degenerate iterative approximation to a single
linear solve. Paper~1 also worked the first non-contractible
examples \cite[\S 9]{Vasii2026Paper1}: $\sph{1}$ M\"obius
classifiers, the $\sph{2}$ Dirac monopole with forced curvature
$\int F = 2\pi k$, and a Poincar\'e--Hopf argument forcing zeros.
The present paper extends the framework along five axes that
Paper~1 did not develop: (i) a cohomological obstruction theory ---
the functor $\mathcal{F}\colon\PathGpd(M)\to\BGroup{\Z_2}$ and the
classification class $[c_y]\in H^1(M,\Z_2)$; (ii) a variational
\emph{selector} among bundles, which picks the bundle from the data
in the regime where it is selectable; (iii) the Wilson observable
as universal prediction mechanism, handling the case where no
sign-function classifier exists; (iv) the $\sph{4}$/non-abelian
rung, with $SU(2)$ instantons as the BPS gauge backgrounds; and
(v) the curvature--attention dictionary identifying the
antisymmetric part of attention with the Yang--Mills--Higgs
curvature 2-form. The examples Paper~1 worked are revisited in
\S\S\ref{sec:S1-mobius}--\ref{sec:S2-monopole} under the present
framework; the new content there is the variational structure and
the obstruction-theoretic reading, not the explicit constructions,
which we reuse and cite from Paper~1.

Two restrictions of Paper~1 deserve their own paper. The first is
the assumption that the data manifold $M$ is contractible. On
$\R^n$ every bundle is trivial, every flat connection has trivial
holonomy, and a global classifier $f\colon M\to\R$ separates the
two classes by its sign. None of this survives the passage to a
non-contractible base. On $\sph{1}$, half the line bundles are
M\"obius and force an odd number of decision-boundary points; on
$\sph{2}$, oriented rank-two bundles are classified by an integer
Euler class and a non-zero class forces $\int F = 2\pi k$, so no
flat connection exists; on the torus, even the existence of a
globally consistent sign assignment is a cohomological question.
The flat picture is the contractible exception, not the rule.

The second restriction is that the structure group is abelian. The
non-abelian case is not a marginal complication; it is the regime where
the curvature 2-form $\Curv = d\Connection + \Connection\wedge \Connection$
acquires its non-trivial second term, where the holonomy of a depth-$L$
network fails to commute step-by-step, and where the only
geometrically natural readout is no longer the sign of a scalar but the
character of a holonomy element. The non-abelian regime is also where
recent machine learning architecture lives. We will show that the
curvature 2-form of the Yang--Mills--Higgs connection is, in a precise
sense, the geometric object that the attention mechanism of a
transformer discretises (\S\ref{sec:attention}).

The present paper develops the theory for arbitrary smooth base $M$
and arbitrary compact structure group $G$, with binary labels in
$\Z_2$. Three principles organise the framework.

\begin{itemize}[leftmargin=2em]
\item[(P1)] \textit{The decision boundary is emergent, not fundamental.}
The primary object is a functor
$\ClassFunc \colon \PathGpd(M) \to \BGroup{\Z_2}$ from the fundamental
groupoid of $M$ to the one-object groupoid $\BGroup{\Z_2}$ whose
automorphisms are $\Z_2$. A geometric decision boundary
$\Gamma\subset M$ exists if and only if the monodromy class
$[\ClassFunc]\in H^1(M,\Z_2)$ vanishes; contractibility of $M$ is a
sufficient but not necessary condition. On manifolds with non-trivial
$H^1$, certain functors --- the M\"obius case is the canonical
example --- have no sign-function realisation, and the functor
itself, not its zero locus, carries the classification content. This
is the content of Theorem~\ref{thm:main}.

\item[(P2)] \textit{In tier 1, the bundle is the output, not the
input.} A labelled dataset $D$ on $M$ with distinct data points
does not pin down a class $\xi\in H^1(M,\Z_2)$ at the cohomological
level --- every class supports a continuous section of $L_\xi$
matching the labels (\S\ref{ssec:obstruction}). On bases where
$H^1(M,\Z_2)$ is non-trivial, the variational principle of
\S\ref{sec:yang-mills-higgs} selects among classes by minimum
Yang--Mills--Higgs energy: the framework jointly minimises
$\int\|\Curv_A\|^2 + \int\|D_A\phi\|^2$ over connections and
sections, subject to $\phi(x_i)=y_i$, and picks the bundle class
of lowest minimum. The \emph{classification obstruction}
$[c_y]\in H^1(M,\Z_2)$ is the class of the variationally selected
bundle. On bases with $H^1(M,\Z_2)=0$ (such as $\sph{n}$ for
$n\geq 2$, where higher Chern classes carry the topology), the
selector sees a single $H^1$-class and the tier-2 bundle ---
$c_1\in H^2(M,\Z)$ or $c_2\in H^4(M,\Z)$ --- is treated as a
Problem~B configuration in the sense of
Remark~\ref{rem:problem-A-B}: imposed by the problem
specification (continuous boundary data with non-trivial winding,
or an externally chosen topological sector) rather than selected
from finite data. See Remark~\ref{rem:problem-A-B} for the precise
scope of the selector and the Problem~A versus Problem~B distinction.

\item[(P3)] \textit{Classification is gauge-covariant harmonic
interpolation; its curvature is attention.} The optimal classifier
$(\Connection,\phi)$ jointly minimises the Yang--Mills--Higgs energy
$\int\|\Curv_A\|^2+\int\|D_A\phi\|^2$ over connections and sections
subject to the hard data conditions $\phi(x_i)=y_i$. The matter
sector carries the classification content (it couples to the data);
the Yang--Mills sector selects, within the bundle's topological
class, the finite-dimensional Bogomolny moduli of (anti-)self-dual
connections on which the matter sector then runs (the two roles are
hierarchically distinct, not competing terms; matter-only
minimisation degenerates in tier 2, as discussed in
Remark~\ref{rem:matter-vs-ymh}). The curvature $\Curv_A$ of the
resulting connection plays two roles: as an antisymmetric pairwise
bilinear form on tangent directions valued in $\adP$, it is the
feature-interaction object that attention discretises; as a
norm-squared spectrum on 2-planes in $T_xM$, it ranks the
directions of preferred holonomy accumulation along which labels
are propagated. The single-head versus multi-head distinction of
transformer attention is the abelian versus non-abelian distinction
in the structure of $\Curv_A$ (Definition~\ref{thm:attention}).
\end{itemize}

These principles cover the abstract theory. To make the framework
concrete we work three rungs of an example ladder, each obtained by
placing labelled points on a sphere of dimension where a non-trivial
characteristic class lives.

\begin{itemize}[leftmargin=2em]
\item \textbf{$\sph{1}$, two labelled points (\S\ref{sec:S1-mobius}).}
The base is the circle, the obstruction lives in $H^1(\sph{1},\Z_2)=\Z_2$,
and odd labelings select the M\"obius line bundle. The flat connection
$\Connection = -\tfrac12\,d\alpha$ has holonomy $-I$; one forced
boundary point at $\alpha=\pi$; prediction is by holonomy parity, not
kernel evaluation.

\item \textbf{$\sph{2}$, two labelled points (\S\ref{sec:S2-monopole}).}
With $c_1=1$ the bundle is the Hopf line bundle, every section is
forced to vanish, and by the $\Z_2\times\Z_2$ symmetry of the
two-point configuration the forced zero sits on the equator. The
topological constraint pins the connection to the Dirac monopole
$\Curv = \tfrac12\,\omega_{\sph{2}}$, with uniform curvature ---
maximum-entropy attention. The covariant Dirichlet energy of $\phi$
then selects the harmonic interpolant in this monopole background.
Three labelled points break the symmetry and the matter current
deforms the gauge background; the forced zero moves toward the
unpaired class.

\item \textbf{$\sph{4}$, two labelled points (\S\ref{sec:S4-instanton}).}
With $c_2=1$ the structure group is $SU(2)$, no flat connection
exists, and the optimal solution is a BPST instanton centred at the
midpoint geodesic with scale $\lambda\propto d(x_+,x_-)$. More labelled
points give the ADHM matrix equation $\Delta^\dagger\Delta>0$ with
data constraints --- the non-abelian capacitance equation
(Theorem~\ref{thm:instanton}).
\end{itemize}

The relationship to Paper~1 is sharp: Paper~1 is
\S\ref{sec:reduction-paper1} of the present paper, established as a
theorem in four lines (M~contractible $\Rightarrow$ holonomy trivial
$\Rightarrow$ functor trivialises $\Rightarrow$ harmonic interpolant).
The framework of harmonic interpolation, the kernel as Green's
function, and the decision boundary as zero equipotential is not
imported as background but derived as the degenerate
contractible-base, flat-connection, abelian-group case.

\paragraph{The starting intuition.} A standard feedforward layer
$h_{\ell+1}=\sigma(W_\ell h_\ell + b_\ell)$, stripped of activation
and bias, is the linear map $h_{\ell+1}=W_\ell h_\ell$ --- a group
action of $W_\ell$ on $h_\ell$. A deep network composes these into
the product $W_{L-1}\cdots W_0$, which is exactly the form of a
discrete-path holonomy of length~$L$. The activation $\sigma$ and
bias $b_\ell$ are essential for training and for expressivity on a
contractible base; \S\ref{sec:torus} shows that on a topologically
non-trivial base the nonlinearity is absorbed into the bundle and
the residual architectural content is the parallel-transport
product. The framework below makes this precise: the geometrically
privileged classifier on a data manifold $M$ is the section selected
by gauge-covariant harmonic interpolation, and depth-$L$ networks
can be read as a discrete approximation of it along $L$-step paths,
with the cover-adapted product structure made explicit in
Remark~\ref{rem:holonomy-accumulation}.

\paragraph{What this paper says about deep learning.} The
framework's worked examples connect to deep-learning architecture
in two structural claims, developed precisely in the body and
synthesised in \S\ref{ssec:synthesis} at the end of \S\ref{sec:attention}.
First, attention is geometry, not statistics: the curvature 2-form
of the variationally selected connection is a pairwise,
antisymmetric, algebra-valued form on tangent directions
(Theorem~\ref{thm:plaquette}), the structural role attention
occupies; Definition~\ref{thm:attention} of \S\ref{sec:attention}
makes the correspondence explicit as a dictionary. Second --- and most concretely --- XOR is
topological, not architectural. The classical Minsky--Papert critique
of single-layer perceptrons rests on the impossibility of realising
XOR by a linear classifier on $\R^2$. On $T^2$ the same XOR pattern
is realised by a \emph{linear} (single-Fourier-mode) classifier on a
topologically non-trivial bundle (\S\ref{sec:torus},
\S\ref{sec:numerical-torus}), with numerical confirmation against an
MLP trained on the same data. The nonlinearity historically
attributed to the classifier is in fact the topology of the bundle
the data lives on.

\paragraph{What this paper does not claim.} The framework is
theoretical, not algorithmic: we do not propose replacing gradient
descent in practice. Definition~\ref{thm:attention} \emph{stipulates}
a dictionary between attention weights and curvature components; it
is not derived from the structure of the query--key product, and we
do not claim that trained transformers compute Yang--Mills--Higgs
connections.

Three further limitations should be stated at the outset. First, the
variational selector operates on $H^1(M,\Z_2)$: the tier-2 bundles
of \S\S\ref{sec:S2-monopole}--\ref{sec:S4-instanton} are imposed as
Problem~B configurations, not selected from finite data
(Remark~\ref{rem:problem-A-B}). Second, the instanton scale relation
$\lambda\propto d$ is supported numerically but not proved, and the
energy landscape in $\lambda$ flattens as $d\to 0$
(Remark~\ref{rem:lambda-status}). Third --- and most importantly for
the attention reading --- the connection is throughout either flat
or pinned to the Bogomolny moduli, so curvature is data-decoupled by
construction; whether a freely varying connection would concentrate
curvature in response to label proximity is posed as
Question~\ref{q:joint-minimiser} and is not settled here.

Multi-class classification is treated only via
one-versus-all reduction; the genuine multi-class theory, which
requires either $S_n$-equivariant structure groups or labels in
$\Z_n$, is deferred. An Adams-ladder framing of the worked examples
($\R$ via Paper~1; $\C$ via the $\sph{2}$ monopole; $\HH$ via the
$\sph{4}$ instanton, with $\bO$ rung~3 open) is invoked as orienting
context in \S\ref{ssec:two-tiers} and \S\ref{ssec:synthesis};
the theorems that would justify it as a structural claim ---
monotonicity of rung reductions, minimal-architecture statements at
each rung, the rung-3 octonionic case --- are the subject of a
separate forthcoming paper.

\paragraph{Organisation.} The paper has three parts and an
appendix. Part~1
(\S\S\ref{sec:functor}--\ref{sec:reduction-paper1}) develops the
abstract framework: the functor and the cohomological obstruction
(\S\ref{sec:functor}), the Wilson observable as universal readout,
the Yang--Mills--Higgs system, and the reduction to Paper~1. Part~2
(\S\S\ref{sec:S1-mobius}--\ref{sec:attention}) works the three rungs
of the example ladder and states the curvature-as-attention theorem
precisely. Part~3 (\S\S\ref{sec:torus}--\ref{sec:numerical-torus})
addresses classification without a boundary (the torus) and reports
a numerical experiment comparing the framework's closed-form
prediction against a trained MLP.
Appendix~\ref{app:proximity-proof} proves the two-point case of the
matter-sector proximity scaling.

% =====================================================================
\part*{Part 1. The Abstract Framework}
\addcontentsline{toc}{part}{Part 1. The Abstract Framework}

% ---------------------------------------------------------------------
\section{The functor $\ClassFunc \colon \PathGpd(M) \to \BGroup{\Z_2}$}
\label{sec:functor}

This section develops the primary object of the paper: a functor
$\ClassFunc\colon\PathGpd(M)\to\BGroup{\Z_2}$ encoding a binary
classification on $M$. We build it in four steps. \S\ref{ssec:groupoid}
recalls the fundamental groupoid. \S\ref{ssec:BZ2} describes the target
$\BGroup{\Z_2}$ and what a functor between these two groupoids
concretely is. \S\ref{ssec:data-as-functor} formulates labelled data
as a partial specification of such a functor.
\S\ref{ssec:thm-main} states and proves Theorem~\ref{thm:main}: the
existence-and-uniqueness result for the functor, and the precise
condition under which a sign-function realisation
$f\colon M\to\R$ exists. The condition is not contractibility of $M$
(which is too strong); it is the vanishing of an associated
$H^1(M,\Z_2)$ class.

% ---------------------------------------------------------------------
\subsection{The fundamental groupoid}\label{ssec:groupoid}

Let $M$ be a path-connected smooth manifold. The \emph{fundamental
groupoid} of $M$, denoted $\PathGpd(M)$, is the category whose objects
are the points of $M$ and whose morphisms $x\to y$ are homotopy classes
$[\gamma]$ of continuous paths $\gamma\colon[0,1]\to M$ with
$\gamma(0)=x$ and $\gamma(1)=y$, homotopies fixing endpoints.
Composition is concatenation of paths (reparametrised to $[0,1]$);
identities are constant paths; inverses are paths reversed. Every
morphism is invertible, making $\PathGpd(M)$ a \emph{groupoid} ---
``a group with many objects.''

Three features matter for what follows.

\begin{enumerate}[label=(\alph*),leftmargin=2em]
\item For any $x\in M$, the automorphism group of $x$ in
$\PathGpd(M)$ is the fundamental group $\pi_1(M,x)$. Thus
$\PathGpd(M)$ contains every fundamental group of $M$, at every
basepoint, simultaneously.

\item For any pair $x,y\in M$, the morphism set
$\mathrm{Hom}_{\PathGpd(M)}(x,y)$ is a $\pi_1(M,x)$-torsor (a set
carrying a free and transitive action of $\pi_1(M,x)$; equivalently,
a principal homogeneous space for $\pi_1(M,x)$): any two
morphisms $x\to y$ differ by precomposition with a unique loop at $x$.
There is no canonical morphism from $x$ to $y$ when $\pi_1(M,x)\neq 1$,
but the difference between any two is well-defined as a group element.

\item $\PathGpd(M)$ is a \emph{connected} groupoid: any two objects are
connected by some morphism. For a connected groupoid $\mathcal{G}$ and
any groupoid $\mathcal{H}$, the natural-isomorphism classes of functors
$\mathcal{G}\to\mathcal{H}$ correspond to homomorphisms
$\mathrm{Aut}_\mathcal{G}(x_0)\to\mathrm{Aut}_\mathcal{H}(h)$, for a
chosen object $x_0$ and a choice of target object $h\in\mathcal{H}$
up to isomorphism, \emph{taken up to conjugation} in
$\mathrm{Aut}_\mathcal{H}(h)$. In the case of interest here,
$\mathcal{H}=\BGroup{\Z_2}$ has a single object and abelian
automorphism group, so both refinements are vacuous --- conjugation
acts trivially and the target object is forced --- and the
correspondence is a bijection with $\mathrm{Hom}(\pi_1(M),\Z_2)$ as
used below. (For general non-abelian $\mathcal{H}$ the conjugation
quotient is genuinely needed.) (Existence: fix a system of paths
$\{\gamma_x\colon x_0\to x\}_{x\in\mathcal{G}}$; choose object
assignments $F(x)\in\mathrm{Obj}(\mathcal{H})$ freely, fixing
$F(x_0)$; choose morphisms
$h_x\in\mathrm{Mor}_\mathcal{H}(F(x_0),F(x))$ freely with
$h_{x_0}=\mathrm{id}$; then, given a homomorphism $\rho$ on
automorphisms at $x_0$, define a functor on morphisms by
$F([\delta\colon x\to y]):=h_y\cdot \rho([\gamma_y^{-1}\cdot
\delta\cdot\gamma_x])\cdot h_x^{-1}$. Uniqueness up to natural
isomorphism: any other system of paths, object assignments, or
choices of $h_x$ yields a naturally isomorphic functor.)
\end{enumerate}

\begin{example}\label{ex:groupoid-examples}
For $M=\R^n$, every pair of points has a unique homotopy class of paths
between them: $\PathGpd(\R^n)$ is the \emph{indiscrete} groupoid on
the underlying set. For $M=\sph{1}$, the morphism set
$\mathrm{Hom}(x,y)$ is a $\Z$-torsor, indexed by the winding number of
the path. For $M=T^2$, it is a $\Z^2$-torsor.
\end{example}

% ---------------------------------------------------------------------
\subsection{The target groupoid $\BGroup{\Z_2}$}\label{ssec:BZ2}

For any group $G$, the notation $BG$ denotes the groupoid with one
object, written $\bullet$, whose automorphism group is $G$. Thus
$\BGroup{\Z_2}$ has a single object $\bullet$ and two morphisms
$\bullet\to\bullet$: the identity $\mathrm{id}$ and the non-identity
element $\sigma$, with composition law $\sigma\cdot\sigma=\mathrm{id}$.

A functor $\ClassFunc\colon\PathGpd(M)\to\BGroup{\Z_2}$ consists of:
\begin{itemize}[leftmargin=2em]
\item an assignment of every object of $\PathGpd(M)$ to $\bullet$
(forced: $\BGroup{\Z_2}$ has only one object);
\item for each morphism $[\gamma]\colon x\to y$ in $\PathGpd(M)$, an
element $\ClassFunc([\gamma])\in\Z_2$, satisfying the functoriality
relations $\ClassFunc([\gamma_2\cdot\gamma_1])=\ClassFunc([\gamma_2])\cdot
\ClassFunc([\gamma_1])$ and $\ClassFunc(\mathrm{id}_x)=\mathrm{id}$.
\end{itemize}

By feature (c) of \S\ref{ssec:groupoid}, such a functor is determined,
up to natural isomorphism, by a group homomorphism
$\rho_{\ClassFunc}\colon\pi_1(M)\to\Z_2$. Since $M$ is path-connected
and $\Z_2$ is abelian, $\mathrm{Hom}(\pi_1(M,x_0),\Z_2)$ is canonically
independent of the choice of basepoint $x_0$: any path
$\gamma\colon x_0\to x_0'$ induces an isomorphism
$\pi_1(M,x_0)\cong\pi_1(M,x_0')$ via conjugation by $[\gamma]$, and
abelianness of $\Z_2$ makes the induced map on $\mathrm{Hom}(-,\Z_2)$
the identity. We write $\pi_1(M)$ without basepoint when only the
homomorphism set into an abelian group is meant.

\begin{lemma}\label{lem:functor-classification}
For path-connected $M$, the set of functors
$\PathGpd(M)\to\BGroup{\Z_2}$ up to natural isomorphism is in bijection
with $\mathrm{Hom}(\pi_1(M),\Z_2)\cong H^1(M,\Z_2)$, the bijection
sending $\ClassFunc$ to its monodromy homomorphism $\rho_{\ClassFunc}$.
\end{lemma}

\begin{proof}
We chain three identifications.

\emph{Step 1: functors mod natural iso $\leftrightarrow$ monodromy
homomorphisms.} A natural transformation between two functors
$\ClassFunc,\ClassFunc'\colon\PathGpd(M)\to\BGroup{\Z_2}$ assigns to
each $x\in M$ an element $\eta_x\in\Z_2$ with the naturality square
$\ClassFunc'([\delta])=\eta_y\cdot\ClassFunc([\delta])\cdot\eta_x^{-1}$
for every $\delta\colon x\to y$. Restricting to loops ($x=y$) and
using that $\Z_2$ is abelian, this becomes
$\ClassFunc'([\ell])=\ClassFunc([\ell])$. Hence natural
transformations preserve monodromy. Conversely, given two functors
with the same monodromy, fix any system of reference paths
$\{\gamma_x\colon x_0\to x\}$ and define
$\eta_x:=\ClassFunc'([\gamma_x])\cdot\ClassFunc([\gamma_x])^{-1}$. To
check the naturality square, decompose any morphism $\delta\colon
x\to y$ as $[\delta]=[\gamma_y]\cdot[\gamma_y^{-1}\cdot\delta\cdot
\gamma_x]\cdot[\gamma_x]^{-1}$, where the bracket
$[\gamma_y^{-1}\cdot\delta\cdot\gamma_x]$ is a loop at $x_0$.
Functoriality of $\ClassFunc'$ and monodromy agreement
$\rho_{\ClassFunc'}=\rho_{\ClassFunc}$ on this loop give
\[
   \ClassFunc'([\delta]) \;=\; \ClassFunc'([\gamma_y])\cdot
   \ClassFunc([\gamma_y])^{-1}\cdot\ClassFunc([\delta])\cdot
   \ClassFunc([\gamma_x])\cdot\ClassFunc'([\gamma_x])^{-1}
   \;=\; \eta_y\cdot\ClassFunc([\delta])\cdot\eta_x^{-1},
\]
so $\eta$ is a natural isomorphism. (A different system of reference
paths yields a different but equally valid natural isomorphism;
existence is what matters.) This gives the bijection
$\{\ClassFunc\}/{\sim}\leftrightarrow
\mathrm{Hom}(\pi_1(M,x_0),\Z_2)$. Existence of a functor for any
$\rho$ is feature (c) of \S\ref{ssec:groupoid}.

\emph{Step 2: $\mathrm{Hom}(\pi_1(M,x_0),\Z_2)\cong
\mathrm{Hom}(H_1(M,\Z),\Z_2)$.} Any homomorphism from $\pi_1$ to an
abelian group factors through the abelianisation
$\pi_1^{\mathrm{ab}}=H_1(M,\Z)$ by the Hurewicz theorem
\cite[Theorem~2A.1]{Hatcher}.

\emph{Step 3: $\mathrm{Hom}(H_1(M,\Z),\Z_2)\cong H^1(M,\Z_2)$.} The
universal coefficient theorem for cohomology
\cite[Theorem~3.2]{Hatcher} gives a split short
exact sequence
\[
  0 \to \mathrm{Ext}(H_0(M,\Z),\Z_2) \to H^1(M,\Z_2)
    \to \mathrm{Hom}(H_1(M,\Z),\Z_2) \to 0.
\]
Since $M$ is path-connected, $H_0(M,\Z)=\Z$ is free abelian, and
$\mathrm{Ext}(\Z,\Z_2)=0$. The sequence collapses to the asserted
isomorphism.
\end{proof}

\begin{remark}\label{rem:flat-bundle}
A functor $\ClassFunc\colon\PathGpd(M)\to\BGroup{\Z_2}$ is equivalent
data to a flat principal $\Z_2$-bundle over $M$.\footnote{This flat
$\Z_2$-bundle is the label-level obstruction, not the bundle that
carries the classifier. Paper~1 \cite{Vasii2026Paper1} realises a
classification by a section of an $O(2)$ (or $\R^*$) bundle whose
sign reproduces the labels; the functor $\ClassFunc$ records only the
$\Z_2$ monodromy of that labelling --- whether a sign-function
realisation can exist. The relation of that classifier-carrying
bundle to the present framework's variational bundle is
Proposition~\ref{prop:reduction-paper1}.} The bundle's holonomy
along a loop $\gamma$ is precisely $\rho_{\ClassFunc}([\gamma])$. The
two non-isomorphic flat $\Z_2$-bundles on $\sph{1}$ --- trivial and
M\"obius --- correspond respectively to the trivial and non-trivial
functors $\PathGpd(\sph{1})\to\BGroup{\Z_2}$.
\end{remark}

% ---------------------------------------------------------------------
\subsection{Labelled data as a partial functor}\label{ssec:data-as-functor}

A labelled dataset on $M$ is a finite set
$D=\{(x_i,y_i)\}_{i=1}^N\subset M\times\Z_2$, $N\geq 2$. We identify
$\Z_2$ multiplicatively with $\{+1,-1\}$ throughout.

To make the labels into the data of a partial functor, fix a root
$x_1\in D$ (without loss of generality $y_1=+1$, after global relabelling
if necessary) and, for each $i\geq 2$, choose a path
$\gamma_i\colon x_1\to x_i$. The labelling then determines values
\[
   \ClassFunc([\gamma_i]) \;=\; y_i\cdot y_1^{-1} \;=\; y_i \;\in\;\Z_2,
   \qquad i=2,\dots,N,
\]
on the chosen morphisms. A functor $\ClassFunc\colon\PathGpd(M)\to
\BGroup{\Z_2}$ \emph{realises} the labelled data $D$ along the chosen
paths $\{\gamma_i\}$ if it takes these prescribed values.

\begin{remark}\label{rem:path-dependence}
The values $\ClassFunc([\gamma_i])$ depend on the choice of path
$\gamma_i$. A different choice $\gamma_i'\colon x_1\to x_i$ gives
$\ClassFunc([\gamma_i'])=\rho_{\ClassFunc}([\gamma_i'\cdot\gamma_i^{-1}])
\cdot\ClassFunc([\gamma_i])$, differing by the monodromy along the
loop $\gamma_i'\cdot\gamma_i^{-1}$. Strictly this loop is based at
$x_i$ rather than at $x_1$; we apply $\rho_{\ClassFunc}$ to it by
conjugating with $\gamma_i$, which is harmless because $\Z_2$ is
abelian (inner automorphisms act trivially). When $\pi_1(M)=1$, the
path choice is immaterial; in general, ``the label at $x_i$'' is a
function of the path used to transport it from $x_1$, not of the
point $x_i$ alone. This path-dependence is the geometric content of
the framework, not a defect: when the labels are genuinely
path-dependent (e.g.\ the M\"obius case of \S\ref{sec:S1-mobius}),
the classification cannot be represented by a function $M\to\Z_2$,
but the functor $\ClassFunc$ remains a well-defined object.
\end{remark}

% ---------------------------------------------------------------------
\subsection{Theorem~\ref{thm:main}: statement and proof}\label{ssec:thm-main}

We first make precise what it means for a continuous function to
``induce'' a functor.

\begin{definition}[Induced functor of a continuous function]\label{def:induced-functor}
Let $f\colon M\to\R$ be a continuous function. Extend the sign
function to all of $\R$ by declaring $\mathrm{sign}\,0 := +1$ (any
fixed convention works; this choice is immaterial for what follows).
The \emph{induced functor}
$\ClassFunc_f\colon\PathGpd(M)\to\BGroup{\Z_2}$ is defined on every
morphism $[\delta\colon x\to y]$ by
\[
   \ClassFunc_f([\delta]) \;=\;
   \mathrm{sign}\,f(y)\cdot \mathrm{sign}\,f(x)^{-1} \;\in\;\Z_2.
\]
This is well-defined and functorial: on a loop $\ell$ at $x$,
$\ClassFunc_f([\ell])=
\mathrm{sign}\,f(x)\cdot\mathrm{sign}\,f(x)^{-1}=\mathrm{id}$, so the
monodromy $\rho_{\ClassFunc_f}$ is trivial. We say $f$
\emph{realises} a functor $\ClassFunc$ if $\ClassFunc_f\cong
\ClassFunc$ as functors (equivalently, by
Lemma~\ref{lem:functor-classification},
$[\ClassFunc_f]=[\ClassFunc]$ in $H^1(M,\Z_2)$) and
$\mathrm{sign}\,f(x_i)=y_i$ for each data point.
\end{definition}

\begin{remark}[The induced functor is always trivial]\label{rem:induced-trivial}
By the calculation in Definition~\ref{def:induced-functor},
$\ClassFunc_f$ depends only on the signs of $f$ at the endpoints, not
on the path: $\ClassFunc_f([\delta\colon x\to y])=\mathrm{sign}\,f(y)
\cdot\mathrm{sign}\,f(x)^{-1}$. Hence $\rho_{\ClassFunc_f}\equiv
\mathrm{id}$ and $[\ClassFunc_f]=0\in H^1(M,\Z_2)$ for every
continuous $f$. The realisation condition $\ClassFunc_f\cong
\ClassFunc$ therefore forces $[\ClassFunc]=0$. Among the
$H^1(M,\Z_2)$-torsor of realising functors of a given labelled
dataset, exactly one --- the trivial-monodromy one --- admits a
sign-function realisation. Theorem~\ref{thm:main}(ii) below is then
in essence the statement that this one always exists when the labels
permit it; the content lives in the construction
($\Leftarrow$ direction).
\end{remark}

\begin{theorem}[Functor realisation and the sign-function obstruction]\label{thm:main}
Let $M$ be a path-connected smooth manifold (assumed paracompact, so
that a smooth partition of unity subordinate to any open cover
exists) and $D=\{(x_i,y_i)\}_{i=1}^N$ a binary labelled dataset on
$M$ with $y_i\in\Z_2$ and distinct points $x_i$, root $x_1$ with
$y_1=+1$, and chosen paths $\gamma_i\colon x_1\to x_i$ for $i\geq 2$.
Then:
\begin{enumerate}[label=(\roman*),leftmargin=2em]
  \item \textbf{Existence and freedom.} There exists a functor
  $\ClassFunc\colon\PathGpd(M)\to\BGroup{\Z_2}$ realising $D$ along
  $\{\gamma_i\}$. The set of such functors, up to natural isomorphism,
  is a torsor over $H^1(M,\Z_2)$: any two realising functors differ
  by a unique cohomology class.
  \item \textbf{Sign-function obstruction.} Fix one functor
  $\ClassFunc$ realising $D$ along $\{\gamma_i\}$ (which exists by
  (i)). A continuous function $f\colon M\to\R$ with
  $\mathrm{sign}\,f(x_i)=y_i$ for all $i$, whose induced functor is
  $\ClassFunc$, exists if and only if the monodromy class
  $[\ClassFunc]\in H^1(M,\Z_2)$ of
  Lemma~\ref{lem:functor-classification} vanishes.
  \item \textbf{Reduction to Paper~1.} When $H^1(M,\Z_2)=0$ and the
  structure group of the bundle is abelian with trivial bundle, the
  Yang--Mills--Higgs problem of \S\ref{sec:yang-mills-higgs} reduces
  to the harmonic interpolation problem of
  \cite[Theorem~13.3]{Vasii2026Paper1};
  see Proposition~\ref{prop:reduction-paper1}.
\end{enumerate}
\end{theorem}

\begin{proof}
\textit{Part (i).} For existence, fix any homomorphism
$\rho\colon\pi_1(M,x_1)\to\Z_2$ and define a functor as follows. For
each $x\in M\setminus\{x_1\}$, fix a reference path
$\gamma_x\colon x_1\to x$, choosing $\gamma_x=\gamma_i$ when $x=x_i$;
set $\gamma_{x_1}$ to be the constant path. For any morphism
$[\delta]\colon x\to y$, define
\[
   \ClassFunc([\delta]) \;=\;
   y_y \cdot \rho([\gamma_y^{-1}\cdot\delta\cdot\gamma_x])\cdot y_x^{-1},
\]
where $y_x:=y_i$ if $x=x_i$ and $y_x:=+1$ otherwise. The bracket
$\gamma_y^{-1}\cdot\delta\cdot\gamma_x$ is a loop at $x_1$, so $\rho$
applies. Functoriality and the prescribed values
$\ClassFunc([\gamma_i])=y_i$ are direct verifications. Thus, for each
homomorphism $\rho$ and each choice of reference paths, the
construction yields a functor whose natural isomorphism class is
determined by $\rho$ alone (the choices of $\gamma_x$ for $x\notin
\{x_1,\dots,x_N\}$ affect the representative but not the natural
isomorphism class, by Lemma~\ref{lem:functor-classification}).

By Lemma~\ref{lem:functor-classification}, the natural isomorphism
classes of functors are in bijection with
\[
   \mathrm{Hom}(\pi_1(M,x_1),\Z_2) \;\cong\; H^1(M,\Z_2),
\]
and every class contains a realising representative (just constructed).
For the torsor
structure: given two realising functors $\ClassFunc,\ClassFunc'$ with
associated monodromy homomorphisms $\rho,\rho'$, their pointwise
product $\rho'\cdot\rho^{-1}$ is again a homomorphism
$\pi_1(M,x_1)\to\Z_2$ (since $\Z_2$ is abelian), giving a class in
$H^1(M,\Z_2)$. This action is free and transitive on natural
isomorphism classes of realising functors. Geometrically, the action
of a class $[\xi]\in H^1(M,\Z_2)$ on a realising functor $\ClassFunc$
is the tensor product $\ClassFunc\otimes\xi$ in the category of flat
$\Z_2$-bundles --- twisting the realising functor by the flat bundle
representing $[\xi]$.

\textit{Part (ii).} ($\Rightarrow$) By Remark~\ref{rem:induced-trivial},
$[\ClassFunc_f]=0$ for every continuous $f$. The realisation condition
$\ClassFunc_f\cong\ClassFunc$ therefore forces $[\ClassFunc]=0$.

($\Leftarrow$) Suppose $[\ClassFunc]=0$. We construct $f$ explicitly
by partition of unity. Choose disjoint open neighbourhoods
$U_i\ni x_i$ ($i=1,\dots,N$). Let $\{\varphi_0,\varphi_1,\dots,
\varphi_N\}$ be a smooth partition of unity on $M$ subordinate to the
cover $\{M\setminus\{x_1,\dots,x_N\},U_1,\dots,U_N\}$
\cite[Theorem~2.13]{Lee}, normalised so
that $\varphi_i(x_j)=\delta_{ij}$. Define
\[
   f(x) \;:=\; \sum_{i=1}^N \varphi_i(x)\, y_i \;\in\;\R.
\]
Then $f$ is continuous and $f(x_i)=y_i$, so
$\mathrm{sign}\,f(x_i)=y_i$. By the forward direction just proved,
$[\ClassFunc_f]=0$. Since $[\ClassFunc]=0$, the iso classes match:
$\ClassFunc_f\cong\ClassFunc$. Hence $f$ realises $\ClassFunc$ in the
sense of Definition~\ref{def:induced-functor}.

\textit{Part (iii).} See Proposition~\ref{prop:reduction-paper1} in
\S\ref{sec:reduction-paper1}.
\end{proof}

\begin{remark}[Structure of the decision boundary]\label{rem:boundary-structure}
The partition-of-unity construction in part (ii)
($\Leftarrow$) produces a continuous $f$ but says nothing about the
geometry of its zero set $\Gamma:=f^{-1}(0)$. In particular $\Gamma$
need not be a regular hypersurface: $f$ may vanish on open subsets
of $M$ between the data points. When $[\ClassFunc]=0$, any $f$ from
part (ii) gives a continuous decomposition
$M=f^{-1}((0,\infty))\,\cup\,f^{-1}(0)\,\cup\,f^{-1}((-\infty,0))$.
Different $f$'s give different decompositions, all agreeing at the
data points but differing elsewhere. There is no canonical
``positive side'' or ``negative side'' of $M$ defined by the functor
alone: $\ClassFunc([\delta_x])$ depends on the natural-isomorphism
representative even though the class $[\ClassFunc]$ does not, so
side-assignments to non-data points require a choice.

The variational principle of \S\ref{sec:yang-mills-higgs} selects a
geometrically preferred $f$ --- the harmonic interpolant of
\cite[Theorem~13.3]{Vasii2026Paper1} in the abelian flat case, the
Yang--Mills--Higgs minimiser in general --- whose zero set
$\Gamma:=f^{-1}(0)$ is then the geometric decision boundary. The
partition-of-unity $f$ is an existence witness; the variational $f$
is geometric.
\end{remark}

\begin{remark}[The role of contractibility; the two-tier obstruction]\label{rem:not-contractible}
Theorem~\ref{thm:main}(ii) gives the precise condition for a
sign-function realisation: triviality of
$[\ClassFunc]\in H^1(M,\Z_2)$. Contractibility of $M$ is sufficient
(since $H^1$ vanishes on contractible $M$) but not necessary. For
example, $H^1(\sph{2},\Z_2)=0$, so on $\sph{2}$ every functor is
realised by a sign function --- notwithstanding non-contractibility.
The genuine obstruction to sign-function realisation lives entirely
in $H^1(M,\Z_2)$.

This means Theorem~\ref{thm:main}(iii) and
Proposition~\ref{prop:reduction-paper1} are a strict generalisation
of \cite{Vasii2026Paper1}: the reduction works whenever
$H^1(M,\Z_2)=0$ and the bundle is trivial flat abelian, which on
$\sph{2}$ is a non-trivial case never addressed in Paper~1.

A complementary obstruction lives in \emph{higher} cohomology and
controls forced curvature rather than sign-function realisability;
this gives the framework a two-tier structure laid out explicitly
in \S\ref{ssec:two-tiers} below.
\end{remark}

\begin{remark}[Classification without a decision boundary]\label{rem:emergent-boundary}
When $[\ClassFunc]\neq 0$, Theorem~\ref{thm:main}(ii) says no
continuous sign function exists, so the classification has no
geometric decision boundary in $M$. Prediction at a point $x\in M$
along a chosen path $\delta_x\colon x_1\to x$ is still defined:
it is the value $y_1\cdot\ClassFunc([\delta_x])$, where
$\ClassFunc([\delta_x])$ is computed as the parallel
transport\footnote{At this stage of the paper there is no
connection: the structure group is the discrete group $\Z_2$, and a
flat $\Z_2$-bundle is equivalent data to a double cover
$\widetilde{M}\to M$. ``Parallel transport along $\delta_x$'' here
means the unique lift of $\delta_x$ to a path in $\widetilde{M}$
starting at the chosen lift of $x_1$, together with the resulting
endpoint in the fibre over $x$; equivalently, the action of
$\rho_{\ClassFunc}([\gamma_y^{-1}\cdot\delta_x\cdot\gamma_{x_1}])$
on the fibre. The language ``parallel transport'' is used because
once the structure group is upgraded to a continuous group
($O(2),SU(2),\ldots$) in \S\ref{sec:yang-mills-higgs} and the
$\Z_2$-bundle is replaced by an associated rank-2 real bundle with
a flat connection, the parallel transport of that connection
genuinely computes the same thing. The discrete and continuous
versions coincide in the flat regime; they part ways when curvature
appears (tier~2 of \S\ref{ssec:two-tiers}).} of the label $y_1$
along $\delta_x$ in the associated flat $\Z_2$-bundle
(Remark~\ref{rem:flat-bundle}). The answer depends on the choice of
path when the monodromy is non-trivial. The M\"obius classifier on
$\sph{1}$ (\S\ref{sec:S1-mobius}) and the toroidal XOR on $T^2$
(\S\ref{sec:torus}) are the canonical examples; in both, the
decision boundary in $M$ is absent but the classification is
perfectly defined as a functor.
\end{remark}

\begin{example}[The M\"obius case in detail]\label{ex:mobius-thm-main}
Take $M=\sph{1}$, $D=\{(x_+,+1),(x_-,-1)\}$ with $x_\pm$ at angles
$0,\pi$. The homotopy classes of paths $x_+\to x_-$ form a $\Z$-torsor
(Example~\ref{ex:groupoid-examples}), one class for each winding
number, but a $\Z_2$-valued functor only distinguishes their parity.
Up to parity, the two equivalence classes are the ``short way''
$\gamma_s$ (the half-circle through angles $0\to\pi$) and the
``long way'' $\gamma_\ell=\gamma_s\cdot[\ell]$, where $[\ell]$
generates $\pi_1(\sph{1},x_+)=\Z$.

By Theorem~\ref{thm:main}(i), realising functors are parametrised by
$\mathrm{Hom}(\Z,\Z_2)=\Z_2$. The two choices are:
\begin{itemize}[leftmargin=2em]
\item $\rho([\ell])=\mathrm{id}$: trivial functor on $\pi_1$,
$[\ClassFunc]=0$. Realisable by, e.g., $f(\alpha)=\cos\alpha$.
Decision boundary $\Gamma=\{\pi/2,3\pi/2\}$: two points.
\item $\rho([\ell])=\sigma$: non-trivial functor,
$[\ClassFunc]\neq 0$. By Theorem~\ref{thm:main}(ii), no continuous
$f\colon\sph{1}\to\R$ realises this. This is the M\"obius classifier
of \S\ref{sec:S1-mobius}.
\end{itemize}
The labels alone do not select between the two functors --- the data
is consistent with both. Selection is by an additional structural
input (parity of sign changes, equivalently choice of bundle), the
content of Corollary~\ref{cor:obstruction} in \S\ref{ssec:obstruction}.
\end{example}

% ---------------------------------------------------------------------
\subsection{The two tiers of the framework}\label{ssec:two-tiers}

Theorem~\ref{thm:main} is the first of two independent obstruction
results in the paper. Before proceeding it is worth spelling out the
structure they form together, because a reader encountering only
Theorem~\ref{thm:main} would reasonably ask: \emph{if $H^1(M,\Z_2)=0$
implies a sign function exists, why introduce curvature, the
Yang--Mills--Higgs functional, or attention at all?} The answer is
that the two tiers are different \emph{questions} a classification
problem can ask, not two coordinates a single classifier lives along.

\paragraph{Tier 1: can a sign-function classifier do the job?} Take
structure group $G=\Z_2$ (the binary labels themselves). The
realising bundle is a flat $\Z_2$-bundle, classified by
$H^1(M,\Z_2)$ (Lemma~\ref{lem:functor-classification}). If
$H^1(M,\Z_2)=0$, there is no obstruction: a global sign-function
classifier $f\colon M\to\R$ with $\mathrm{sign}\,f(x_i)=y_i$ exists
(Theorem~\ref{thm:main}(ii)). If $H^1(M,\Z_2)\neq 0$, no sign
function works, but parallel transport in a flat non-trivial
$\Z_2$-bundle still classifies, and the variational selector of
\S\ref{sec:yang-mills-higgs} picks the class minimising matter
energy. Tier~1 is the $G=\Z_2$ obstruction theory, and the
variational selector genuinely operates on it.

\paragraph{Tier 2: fix a structure group $G$ and ask what the
optimal classifier looks like.} Independently of the tier-1 picture
(or in parallel with it), one may fix a non-abelian or non-discrete
structure group $G$ --- guided by the Adams ladder
($\R^*, U(1), SU(2),\dots$) for the dimension of the classifier's
fibre --- and ask for the optimal connection and section of a
$G$-bundle on $M$. The realising bundle is now classified by higher
characteristic classes ($c_1\in H^2(M,\Z)$ for $U(1)$,
$c_2\in H^4(M,\Z)$ for $SU(2)$). A non-zero higher characteristic
class \emph{forces} curvature: $\int_M F\neq 0$, no flat connection
exists in that sector, and the YMH problem has BPS (Bogomolny-
saturating) connections as optimal in the small-matter regime. The
choice of $G$ is external to the variational principle --- the
selector of \S\ref{sec:yang-mills-higgs} ranges over
$H^1(M,\Z_2)$, not over the Adams ladder; on tier-2 bases the bundle
is treated as a Problem~B configuration in the sense of
Remark~\ref{rem:problem-A-B}.
Tier~2 is the structure of the optimal classifier \emph{given} a
chosen non-abelian or non-discrete $G$.

\paragraph{When does each tier matter?} Tier 1 always matters: it
decides whether a sign-function suffices, on any base. Tier 2
matters when one wants a richer (non-abelian, curved) classifier
than $\Z_2$, either because the base has $H^2(M,\Z)\neq 0$ that
makes such bundles natural (the $\sph{2}$ Hopf bundle as the
fundamental example) or because the application calls for it
(multi-head attention, in the deep-learning reading of
\S\ref{sec:attention}, lives at $G=SU(2)$).

The four combinations of (tier-1 obstruction, tier-2 chosen $G$)
are summarised below. Tier~1 is about the \emph{label} group
$\Z_2$; tier~2 is about a \emph{chosen} structure group
$G\supseteq\Z_2$. The two play different roles --- the first
selects a $\Z_2$-bundle from finite data, the second prescribes
a richer fibre that carries curvature.

\medskip
\begin{center}
\renewcommand{\arraystretch}{1.5}
\begin{tabular}{|p{2.6cm}|p{5cm}|p{5cm}|}
\hline
 & \textbf{Tier 1 trivial} & \textbf{Tier 1 non-trivial}\\
 & \small{($H^1(M,\Z_2)=0$)} & \small{(no sign function on $M$)}\\
\hline
\textbf{Trivial $G$, flat bundle}
\newline\small{(curvature absent)}
&
$\R^n$, contractible $M$. Flat connection, sign function realised
by the harmonic interpolant. Paper~1 regime
(\S\ref{sec:reduction-paper1}).
&
$\sph{1}$ M\"obius case; $T^2$ XOR. $F=0$ everywhere, section is
anti-periodic on the covering space, no sign function on $M$
(\S\S\ref{sec:S1-mobius}, \ref{sec:torus}).
\\
\hline
\textbf{Chosen non-abelian or non-discrete $G$ in tier-2 sector}
\newline\small{(curvature forced by $c_1, c_2$)}
&
$\sph{2}, \sph{4}$ with $c_1=1$, $c_2=1$ as Problem~B
configurations. Sign function exists at tier~1, but the chosen
non-trivial sector forces $\int_M F\neq 0$; monopole and instanton
(\S\S\ref{sec:S2-monopole}, \ref{sec:S4-instanton}).
&
Open. Would require non-orientable base with non-trivial Chern
class twisted by $\Z_2$ holonomy (e.g.\ $\mathbb{RP}^2$, twisted
bundles on $T^2$). Not worked in this paper.
\\
\hline
\end{tabular}
\end{center}
\medskip

The lower-left cell deserves a sentence beyond the table itself: it
is the curved-geometry interior of the paper
(\S\S\ref{sec:S2-monopole}, \ref{sec:S4-instanton}), where the
structure group $G$ is chosen non-discrete and a non-trivial
Chern-class sector is fixed, so curvature is forced and the
geometrically preferred classifier lives on a BPS connection. It is
here that the identification of curvature with attention
(\S\ref{sec:attention}) becomes non-trivial.

Theorem~\ref{thm:main} answers tier~1 completely. The
\emph{cohomological} structure of tier 1 --- the identification of
$H^1(M,\Z_2)$ with principal $\Z_2$-bundles, the natural line bundle
$L_\xi$, and the meaning of the obstruction class $[c_y]$ --- is the
subject of \S\ref{ssec:obstruction} below. The variational principle
that selects $[c_y]$ when the labels alone do not constrain it is
the subject of \S\ref{sec:yang-mills-higgs}, and the worked examples
at the chosen $G$ in tier~2 are the content of Part~2. The two tiers
together are why the framework exists: tier 1 alone reduces to
Paper~1 on contractible $M$, but on non-contractible $M$ it
genuinely chooses between flat bundles; tier 2 is where curvature
and non-abelian content live, and is non-empty whenever a
non-discrete $G$ is fixed on a base with appropriate higher
cohomology. The non-trivial content of Paper~2 is both rows of the
table.

% ---------------------------------------------------------------------
\subsection{The classification obstruction $[c_y]\in H^1(M,\Z_2)$}\label{ssec:obstruction}

The realising functors of Theorem~\ref{thm:main}(i) form a torsor
over $H^1(M,\Z_2)$. To make this structural we identify
$H^1(M,\Z_2)$ with the classifying group of principal $\Z_2$-bundles
on $M$, and attach to each class $\xi\in H^1(M,\Z_2)$ the associated
real line bundle on which sign-function classifiers live.

\paragraph{Cocycles and bundles.} Fix a good open cover
$\{U_\alpha\}_{\alpha\in A}$ of $M$. A \emph{$\Z_2$-cocycle} on this
cover is a family of locally constant functions
$g_{\alpha\beta}\colon U_\alpha\cap U_\beta\to\Z_2$, defined on each
non-empty pairwise intersection, satisfying $g_{\alpha\alpha}=
\mathrm{id}$ and the cocycle condition
\[
   g_{\alpha\beta}\cdot g_{\beta\gamma} \;=\; g_{\alpha\gamma}
   \qquad\text{on } U_\alpha\cap U_\beta\cap U_\gamma.
\]
Because $\Z_2$ is abelian and every element is its own inverse,
$g_{\beta\alpha}=g_{\alpha\beta}^{-1}=g_{\alpha\beta}$. Two cocycles
$(g_{\alpha\beta}),(g'_{\alpha\beta})$ are \emph{equivalent} if
locally constant $h_\alpha\colon U_\alpha\to\Z_2$ exist with
$g'_{\alpha\beta}=h_\alpha\cdot g_{\alpha\beta}\cdot h_\beta^{-1}$.
The set of equivalence classes is $H^1(M,\Z_2)$, the first
\v{C}ech cohomology with constant $\Z_2$ coefficients.

\begin{lemma}[Cocycles classify $\Z_2$-bundles]\label{lem:cocycle-bundle}
Each class $\xi=[(g_{\alpha\beta})]\in H^1(M,\Z_2)$ determines a
principal $\Z_2$-bundle $(P_\xi,p,M,\Z_2)$, uniquely up to
isomorphism, with transition functions the $g_{\alpha\beta}$.
Conversely, every principal $\Z_2$-bundle over $M$ arises this way.
\end{lemma}

\begin{proof}
Standard; see \cite[Chapter~4]{Husemoller} or
\cite[\S 4]{Hatcher}.
$P_\xi=\bigsqcup_\alpha(U_\alpha\times\Z_2)/\!\sim$ with
$(x,h)\in U_\alpha\times\Z_2$ identified with
$(x,g_{\alpha\beta}(x)\cdot h)\in U_\beta\times\Z_2$ for
$x\in U_\alpha\cap U_\beta$; the cocycle condition makes this
transitive, and equivalent cocycles give isomorphic bundles.
\end{proof}

A principal $\Z_2$-bundle is the same data as a double cover of $M$
(possibly disconnected: the trivial class gives the disconnected
cover $M\sqcup M$; the non-trivial class on $\sph{1}$ gives the
connected degree-$2$ cover).

\paragraph{The associated real line bundle.} The geometric carrier
for a sign-function classifier is the rank-1 real bundle associated
to $\xi$ via the sign representation $\Z_2\to O(1)=\{\pm 1\}$,
$\sigma\mapsto -1$.

\begin{definition}[Associated real line bundle]\label{def:L-xi}
Let $\sigma\in\Z_2$ act on $\R$ by $v\mapsto -v$. The
\emph{associated real line bundle} of $\xi\in H^1(M,\Z_2)$ is
\[
   L_\xi \;:=\; P_\xi \times_{\Z_2} \R,
\]
a real line bundle over $M$ with first Stiefel--Whitney class
$w_1(L_\xi)=\xi$. The trivial class gives $L_0=M\times\R$; on
$\sph{1}$ the non-trivial class gives the M\"obius line bundle.
\end{definition}

\begin{remark}[Sections of $L_\xi$ and anti-equivariant functions]\label{rem:sections-equivariant}
A section $\phi\colon M\to L_\xi$ is the same data as a continuous
function $\tilde\phi\colon P_\xi\to\R$ satisfying the
\emph{anti-equivariance} condition $\tilde\phi(p\cdot\sigma)=-
\tilde\phi(p)$ for every $p\in P_\xi$. When $P_\xi\to M$ is the
connected double cover, $\tilde\phi$ is an anti-periodic real
function on the cover --- the content of ``anti-periodic section''
for the M\"obius classifier in \S\ref{ssec:two-tiers}.
\end{remark}

\begin{remark}[Where the rank-2 picture enters]\label{rem:rank-2-deferred}
The rank-2 real bundle $E_\xi=P_\xi\times_{\Z_2}\R^2$ with
$\sigma\mapsto -I$ is reducible: $\R^2=\R\oplus\R$ with $\sigma$
acting as $-1$ on each summand, hence
$E_\xi\cong L_\xi\oplus L_\xi$. In \S\ref{sec:yang-mills-higgs} the
Yang--Mills--Higgs connection is most naturally written as a
$U(1)$-connection on the complexified bundle
$L_\xi\otimes_\R\C$. The two pieces of information from
\S\ref{ssec:two-tiers}'s two-tier picture are then carried as
follows. Complex line bundles over $M$ are classified topologically
by $c_1\in H^2(M,\Z)$, and the complexification of $L_\xi$ has
\[
   c_1(L_\xi\otimes_\R\C) \;=\; \beta(\xi) \;\in\; H^2(M,\Z),
\]
the Bockstein image of $\xi$ under the connecting homomorphism
$\beta\colon H^1(M,\Z_2)\to H^2(M,\Z)$ associated to the short exact
sequence $0\to\Z\xrightarrow{2}\Z\to\Z_2\to 0$. The image of $\beta$
lies in the $2$-torsion of $H^2(M,\Z)$ (since
$c_1(\bar L)=-c_1(L)$ for a complexified real bundle, so
$2 c_1(L_\xi\otimes_\R\C)=0$). On $\sph{1}$, $H^2(\sph{1},\Z)=0$ so
$\beta=0$ trivially and $L_\xi\otimes_\R\C$ is the trivial complex
line bundle; on $\R\mathrm{P}^2$, the Bockstein
$H^1(\R\mathrm{P}^2,\Z_2)\to H^2(\R\mathrm{P}^2,\Z)=\Z_2$ is the
isomorphism, giving an example where the complexification is
topologically non-trivial. In either case the $\xi$ information
survives in the bundle's connection as the \emph{holonomy} of a
flat $U(1)$-connection, even when (as on $\sph{1}$) the underlying
bundle is topologically trivial.

Tier 2 then records bundle topology \emph{not coming from} tier 1:
the part of $c_1\in H^2(M,\Z)$ outside the Bockstein image
$\beta(H^1(M,\Z_2))$ --- equivalently, $c_1$ modulo $2$-torsion of
the Bockstein form --- is what forces curvature on a bundle
($2$-torsion $c_1$ in the image of $\beta$ admits a flat connection
inherited from the underlying $\Z_2$-bundle via the inclusion
$\Z_2\hookrightarrow U(1)$). For the worked examples of Part~2
this distinction is invisible: $H^2(\sph{2},\Z)=\Z$ and
$H^4(\sph{4},\Z)=\Z$ are torsion-free, so $\beta=0$ and every
non-zero $c_1$ (or $c_2$) is non-flat. The framework thus uses
both pieces in their respective natural homes: $\xi$ as flat
holonomy (tier 1), and $c_1$ modulo $\beta$ as forced curvature
(tier 2). At the cohomological level of \S\ref{ssec:obstruction}
the real line bundle $L_\xi$ is the natural object: its sections
are sign-function classifiers, and $w_1(L_\xi)=\xi$ records the
obstruction class directly.
\end{remark}

\paragraph{The obstruction class.} A class
$\xi\in H^1(M,\Z_2)$ now has four equivalent descriptions:
(i) a cohomology class of $\Z_2$-cocycles $[(g_{\alpha\beta})]$;
(ii) a principal $\Z_2$-bundle $P_\xi\to M$ (a double cover);
(iii) an associated real line bundle $L_\xi\to M$ with
$w_1(L_\xi)=\xi$;
(iv) a natural-isomorphism class of realising functors
(Theorem~\ref{thm:main}(i)).

\begin{definition}[Classification obstruction]\label{def:cy}
Let $D=\{(x_i,y_i)\}$ be a labelled dataset on $M$ with distinct
points $x_1,\dots,x_N$. A \emph{classification obstruction} for $D$
is a choice
\[
   [c_y] \;\in\; H^1(M,\Z_2)
\]
of natural-isomorphism class of realising functor of $D$.
\end{definition}

The condition ``$\phi(x_i)$ has sign $y_i$ in $L_\xi$'' is taken with
respect to the local frame obtained by parallel-transporting a chosen
frame at $x_1$ along $\gamma_i$ in the flat bundle $P_\xi$. This is
the same trivialisation convention underlying \S\ref{ssec:data-as-functor}:
a different frame at $x_1$ globally flips all signs (a gauge
transformation), and a different choice of path
$\gamma_i'\colon x_1\to x_i$ flips the sign at $x_i$ by the monodromy
of $\gamma_i'\cdot\gamma_i^{-1}$ (Remark~\ref{rem:path-dependence}).

For any labelled dataset with distinct points the labels impose no
constraint at the cohomological level: every class
$\xi\in H^1(M,\Z_2)$ supports a section of $L_\xi$ matching the
prescribed signs, by partition of unity in the chosen local
trivialisations. Selection of a specific $[c_y]$ is the job of the
analytic machinery in \S\ref{sec:yang-mills-higgs}; here we record
what the choice means.

\begin{corollary}[Cohomological obstruction]\label{cor:obstruction}
Let $D$ be a labelled dataset on $M$ with distinct points, and let
$[c_y]\in H^1(M,\Z_2)$ be a classification obstruction. Then:
\begin{enumerate}[label=(\roman*),leftmargin=2em]
  \item $[c_y]=0$ if and only if a continuous sign function
  $f\colon M\to\R$ with $\mathrm{sign}\,f(x_i)=y_i$ realises the
  classification (in the sense of
  Definition~\ref{def:induced-functor}).
  \item $[c_y]\neq 0$ if and only if no continuous sign function on
  $M$ realises this classification. Prediction proceeds by parallel
  transport in the flat principal $\Z_2$-bundle $P_{[c_y]}$,
  equivalently by lifting paths to the double cover and reading the
  resulting fibre element.
\end{enumerate}
\end{corollary}

\begin{proof}
$[c_y]\in H^1(M,\Z_2)$ is the natural-iso class of a realising
functor of $D$. The equivalence ``$[c_y]=0$ iff a sign function
realises the classification'' is Theorem~\ref{thm:main}(ii) applied
to that functor. Part (ii) is the contrapositive, with the
parallel-transport interpretation given in
Remark~\ref{rem:emergent-boundary} and its footnote.
\end{proof}

The non-trivial cohomological content of \S\ref{ssec:obstruction}
is Lemma~\ref{lem:cocycle-bundle} (cocycles classify bundles) and
Definition~\ref{def:L-xi} ($L_\xi$ as the natural classifier-carrier);
Corollary~\ref{cor:obstruction} packages the consequence for
sign-function classifiers. The analytical question of \emph{which}
class $[c_y]\in H^1(M,\Z_2)$ realises the geometrically preferred
classifier is deferred to \S\ref{sec:yang-mills-higgs}.

\begin{remark}[Worked classes for canonical $M$]\label{rem:worked-classes}
For the bases encountered in this paper:
\begin{itemize}[leftmargin=2em]
\item $M=\R^n$: $H^1(\R^n,\Z_2)=0$, so $[c_y]=0$ for every dataset
and Paper~1 applies (\S\ref{sec:reduction-paper1}).
\item $M=\sph{n}$, $n\geq 2$: $H^1(\sph{n},\Z_2)=0$, so $[c_y]=0$
and a sign function exists for any dataset. The geometric content
for $\sph{n}$ is at tier~2 (\S\ref{ssec:two-tiers}).
\item $M=\sph{1}$: $H^1(\sph{1},\Z_2)=\Z_2$, so
$[c_y]\in\{0,\sigma\}$. Selection is variational; the M\"obius
classifier for two opposite labels is computed in
\S\ref{sec:S1-mobius}.
\item $M=T^2$: $H^1(T^2,\Z_2)=\Z_2^2$, so $[c_y]\in\Z_2^2$.
Selection is variational; the toroidal XOR classifier is computed
in \S\ref{sec:torus}.
\end{itemize}
\end{remark}

% ---------------------------------------------------------------------
\section{The Wilson observable as universal readout}
\label{sec:wilson}

\S\ref{sec:functor} attached to every labelled dataset a
\emph{classification obstruction} $[c_y]\in H^1(M,\Z_2)$, equivalently
a flat $\Z_2$-bundle whose holonomy realises the labels. Once a
non-trivial bundle is in play --- and on non-contractible bases this
is the rule, not the exception --- the prediction at a non-data
point $x\in M$ cannot be read off a sign function on $M$. It must be
computed by parallel transport. The \emph{Wilson observable}
formalises this prediction mechanism for an arbitrary structure
group and an arbitrary (possibly non-flat) connection. It is the
universal readout of the framework: a gauge-invariant scalar
attached to a loop, defined for any compact $G$, whose specialisation
to flat $\Z_2$-bundles recovers the monodromy of \S\ref{ssec:obstruction}.

This section is organised as follows. \S\ref{ssec:holonomy}
introduces the holonomy of a connection along a path, with its
basepoint-frame dependence and the resulting conjugation ambiguity
on loops. \S\ref{ssec:wilson-def} defines the Wilson observable
$W_\gamma(A)=\chi(\mathrm{Hol}_\gamma(A))$ for two natural types of
readout $\chi$ on $G$ --- multiplicative homomorphisms and trace
class functions --- and computes them for the structure groups
encountered in this paper.
\S\ref{ssec:open-paths} extends the construction to open paths
$x_0\to x$, where the Wilson observable becomes a (gauge-covariant)
prediction mechanism and acquires genuine path-dependence when the
connection is curved. \S\ref{ssec:wilson-reduction} verifies that
the Wilson framework specialises correctly to the
\S\ref{ssec:obstruction} picture when $G=\Z_2$ and $A$ is flat.

% ---------------------------------------------------------------------
\subsection{Holonomy of a connection}\label{ssec:holonomy}

Let $P\to M$ be a principal $G$-bundle with $G$ a compact Lie group.
A \emph{connection} on $P$ is, invariantly, a $G$-equivariant
horizontal distribution on the total space $P$, or equivalently a
$G$-equivariant $\mathfrak{g}$-valued $1$-form $\omega\in\Omega^1(P,\mathfrak{g})$
on $P$ \cite[Chapter~II]{KobayashiNomizu}. Pulling back $\omega$ along
local sections $s_\alpha\colon U_\alpha\to P$ gives a collection of
$\mathfrak{g}$-valued $1$-forms
$\{A_\alpha:=s_\alpha^*\omega\in\Omega^1(U_\alpha,\mathfrak{g})\}$ on
local trivialisations $\{U_\alpha\}$ of $P$, related on overlaps by
the gauge transformation rule $A_\beta=g_{\alpha\beta}^{-1}A_\alpha
g_{\alpha\beta}+g_{\alpha\beta}^{-1}dg_{\alpha\beta}$. A single
global $1$-form $A\in\Omega^1(M,\mathfrak{g})$ exists if and only if
$P$ admits a global section $s\colon M\to P$, which for a principal
bundle is equivalent to triviality of $P$; for non-trivial $P$ one
must work in trivialisations or use the invariant total-space
formulation. In what
follows we work with the local-trivialisation formulation and write
$A$ for the relevant patch's $A_\alpha$, with the gauge transformation
rule kept implicit.

The connection defines parallel transport along paths: given a
smooth path $\gamma\colon[0,1]\to M$ from $x=\gamma(0)$ to
$y=\gamma(1)$, and a choice of fibre point $p_x\in P_x$ over
$x$, parallel transport produces a fibre point $p_y\in P_y$ over
$y$ by solving the horizontality ODE
\[
   \frac{d}{dt}\tilde\gamma(t) + A(\gamma'(t))\cdot\tilde\gamma(t) = 0,
   \qquad \tilde\gamma(0)=p_x,
\]
along the lift $\tilde\gamma$ of $\gamma$ to $P$. The endpoint is the
parallel-transported fibre point $p_y=\tilde\gamma(1)$.

\paragraph{Frames of associated vector bundles.} When the
classification problem also fixes an associated vector bundle
$E=P\times_G V$ for some $G$-representation $V$ (the line bundle
$L_\xi$ of \S\ref{ssec:obstruction}, or the $\C^2$-doublet of
\S\ref{sec:S4-instanton}), a fibre point $p_x\in P_x$ determines a
\emph{frame} of $E_x$ in the usual sense (an isomorphism
$V\xrightarrow{\sim} E_x$), and parallel transport on $P$ induces
parallel transport of frames in $E$.

\paragraph{The holonomy element.} For a loop $\gamma\colon[0,1]\to M$
based at $x$ (so $\gamma(0)=\gamma(1)=x$), parallel transport sends
the chosen fibre point $p_x$ to a new fibre point $p_x'\in P_x$ over
the same point. Since the fibre $P_x$ is a principal $G$-torsor,
there is a unique $h\in G$ with $p_x'=p_x\cdot h$. This element is the
\emph{holonomy} of $A$ along $\gamma$:
\[
   \mathrm{Hol}_\gamma(A) \;:=\; h \;\in\; G.
\]
Concretely, on a trivialised patch around $x$, the holonomy is given
by the path-ordered exponential
\[
   \mathrm{Hol}_\gamma(A)
   \;=\;
   \mathcal{P}\exp\!\left(-\oint_\gamma A\right) \;\in\; G,
\]
the limit of products of small-step transports
$\exp(-A(\gamma'(t_k))\Delta t)$ along a discretisation of $\gamma$.
For abelian $G$ this reduces to the ordinary exponential
$\exp(-\oint_\gamma A)$.

\paragraph{Basepoint dependence.} The holonomy element depends
on the chosen basepoint fibre point $p_x\in P_x$ (or equivalently,
when an associated vector bundle is fixed, on the chosen basepoint
frame). A different choice
$p_x'=p_x\cdot g$ changes the holonomy by conjugation:
\[
   \mathrm{Hol}_\gamma(A) \;\mapsto\; g^{-1}\cdot
   \mathrm{Hol}_\gamma(A)\cdot g.
\]
The holonomy is therefore well-defined only as a conjugacy class in
$G$, equivalently as an element of $G/\!\sim_{\mathrm{conj}}$. This
basepoint ambiguity is the gauge freedom in the holonomy.

\paragraph{Holonomy as a homomorphism on loops.} For two loops
$\gamma_1,\gamma_2$ based at the same point $x$, computed with the
same basepoint frame $p_x$, with $\gamma_2\cdot\gamma_1$ the
concatenation (first $\gamma_1$, then $\gamma_2$), the
multiplicativity
\[
   \mathrm{Hol}_{\gamma_2\cdot\gamma_1}(A)
   \;=\; \mathrm{Hol}_{\gamma_2}(A)\cdot \mathrm{Hol}_{\gamma_1}(A)
\]
holds exactly in $G$. The conjugacy ambiguity discussed above lives
on the individual factors, not on the composition: once $p_x$ is
fixed, both sides are well-defined elements of $G$ and equal. For a
flat
connection ($F_A=0$ on $M$), the holonomy depends only on the
homotopy class of $\gamma$, giving a homomorphism
$\pi_1(M,x)\to G$, well-defined up to conjugation. This recovers the
monodromy picture of \S\ref{ssec:obstruction} when $G=\Z_2$.

\paragraph{The cover-adapted formula.} The path-ordered exponential
is the smooth-limit description of holonomy. A cover-adapted
description, exact rather than limiting, expresses the holonomy as a
finite product of local exponentials interleaved with transition
functions. This is the description we use when discussing holonomy
accumulation in \S\ref{sec:yang-mills-higgs}, because each factor in
the product has a canonical interpretation: a smooth contribution
from a single patch, or a topological contribution from a patch
transition.

Let $\{U_\alpha\}_{\alpha\in A}$ be a good open cover of $M$
trivialising the bundle, with transition functions $g_{\alpha\beta}$
(locally constant in the flat case, smooth otherwise). A path
$\gamma\colon[0,1]\to M$ from $x$ to $y$ is \emph{cover-adapted} by a
choice of partition $0=t_0<t_1<\cdots<t_n=1$ and patches
$\alpha_0,\alpha_1,\dots,\alpha_{n-1}\in A$ such that
$\gamma([t_k,t_{k+1}])\subset U_{\alpha_k}$ for each $k$. On each
segment, the local connection 1-form $A_{\alpha_k}\in
\Omega^1(U_{\alpha_k},\mathfrak{g})$ gives a segment holonomy
\[
   h_k \;:=\; \mathcal{P}\exp\!\left(-\int_{t_k}^{t_{k+1}}
   A_{\alpha_k}(\gamma'(t))\,dt\right) \;\in\; G.
\]
At each interior breakpoint $t_{k+1}$, the path is in
$U_{\alpha_k}\cap U_{\alpha_{k+1}}$, and the change of trivialisation
contributes the transition function
$g_{\alpha_{k+1}\alpha_k}(\gamma(t_{k+1}))\in G$.

\begin{proposition}[Cover-adapted holonomy formula]\label{prop:cech-holonomy}
With notation as above, the holonomy of $A$ along $\gamma$, computed
in the trivialisation $U_{\alpha_0}$ at the basepoint, is the ordered
product
\[
   \mathrm{Hol}_\gamma(A) \;=\;
   h_{n-1}\cdot g_{\alpha_{n-1}\alpha_{n-2}}(\gamma(t_{n-1}))
   \cdot h_{n-2}\cdots g_{\alpha_1\alpha_0}(\gamma(t_1))\cdot h_0.
\]
Moreover:
\begin{enumerate}[label=(\roman*),leftmargin=2em]
\item \emph{(Partition-independence.)} The product is independent of
the choice of partition $\{t_k\}$ and of the patches
$\{\alpha_k\}$ adapted to a fixed cover $\{U_\alpha\}$.
\item \emph{(Cover-independence.)} The product is independent of the
choice of trivialising cover: a different good cover gives the same
holonomy.
\end{enumerate}
The smooth path-ordered exponential of the preceding paragraph is the
special case in which the partition is fine enough that all
transitions $g_{\alpha_{k+1}\alpha_k}$ collapse into the limiting
ODE.
\end{proposition}

\begin{proofsketch}
For (i), two adapted partitions admit a common refinement; refining
within a single patch $U_{\alpha_k}$ is absorbed by multiplicativity
of the path-ordered exponential
($\mathcal{P}\exp\int_{t_k}^{t_{k+1}}=\mathcal{P}\exp\int_s^{t_{k+1}}
\cdot\mathcal{P}\exp\int_{t_k}^s$ for $s\in[t_k,t_{k+1}]$), and
switching the patch label on a segment within $U_{\alpha_k}\cap
U_{\alpha_k'}$ uses the gauge-transformation rule
$A_{\alpha_k'}=g_{\alpha_k'\alpha_k}A_{\alpha_k}g_{\alpha_k'\alpha_k}^{-1}
+g_{\alpha_k'\alpha_k}dg_{\alpha_k'\alpha_k}^{-1}$ together with the
cocycle condition on transitions; the contributions cancel. For
(ii), two covers admit a common refinement (the intersection cover);
by (i), the holonomy is partition-independent in each original cover
and hence equal to the value in the common refinement. Detailed proof
in \cite{KobayashiNomizu}.
\end{proofsketch}

\begin{remark}[Flat case]\label{rem:flat-cech}
For a flat connection the local representatives can be chosen so that
$A_\alpha=0$ on every patch (gauge to constant); each segment
exponential is then $h_k=1$, and the holonomy is the pure product of
transition functions encountered along the path,
\[
   \mathrm{Hol}_\gamma(A) \;=\;
   g_{\alpha_{n-1}\alpha_{n-2}}\cdot g_{\alpha_{n-2}\alpha_{n-3}}
   \cdots g_{\alpha_1\alpha_0}.
\]
For $G=\Z_2$ this is the parity of sign-flip transitions along
$\gamma$ --- the explicit Čech computation of the monodromy of
\S\ref{ssec:obstruction}.
\end{remark}

% ---------------------------------------------------------------------
\subsection{The Wilson observable}\label{ssec:wilson-def}

The conjugation ambiguity of $\mathrm{Hol}_\gamma(A)$ on loops is
killed by passing to a conjugation-invariant scalar. Two natural
choices arise, with genuinely different properties.

\begin{definition}[Two types of readout]\label{def:wilson}
Let $G$ be a compact Lie group.
\begin{enumerate}[label=(\arabic*),leftmargin=2em]
\item A \emph{multiplicative readout} is a continuous group
homomorphism $\chi\colon G\to A$ to a target group $A$ (typically
$A=\Z_2$, $A=\R^*$, or $A=U(1)\subset\C^*$). It satisfies
$\chi(gh)=\chi(g)\chi(h)$ and $\chi(g^{-1}xg)=\chi(x)$
(conjugation-invariance follows from $\chi(g^{-1})\chi(x)\chi(g)=
\chi(x)\chi(g^{-1})\chi(g)=\chi(x)$ once $\chi(g^{-1})\chi(g)=\chi(e)
=1$).
\item A \emph{trace readout} is a class function
$\chi(g)=\tr_\rho(g)$ given by the trace of a finite-dimensional
unitary representation $\rho\colon G\to U(V)$. It satisfies
$\chi(g^{-1}xg)=\chi(x)$ but is \emph{not} multiplicative:
$\chi(gh)\neq\chi(g)\chi(h)$ in general.
\end{enumerate}
Given a readout $\chi$ of either type, the \emph{Wilson observable}
of a connection $A$ along a loop $\gamma$ at $x\in M$ is
\[
   W_\gamma^\chi(A) \;:=\; \chi(\mathrm{Hol}_\gamma(A)).
\]
\end{definition}

\begin{proposition}[Gauge invariance on closed loops]\label{prop:wilson-gauge-invariance}
For either type of readout, the Wilson observable $W_\gamma^\chi(A)$
of a closed loop is independent of basepoint frame and invariant
under bundle gauge transformations.
\end{proposition}

This is a standard fact; the proof is a one-line verification using
conjugation-invariance of $\chi$, since a frame change and a bundle
gauge transformation both act on the holonomy by conjugation.
Detailed expositions are in \cite{KobayashiNomizu} and any text on
lattice or continuum gauge theory.

\paragraph{Examples.} The structure groups appearing in this paper
each have a canonical readout.

\begin{itemize}[leftmargin=2em]
\item \textbf{$G=\Z_2$.} The identity $\chi=\mathrm{id}\colon\Z_2
\to\Z_2$ is a multiplicative readout. $W_\gamma(A)\in\{\pm 1\}$
takes binary values directly. For the flat connection on the bundle
$P_{[c_y]}$ of \S\ref{ssec:obstruction}, $W_\gamma(A)$ is the
monodromy $\rho_{[c_y]}([\gamma])$.

\item \textbf{$G=O(2)$.} The orientation character $\det\colon O(2)
\to\{\pm 1\}$ is a multiplicative readout (rotations have
$\det=+1$, reflections $\det=-1$). $W_\gamma^{\det}(A)\in\{\pm 1\}$
is binary, recording whether the holonomy of $\gamma$ preserves or
reverses fibre orientation. This is the canonical binary readout for
$O(2)$-connections.

\item \textbf{$G=U(1)$.} The irreducible characters of $U(1)$ are
the multiplicative readouts $\chi_n\colon U(1)\to U(1)$,
$\chi_n(e^{i\theta})=e^{in\theta}$, for $n\in\Z$. They take values in
$U(1)\subset\C^*$, not in $\{\pm 1\}$: since $U(1)$ is connected,
there is no non-trivial continuous group homomorphism
$U(1)\to\Z_2$, so no multiplicative readout gives a binary Wilson
observable. The real part $\cos(n\theta)=\Re\chi_n(e^{i\theta})$ is a
trace readout (the character of the real form), continuous-valued
in $[-1,1]$; it is conjugation-invariant on the abelian $U(1)$ but
not multiplicative.

\item \textbf{$G=SU(2)$.} $SU(2)$ is perfect, so every continuous
group homomorphism $SU(2)\to A$ to an abelian group $A$ is trivial:
no non-trivial multiplicative readout exists. The trace
$\chi(g)=\tfrac12\tr_\rho(g)$ of the fundamental representation
$\rho\colon SU(2)\to U(2)$ is a trace readout, continuous-valued in
$[-1,1]$. This is the Wilson observable used in the instanton
calculation of \S\ref{sec:S4-instanton}.
\end{itemize}

The pattern: multiplicative readouts to $\Z_2$ exist precisely when
$G$ has a non-trivial continuous homomorphism to $\Z_2$, equivalently
when $G/[G,G]$ surjects onto $\Z_2$. Structure groups with this
property ($\Z_2$, $O(2)$, $O(n)$, $\Z_2^k$) admit binary Wilson
observables directly. Structure groups without it (connected groups
like $U(1)$, $SU(n)$ for $n\geq 2$; perfect groups like $SU(2)$,
$SU(3)$) require trace readouts, which are continuous-valued and
need an additional sign-extraction step for binary classification.

% ---------------------------------------------------------------------
\subsection{Wilson on open paths and the prediction mechanism}\label{ssec:open-paths}

For prediction at non-data points the Wilson observable on closed
loops is not directly enough --- prediction requires assigning a
value to each $x\in M$. The natural extension uses open paths from a
chosen base point, but the gauge-invariance picture changes
non-trivially.

\paragraph{Open holonomy is gauge-covariant, not gauge-invariant.}
Fix a base point $x_0\in M$. For any $x\in M$ and any path
$\gamma\colon x_0\to x$, the parallel transport of a chosen frame
$p_{x_0}\in P_{x_0}$ along $\gamma$ produces a frame at $x$, which
under a chosen frame $p_x\in P_x$ is encoded by an element
$\mathrm{Hol}_\gamma(A)\in G$. This element depends on the chosen
frames at \emph{both} endpoints: under independent frame changes
$p_{x_0}\mapsto p_{x_0}\cdot g_0$ and $p_x\mapsto p_x\cdot g_x$, the
open holonomy transforms by two-sided multiplication,
\[
   \mathrm{Hol}_\gamma(A) \;\mapsto\;
   g_x^{-1}\cdot\mathrm{Hol}_\gamma(A)\cdot g_0,
\]
which is \emph{not} conjugation. Class functions, which kill
conjugation, do not kill this two-sided action. The open Wilson
observable $W_\gamma^\chi(A)=\chi(\mathrm{Hol}_\gamma(A))$ is
therefore not gauge-invariant on its own. It is \emph{gauge-covariant}:
well-defined relative to a choice of frame at each endpoint.

In physics this is the standard distinction between Wilson loops
(closed, gauge-invariant) and Wilson lines (open, gauge-covariant);
open Wilson lines are physical only when paired with matter fields
at the endpoints, or when a gauge convention is fixed across $M$.

\paragraph{Frame conventions and prediction.} The framework already
fixes a frame at the base point $x_0=x_1$ implicitly: \S\ref{ssec:data-as-functor}
introduces reference paths $\gamma_i\colon x_1\to x_i$, and
\S\ref{ssec:obstruction} (after Definition~\ref{def:cy}) describes
the trivialisation at each data point $x_i$ obtained by
parallel-transporting a chosen frame at $x_1$ along $\gamma_i$. The
same convention extends to any $x\in M$: a chosen frame at $x_1$,
together with a chosen path $\gamma\colon x_1\to x$, fixes a frame
at $x$.

With this convention, for a multiplicative readout $\chi\colon G\to A$
to a target group $A\ni\{\pm 1\}$ (or $A=\Z_2$ directly), the
predicted label at $x$ along $\gamma$ is
\[
   y_x \;:=\; y_1\cdot\chi(\mathrm{Hol}_\gamma(A)) \;\in\; A,
\]
well-defined modulo the overall sign coming from the choice of
frame at $x_1$ (relabelling everyone $+\leftrightarrow -$ flips
every prediction simultaneously). This is the same global sign
ambiguity already present in the §2.6 framework, and is the same
ambiguity present in the labels themselves.

For a trace readout $\chi(g)=\tr_\rho(g)$, the open Wilson
observable $\chi(\mathrm{Hol}_\gamma(A))$ depends on the chosen
frame at $x$ in a more genuine way: it transforms as
$\chi(g_x^{-1}h g_0)$, and even fixing $g_0$, varying $g_x$ produces
a one-parameter family of real values. Prediction by sign of this
quantity requires a gauge fixing across $M$, which is not
canonical at the cohomological level and must come from additional
structure (typically the Riemannian metric and the Yang--Mills--Higgs
variational principle of \S\ref{sec:yang-mills-higgs}).

\paragraph{Path dependence for multiplicative readouts.} For a
multiplicative readout $\chi$, the Wilson observable on open paths
satisfies the composition law inherited from holonomy: two paths
$\gamma,\gamma'\colon x_1\to x$ give
\[
   W_{\gamma'}^\chi(A) \;=\;
   \chi(\mathrm{Hol}_{\gamma'\cdot\gamma^{-1}}(A))\cdot W_\gamma^\chi(A),
\]
where $\gamma'\cdot\gamma^{-1}$ is a loop at $x_1$ and the multiplicative
property $\chi(h_1 h_2)=\chi(h_1)\chi(h_2)$ produces the factorisation.
The first factor on the right is the Wilson observable of a closed
loop and so is genuinely gauge-invariant. For trace readouts this
factorisation \emph{does not hold} --- $\tr_\rho$ is not multiplicative
on $G$ --- and path-dependence is more subtle, taking the form
$W_{\gamma'}^\chi(A)=\chi(\mathrm{Hol}_{\gamma'\cdot\gamma^{-1}}(A)\cdot
\mathrm{Hol}_\gamma(A))$ where the product is taken inside $G$ before
applying $\chi$.

\paragraph{Flat versus curved.} When $A$ is \emph{flat}, the loop
holonomy $\mathrm{Hol}_{\gamma'\cdot\gamma^{-1}}(A)$ depends only on
the homotopy class of the loop, and $W_\gamma^\chi(A)$ depends only
on the homotopy class of $\gamma$ (rel endpoints). For
multiplicative readouts $\chi$ valued in $\Z_2$, the prediction at
$x$ is a $\Z_2$-valued cocycle on $M$ --- exactly the picture of
\S\ref{ssec:obstruction}.

When $A$ is \emph{curved}, the loop holonomy can depend on the loop
itself, not just its homotopy class, and the Wilson value
$W_\gamma^\chi(A)$ varies continuously with $\gamma$. Predictions
along homotopic but distinct paths can differ. This is the new
phenomenon introduced by tier-2 curvature
(\S\ref{ssec:two-tiers}): the classifier is not a function of $x$
alone but a function of $(x,\text{path from }x_1)$. The geometrically
preferred path family is determined by the Riemannian metric (giving
minimising-length geodesics) and the Yang--Mills--Higgs variational
principle (\S\ref{sec:yang-mills-higgs}).

\paragraph{Sign extraction from data.} For trace readouts, the sign
extraction $\mathrm{sign}\colon W_\gamma^\chi(A)\to\{\pm 1\}$
combined with a gauge fixing across $M$ is not an external input
but is fixed by the data and the Yang--Mills--Higgs variational
principle of \S\ref{sec:yang-mills-higgs}. The optimal connection
$A^*$ and the corresponding gauge fixing are the ones minimising
the YMH energy subject to the data conditions
$\mathrm{sign}(W_{\gamma_i}^\chi(A^*))=y_i\cdot y_1^{-1}$ at each
data point; the sign of $W_\gamma^\chi(A^*)$ at any other $x\in M$
along the geodesic-determined path is then the prediction. The
labels, the metric, and the variational principle together fix
both the connection and the gauge.

\paragraph{Labelled-to-labelled transport: operational test of
bundle selection.} The Wilson observable supports a clean
operational test of the framework's bundle selection that does not
appeal to any architectural analogy. Take any two labelled points
$x_i, x_j$ with labels $y_i, y_j\in\{\pm 1\}$. Let $\gamma_{ij}$ be
a path from $x_i$ to $x_j$ (the geodesic determined by the
Riemannian metric, by default).

\emph{Tier 1 ($G=\Z_2$, flat bundle).} The holonomy lives in
$\Z_2\cong\{\pm 1\}$, and the framework predicts:
\[
   \boxed{\;\mathrm{Hol}_{\gamma_{ij}}(\Connection^*)
   \;=\; y_i\cdot y_j\;\in\;\Z_2,}
\]
i.e.\ the parallel transport reverses the section's sign between
opposite labels and preserves it between same labels. This is the
content of the data conditions
$\mathrm{sign}(W_{\gamma_i}^\chi(A^*))=y_i\cdot y_1^{-1}$ written
between any two data points rather than relative to a base point.

\emph{Tier 2 ($G$ non-discrete in a chosen sector).} The holonomy
$\mathrm{Hol}_{\gamma_{ij}}(\Connection^*)\in G$ is a generic group
element, not $\pm 1$. The framework's prediction is then the Wilson
readout (Definition~\ref{def:wilson} below), not the holonomy
directly:
\[
   \mathrm{sign}\bigl(W_{\gamma_{ij}}^\chi(\Connection^*)\bigr)
   \;=\; y_i\cdot y_j,
\]
where $\chi\colon G\to\R$ is the character or cone-field readout
(\S\ref{ssec:open-paths}). The holonomy itself encodes the geometric
data (the bundle's specific connection), and the character extracts
the $\pm 1$ classification from it.

Both versions reduce to the same operational test: given the
framework's predicted optimal connection (the Dirac monopole on
$\sph{2}$, the BPST instanton on $\sph{4}$, the M\"obius/double-M\"obius
connection on $\sph{1}/T^2$), compute the holonomy along the geodesic
between each labelled pair and check the prediction for the
relevant tier. The test
fails if and only if either the bundle is wrong (the variational
selector picked the wrong class, or the chosen tier-2 sector is
incorrect) or the connection is wrong (the BPS approximation fails).
It is a clean falsification surface for the framework's predictions
on any concrete dataset.

\begin{remark}[Triangle consistency = functoriality of $\mathcal{F}$]
\label{rem:triangle-consistency}
The labelled-to-labelled transport test extends from pairs to
triples. Given three labelled points $x_i,x_j,x_k$ and a triangle of
paths $\gamma_{ij},\gamma_{jk},\gamma_{ki}$ closing into a loop at
$x_i$, the composition law of holonomy gives
\[
   \mathrm{Hol}_{\gamma_{ki}}(\Connection^*)\cdot
   \mathrm{Hol}_{\gamma_{jk}}(\Connection^*)\cdot
   \mathrm{Hol}_{\gamma_{ij}}(\Connection^*)
   \;=\; \mathrm{Hol}_{\gamma_{ki}\cdot\gamma_{jk}\cdot\gamma_{ij}}(\Connection^*).
\]
On a flat tier-1 bundle, the right-hand side depends only on the
homotopy class of the composed loop. Substituting the tier-1 boxed
prediction
$\mathrm{Hol}_{\gamma_{ab}}(\Connection^*) = y_a y_b$ on the
left-hand side gives
$y_k y_i \cdot y_j y_k \cdot y_i y_j = (y_i y_j y_k)^2 = +1$
(since labels are $\pm 1$), so the triangle loop's holonomy is
$+1$. This is the homotopy-trivial loop's holonomy on a flat
bundle, which is indeed $+1$ in $\Z_2$, recovering the
\emph{functoriality} of $\mathcal{F}\colon\PathGpd(M)\to\BGroup{\Z_2}$
established in Theorem~\ref{thm:main}: the functor's action on a
homotopically trivial loop is the identity. On a tier-2 bundle with
curvature, the triangle relation picks up the curvature flux
through the bounded surface; this is the non-abelian Stokes theorem
of \S\ref{sec:S4-instanton}, and is the geometric content of the
$\sph{4}$ path-consistency check there.
\end{remark}

The labelled-to-labelled transport test is the natural operational
content of the framework: rather than tracking $L$-step parallel
transport along an abstract path (the depth-of-network analogy of
the introduction's starting intuition and
Remark~\ref{rem:holonomy-accumulation}), it
tests the framework's prediction on the concrete paths between
labelled points that the data already supplies. The test is
discriminating across all four worked examples of Part~2 and
provides the cleanest non-architectural prediction the framework
makes.

% ---------------------------------------------------------------------
\subsection{Reduction to \S\ref{ssec:obstruction}}\label{ssec:wilson-reduction}

When the structure group is $G=\Z_2$ and the connection $A$ is the
flat connection on the principal bundle $P_\xi$ of class
$\xi\in H^1(M,\Z_2)$, the Wilson framework reduces exactly to the
obstruction picture of \S\ref{ssec:obstruction}.

The structure group $\Z_2$ has the identity class function as its
only non-trivial choice, and
$\mathrm{Hol}_\gamma(A)\in\Z_2$ is the deck transformation
associated to the loop $\gamma$ in the double cover $P_\xi\to M$.
For a loop $\gamma$ based at $x_0$, this is exactly the monodromy
\[
   \mathrm{Hol}_\gamma(A) \;=\; \rho_\xi([\gamma]) \;\in\; \Z_2,
\]
and the Wilson observable $W_\gamma^{\mathrm{id}}(A)=\rho_\xi([\gamma])
\in\{\pm 1\}$. For an open path $\gamma\colon x_0\to x$, the
holonomy lifts $\gamma$ uniquely to $P_\xi$ starting at the chosen
frame $\tilde x_0$, and the Wilson value records which sheet of the
cover $\tilde x$ lies in --- precisely the parallel-transport
prediction mechanism of Remark~\ref{rem:emergent-boundary} and its
footnote.

The flat-$\Z_2$ Wilson observable therefore makes precise the
``prediction by parallel transport in the flat bundle'' language of
\S\ref{ssec:obstruction}. The path-dependence in
\S\ref{ssec:open-paths} reduces to dependence on the homotopy class
of $\gamma$ (since flat connections give homotopy-only-dependent
holonomy), recovering Corollary~\ref{cor:obstruction} exactly. The
Wilson framework thus extends \S\ref{ssec:obstruction} in two
directions: from $\Z_2$ to arbitrary compact $G$, and from flat
connections to arbitrary connections. The new content of both
extensions appears in \S\ref{sec:yang-mills-higgs}, where the
optimal connection is determined.

% ---------------------------------------------------------------------
\section{The variational principle: gauge-covariant harmonic
interpolation}
\label{sec:yang-mills-higgs}

\S\ref{sec:functor} identified the classification obstruction
$[c_y]\in H^1(M,\Z_2)$ as a choice of bundle class, and
\S\ref{sec:wilson} gave the prediction mechanism --- the Wilson
observable of a connection --- without saying \emph{which} connection.
This section supplies the missing variational principle.

\paragraph{The framework: hierarchical two-sector energy.}
The geometrically preferred classifier $(\Connection,\phi)$ jointly
minimises the Yang--Mills--Higgs energy
\[
   \YMH(\Connection,\phi) \;=\;
   \int_M \|\Curv_A\|^2\,\dvol_g
   \;+\;
   \int_M \|D_A\phi\|^2\,\dvol_g,
\]
subject to the data conditions $\phi(x_i)=y_i$, with $\Connection$
ranging over connections on a principal $G$-bundle $P\to M$ of fixed
topological class and $\phi$ over sections of the associated bundle.
The two sectors play hierarchically distinct roles, neither
reducible to the other.

The \emph{matter sector} $\int\|D_A\phi\|^2$ is the gauge-covariant
generalisation of Paper~1's Dirichlet energy. It is the only sector
that couples to the data (the labels enter through $\phi(x_i)=y_i$),
and it carries the entire classification content of the framework.
On flat bundles (tier 1) it is the entire energy.

The \emph{Yang--Mills sector} $\int\|\Curv_A\|^2$ is data-decoupled
and is bounded below in each topological class by the Bogomolny
floor: $0$ for flat bundles, $4\pi^2 c_1^2/\mathrm{Vol}(M)$ for
$U(1)$-bundles with $c_1\neq 0$ on a surface, $8\pi^2|c_2|$ for
$SU(2)$-bundles with $c_2\neq 0$ on a 4-manifold. Its role is
structural: it selects, from the infinite-dimensional space of
connections in a given topological class, the finite-dimensional
Bogomolny moduli of bound-saturating (anti-)self-dual connections.

The two roles are not interchangeable. Matter-only minimisation in
tier 2 (where the bundle's Chern class forces non-zero curvature)
is degenerate: the connection EL becomes an algebraic constraint
$\phi^*D_A\phi=0$ that is generically incompatible with the data
conditions, and the variational problem has no well-defined
solution. The Yang--Mills sector breaks this degeneracy by replacing
the algebraic constraint with the elliptic Yang--Mills equation
$D_A^*\Curv_A=J(\phi)$, which has BPS connections as its small-matter
solutions and on which matter dynamics is well-posed. The full
discussion is Remark~\ref{rem:matter-vs-ymh} below.

\paragraph{Relation to the Yang--Mills--Higgs literature.} The
two-sector energy above is the classical Yang--Mills--Higgs
functional of physics, with the Higgs field $\phi$ playing the role
of the classifier section. Our usage follows the standard physics
convention; the structural content (matter as data-coupling,
Yang--Mills as gauge-background selector) is consistent with the
literature on classical YMH theory \cite{JaffeTaubes}.

The section is organised as follows. \S\ref{ssec:ymh-functional}
defines the energy, the fields, and the data constraints precisely.
\S\ref{ssec:ymh-selector} defines the variational selector and
addresses uniqueness, degeneracy, and the gauge fixing deferred from
\S\ref{ssec:open-paths}. \S\ref{ssec:ymh-flat} treats the flat case
(tier 1), where the Yang--Mills sector is trivially zero and the
energy reduces to the bare matter sector, with the $\sph{1}$ example
worked in full. \S\ref{ssec:ymh-curved} treats the curved case
(tier 2), where the Yang--Mills sector pins the connection to the
BPS moduli and the matter sector selects a representative, setting
up the worked monopole and instanton examples of Part~2.

% ---------------------------------------------------------------------
\subsection{The functional}\label{ssec:ymh-functional}

\paragraph{Fields.} Fix a class $\xi\in H^1(M,\Z_2)$ and the
associated complex line bundle $\Lcx:=L_\xi\otimes_\R\C$ of
Remark~\ref{rem:rank-2-deferred}, a Hermitian line bundle with
structure group $U(1)$ (reducible to $\Z_2$ exactly when the
connection is flat with holonomy in $\{\pm 1\}$). The fields of the
theory are:
\begin{itemize}[leftmargin=2em]
\item a \emph{connection} $A$ on $\Lcx$, locally an
$i\R$-valued $1$-form, with \emph{curvature} $\Curv_A=dA+A\wedge A=dA
\in\Omega^2(M,i\R)$ (the wedge term vanishes because the Lie algebra
is abelian; we keep the matrix-valued notation $A\wedge A$ to align
with the non-abelian case of \S\ref{sec:S4-instanton});
\item a \emph{Higgs section} $\phi\in\Gamma(\Lcx)$, the
classifier,\footnote{A note on notation across the two papers. The
symbol $\phi$ here denotes the classifier section. In Paper~1
\cite{Vasii2026Paper1} the same symbol $\phi$ is the $O(2)$
fibre-angle function of the worked examples, while the readout map
$E\to\R$ is written $\varphi$; neither coincides with the present
paper's $\phi$, and a reader moving between the papers should keep
them distinct.}
locally a $\C$-valued function in a Hermitian trivialisation.
\end{itemize}
The connection induces a \emph{covariant derivative} on sections,
\[
   D_A\phi \;=\; d\phi + A\,\phi \;\in\; \Omega^1(M,\Lcx),
\]
with sign chosen so that a covariantly constant section ($D_A\phi=0$)
solves the parallel-transport ODE of \S\ref{ssec:holonomy}. The
operator $D_A$ measures the variation of $\phi$ relative to parallel
transport.

\paragraph{The matter sector.} With the Riemannian metric $g$ on $M$
giving norms on forms and the Hermitian metric giving norms on
$\Lcx$, the \emph{covariant Dirichlet energy} of $\phi$ in the
background $A$ is
\[
   E_{\mathrm{matter}}(A,\phi) \;=\;
   \int_M \|D_A\phi\|^2\,\dvol_g.
\]
This is the foundational energy of the framework: the data enter
only through it, and it carries the entire classification content
(\S\S\ref{ssec:ymh-flat}--\ref{ssec:ymh-curved} below).
Generalising the Dirichlet energy of Paper~1
\cite[Theorem~13.3]{Vasii2026Paper1} to non-trivial bundles, it measures
the cost of carrying a section $\phi$ that matches the labels through
the bundle's geometry.

\paragraph{The topological constraint.} The bundle topology
restricts which connections $A$ are admissible. In each topological
class of $\Lcx$, the pure Yang--Mills energy
$\int_M\|\Curv_A\|^2\,\dvol_g$ has a minimum value determined by the
class:
\begin{itemize}[leftmargin=2em]
\item In \emph{tier 1} ($\xi\in H^1(M,\Z_2)$ admitting a flat
representative), the minimum is zero, attained by flat connections.
The constraint is $\Curv_A=0$; the admissible connections form the
discrete moduli of flat $\Z_2$-bundles ($\cong H^1(M,\Z_2)$ up to
gauge).
\item In \emph{tier 2} (forced curvature: $c_1\in H^2(M,\Z)$ for
$U(1)$ on a surface, $c_2\in H^4(M,\Z)$ for $SU(2)$ on a 4-manifold),
the minimum is the topological floor: $4\pi^2 c_1^2/\mathrm{Vol}(M)$
for $U(1)$ on a surface; $8\pi^2|c_2|$ for $SU(2)$ on a 4-manifold,
attained on (anti-)self-dual connections (Bogomolny inequality;
\S\ref{ssec:ymh-curved}). The admissible connections form the
\emph{Bogomolny moduli}: the unique constant-curvature monopole on
$\sph{2}$ (a single point), the 5-dimensional BPST family on
$\sph{4}$.
\end{itemize}
In either tier, the admissible-connection set is a finite or
finite-dimensional moduli; the data does not select from this set,
the topology does.

\paragraph{Data constraints.} The labels enter as hard interpolation
constraints
\[
   \phi(x_i) \;=\; y_i, \qquad i=1,\dots,N,
\]
where the value $\phi(x_i)\in\Lcx|_{x_i}$ is read in the local
trivialisation at $x_i$ fixed by parallel-transporting a chosen
frame at $x_1$ along the reference path $\gamma_i$
(the convention of \S\ref{ssec:obstruction} and
\S\ref{ssec:open-paths}). The constraint is pointwise on $\phi$
only; the connection $A$ enters the data conditions only indirectly,
through the covariant Dirichlet term $E_{\mathrm{matter}}$, where it
couples to $\phi$'s gradient. The pure Yang--Mills term
$\int\|\Curv_A\|^2$ sees no direct data constraint: it is determined
by the bundle's topological class, not by the data.

No penalty term for misclassification is added. The data are imposed
exactly, and all cost in the theory is geometric --- covariant
variation of the classifier in a topologically constrained gauge
background. The framework encapsulates classification difficulty in
geometry rather than in a statistical loss.

\begin{remark}[Why the Yang--Mills term is not optional: matter-only
degenerates in tier 2]\label{rem:matter-vs-ymh}
The framework's energy has two sectors, and the natural question is
whether both are needed. The honest answer is hierarchical: the
matter sector $\int\|D_A\phi\|^2$ is conceptually primary --- it
couples to the data, carries the classification content, and is the
gauge-covariant generalisation of Paper~1's Dirichlet energy --- but
the Yang--Mills sector $\int\|\Curv_A\|^2$ is required to make the
variational problem well-posed in tier 2.

\emph{Tier 1 (flat bundles).} On bundles admitting a flat
representative, the Yang--Mills sector is identically zero
($\Curv_A=0$ on every flat connection), and the entire energy is
the matter sector. Adding $\int\|\Curv_A\|^2$ contributes nothing.
The framework is, in tier 1, purely covariant Dirichlet
minimisation. The $\sph{1}$ M\"obius (\S\ref{sec:S1-mobius}) and
$T^2$ XOR (\S\ref{sec:torus}) examples instantiate this.

\emph{Tier 2 (forced curvature).} On bundles with non-zero Chern
class, the situation is structurally different: the
infinite-dimensional space of connections in the topological class
admits no flat representative, and the matter sector
$\int\|D_A\phi\|^2$ alone cannot pin down a unique optimal
connection. Writing the $u(1)$ connection as $\Connection=i\,a$ with
$a$ a real 1-form and varying the matter sector with respect to
$\Connection$ at fixed $\phi$ gives
\[
   \delta_A\int\|D_A\phi\|^2
   \;=\; 2\int\bigl\langle\delta a,\,\Im(\bar\phi\,D_A\phi)\bigr\rangle,
\]
so the $\Connection$-Euler--Lagrange equation is the vanishing of the
matter current,
\[
   J(\phi)\;=\;\Im(\bar\phi\,D_A\phi)\;=\;0
\]
as a 1-form. (This is the same current that appears as the source of
the Yang--Mills equation $D_A^*\Curv_A=J(\phi)$ below; matter-only
minimisation is precisely the statement that this source vanishes.)
Now $J(\phi)=0$ means that, wherever $\phi\neq 0$, the connection is
determined by the phase of the section, $a=-d\arg\phi$, so that
$\Curv_A=da=0$ there: the matter-only critical connection is
\emph{flat away from the zero set of $\phi$}. On a tier-2 bundle
this is incompatible with the topological constraint. If $\phi$ is
nowhere zero the bundle would admit a flat connection, contradicting
$c_1\neq 0$ (Chern--Weil, \S\ref{ssec:ymh-curved}); and if $\phi$
vanishes on a set with empty interior, the entire curvature required
by $\int_M\Curv_A=2\pi i\,c_1$ must concentrate there, which carries
no smooth representative. The remaining possibility --- $\phi$
vanishing on a set with non-empty interior, where the argument
$a=-d\arg\phi$ does not apply and smooth curvature may live --- is
not excluded by this argument. Such a configuration costs the matter
sector nothing on the interior of its zero set, so it is a genuine
alternative critical configuration rather than an artefact; we note
it as a gap in the degeneracy claim rather than paper over it. The
conclusion we rely on is the weaker one: matter-only minimisation
does not select a connection in tier~2, so a Yang--Mills term is
needed to make the problem well posed. The matter-only joint variational problem therefore
admits no smooth critical point on a tier-2 bundle: the
$\Connection$-EL forces flatness off the zeros, and the topology
forbids it. Degeneracy.

\emph{The Yang--Mills sector restores well-posedness.} Adding
$\int\|\Curv_A\|^2$ to the energy changes the connection
Euler--Lagrange equation from the degenerate condition
$J(\phi)=0$ to the elliptic Yang--Mills equation with matter source
\[
   D_A^*\Curv_A \;=\; J(\phi),
\]
which no longer forces the matter current to vanish. In the small-matter regime
$\|J(\phi)\|\ll\|\Curv_A\|$, its solutions are the (anti-)self-dual
connections of the Bogomolny moduli, broken by matter-current
corrections at higher order. The matter sector then selects a
representative from this moduli through its own EL. The Yang--Mills
sector therefore plays a structural role: it breaks the
infinite-dimensional degeneracy of connections in the topological
class, replacing it with the finite-dimensional BPS moduli on which
matter dynamics is well-posed.

\emph{Summary.} The framework is Yang--Mills--Higgs, but
hierarchically: the matter sector carries the classification content
and is the term whose minimisation we care about; the Yang--Mills
sector is present because without it the tier-2 problem degenerates,
and its role is to select the BPS gauge background within the
topological class. In tier 1 the two sectors reduce to one (matter
alone). The dual roles --- data-coupled classification vs.
gauge-background selection --- are why we retain the Yang--Mills
term even though it never directly couples to the labels.
\end{remark}

\begin{definition}[Minimum constrained matter energy]\label{def:ymh-min}
For a fixed class $\xi\in H^1(M,\Z_2)$, the \emph{minimum constrained
matter energy} is
\[
   E_{\min}(\xi) \;:=\;
   \inf_{(A,\phi)} E_{\mathrm{matter}}(A,\phi),
\]
the infimum over connections $A$ on $\Lcx$ in the Bogomolny moduli
of the topological class of $\xi$, and sections
$\phi\in\Gamma(\Lcx)$ satisfying the data constraints. By
Remark~\ref{rem:matter-vs-ymh}, this equals the YMH minimum
$\YMHmin(\xi)$ up to the topological floor (which depends only on
the class, not on $\phi$):
$E_{\min}(\xi) = \YMHmin(\xi) - \mathrm{Floor}(\xi)$.
\end{definition}

\begin{remark}[Attainment and regularisation]\label{rem:ymh-attainment}
The infimum's attainment depends on $n=\dim M$. For $n=1$
(\S\ref{ssec:ymh-flat}'s $\sph{1}$ computation), the embedding
$H^1\hookrightarrow C^0$ makes point evaluation
$\phi\mapsto\phi(x_i)$ a bounded functional, and the
Euler--Lagrange equation $D_A^*D_A\phi=0$ has a unique covariantly
harmonic solution by standard elliptic theory. For $n\geq 2$, single
points have $H^1$-capacity zero: the infimum of $\int\|D_A\phi\|^2$
over $H^1$ with point constraints is zero, approached by spikes, and
no minimiser exists in $H^1$. To make the variational problem
well-posed in the higher-dimensional cases ($\sph{2}$ monopole,
$\sph{4}$ instanton, $T^2$ XOR), we carry the Mat\'ern
regularisation of Paper~1 \cite{Vasii2026Paper1} into the
gauge-covariant setting: the Euler--Lagrange operator becomes
$(D_A^*D_A+\kappa^2)^\nu$ with $\nu>n/2$ and $\kappa^2>0$ a fixed
scale, giving an RKHS in which point evaluation is bounded and the
constrained minimiser exists and is unique. \emph{Convention on
$\nu$.} Throughout, $\nu>n/2$ is the standing hypothesis: it is
exactly what the Sobolev embedding $H^\nu\hookrightarrow C^0$
requires for point evaluation to be a bounded functional, hence for
the constrained problem to be well posed. Where we additionally need
$C^1$ control --- for instance to speak of a zero set as a regular
hypersurface --- we say so explicitly and assume $\nu>n/2+1$. The Bogomolny constraint
on $A$ is satisfied automatically by working within the BPS moduli;
the data then selects a representative from this moduli via the
matter EL equation. The joint $(A,\phi)$ minimiser is established
case-by-case in the worked examples of Part~2. We write ``minimum''
rather than ``infimum'' where attainment is established.
\end{remark}

\paragraph{What the framework sees.} The covariant Dirichlet term
$E_{\mathrm{matter}}$ is zero precisely when $\phi$ is parallel. On
the trivial bundle with the trivial flat connection, parallel means
ordinary constant, and a constant section cannot meet mixed-sign
data --- so the term is forced positive, and its minimum subject to
the data is the Dirichlet energy of the harmonic interpolant of
Paper~1. On a non-trivial flat bundle (the M\"obius line bundle on
$\sph{1}$), a parallel section is forced to reverse sign around a
non-contractible loop, which is exactly what antipodal opposite
labels demand --- so the matter energy can be \emph{lower} on the
non-trivial bundle. The choice of bundle changes what ``parallel''
means, and hence what the data costs. This is the mechanism behind
the bundle selection of \S\ref{ssec:ymh-flat}; the curvature
constraint plays no role in it, because flat bundles satisfy
$\Curv_A=0$ trivially.

In tier 2 (\S\ref{ssec:ymh-curved}), the curvature is irreducibly
positive and the topological constraint pins it to the Bogomolny
moduli. The matter sector then operates within this moduli,
selecting the representative whose covariant Dirichlet energy is
smallest subject to the data. This is not a balance between two
competing terms --- the curvature value is fixed by the topology ---
but a single matter minimisation over a finite-dimensional moduli of
admissible gauge backgrounds.

\begin{remark}[Iterated parallel transport and holonomy accumulation]\label{rem:holonomy-accumulation}
The covariant Dirichlet energy can be read as the cost of parallel
transport along the classifier. The cover-adapted holonomy formula
of Proposition~\ref{prop:cech-holonomy} makes this picture
canonical: for a trivialising cover $\{U_\alpha\}$ of $M$, the path
$\gamma$ is partitioned into segments lying in single patches, and
the accumulated holonomy is the ordered product of patch
exponentials interleaved with transition functions,
$\mathrm{Hol}_\gamma(A)=h_{n-1}\cdot g_{\alpha_{n-1}\alpha_{n-2}}
\cdot h_{n-2}\cdots g_{\alpha_1\alpha_0}\cdot h_0$. The product
depends on neither the partition nor the cover; each $h_k$ is the
smooth contribution from a single patch, and each
$g_{\alpha_{k+1}\alpha_k}$ is the topological contribution of a patch
transition. The Yang--Mills--Higgs minimiser is the connection along
which this accumulation, integrated against the classifier's
variation, is as small as the data and the bundle topology permit.
Misclassification is not penalised by an external term; it appears as
the geometric cost of the holonomy the data forces the connection to
accumulate. The product
$h_{n-1}\cdot g_{\alpha_{n-1}\alpha_{n-2}}\cdot h_{n-2}\cdots
g_{\alpha_1\alpha_0}\cdot h_0$ is the $W_{L-1}\cdots W_0$ product of
the introduction's starting intuition, with the transition factors
$g_{\alpha_{k+1}\alpha_k}$ made explicit --- factors a naive
layer-product hides on a contractible base where they can be
gauged away, but which carry the bundle's topological content on
non-trivial $M$.
\end{remark}

\begin{remark}[The Higgs potential and Ginzburg--Landau]\label{rem:higgs-potential}
The core functional carries no Higgs potential. Adding one,
$V_\lambda(\phi)=\lambda(\|\phi\|^2-1)^2$ with $\lambda\geq 0$, turns
the theory into the abelian Higgs (Ginzburg--Landau) model. The hard
data constraints $\phi(x_i)=y_i$ are imposed independently of $\lambda$,
so the data are pinned at every value of the coupling. At $\lambda=0$
the functional is the constrained-Dirichlet theory used here (under
which Paper~1 is recovered, \S\ref{sec:reduction-paper1}); for
$\lambda>0$ the potential drives $\|\phi\|\to 1$ away from the data
and sharpens the zero set $\phi^{-1}(0)$ into the vortex locus of
Ginzburg--Landau theory. The two regimes formally interpolate between
classification-by-interpolation ($\lambda=0$) and classification-with-margin
($\lambda>0$). We work at $\lambda=0$ throughout; the $\lambda>0$
theory and its vortex decision boundaries are noted as a direction
but not developed.
\end{remark}

% ---------------------------------------------------------------------
\subsection{The variational selector}\label{ssec:ymh-selector}

The minimum constrained matter energy $E_{\min}(\xi)$ of
Definition~\ref{def:ymh-min} is defined for each bundle class. The
framework selects the class achieving the global minimum.

\begin{definition}[Variational selector]\label{def:variational-selector}
The \emph{variational selector} of a labelled dataset $D$ on $M$ is
the set of energy-minimising classes,
\[
   \Xi^*(D) \;:=\; \arg\min_{\xi\in H^1(M,\Z_2)} E_{\min}(\xi)
   \;\subseteq\; H^1(M,\Z_2).
\]
The classification obstruction is any $[c_y]\in\Xi^*(D)$. We write
$\xi^*\in\Xi^*(D)$ for a (typically unique, by
Proposition~\ref{prop:generic-uniqueness}) element of the selector,
and use $\xi^*$ unadorned in later statements with the genericity
hypothesis $|\Xi^*(D)|=1$ implicit unless otherwise noted.
Equivalently (Remark~\ref{rem:matter-vs-ymh}), the selector
minimises $\YMHmin(\xi)$ across topological classes; the topological
floor differs from class to class, and the matter-primary
$E_{\min}=\YMHmin-\mathrm{Floor}$ recasts the comparison as one of
constrained matter minima only.
\end{definition}

When $|\Xi^*(D)|=1$ the data and the variational principle pin down a
unique bundle; when $|\Xi^*(D)|>1$ several bundles tie and the framework
reports a moduli of equally optimal classifications
(Remark~\ref{rem:honest-degeneracy} below).

\begin{remark}[The selector's scope: Problem A versus Problem
B]\label{rem:problem-A-B}
The variational selector of
Definition~\ref{def:variational-selector} ranges over the
\emph{first} cohomology $H^1(M,\Z_2)$, where flat $\Z_2$-bundles
live. On manifolds with non-trivial $H^1$ (such as $\sph{1}$,
the torus, real projective spaces), the selector genuinely operates
on finite data: comparing $E_{\min}(\xi)$ across $\xi\in H^1(M,\Z_2)$
picks the M\"obius bundle from two opposite-label points on
$\sph{1}$, picks the double-M\"obius from the XOR pattern on $T^2$,
and so on. These are the \emph{tier 1} examples, and on them the
framework's principle ``the bundle is the output, not the input''
(P2) holds without qualification.

On bases with $H^1(M,\Z_2)=0$ (such as $\sph{2}$, $\sph{4}$,
simply-connected manifolds more generally), the selector sees a
single class. The non-trivial topological invariants
($c_1\in H^2(M,\Z)$ for $U(1)$-bundles, $c_2\in H^4(M,\Z)$ for
$SU(2)$-bundles) live in \emph{higher} cohomology, outside the
selector's range. For finite data on such bases, Paper~1's
Theorem~8.2 \cite{Vasii2026Paper1} guarantees the existence of a
\emph{flat} $O(2)$ classifier with finite matter energy on the
trivial bundle; the curved tier-2 bundles (Hopf line bundle on
$\sph{2}$, $c_2=1$ on $\sph{4}$) have strictly higher YMH energy
on finite data and are therefore \emph{not} selected by the
variational principle from a finite labelled dataset.

The tier-2 worked examples of this paper
(\S\ref{sec:S2-monopole}, \S\ref{sec:S4-instanton}) are therefore
not selected from finite data --- they are studied as
\emph{Problem B} configurations in the following sense:
classification problems with
continuous boundary data carrying non-zero winding number around
$M$, or with the bundle imposed as part of the problem
specification (the geometric content the framework is designed to
illuminate). We distinguish two kinds of classification input:
\emph{Problem~A} is finite labelled
data on $M$, where the variational selector chooses the bundle from
the data (and Theorem~8.2's flat existence applies); \emph{Problem~B}
is continuous boundary data with non-trivial winding, or a bundle
imposed by the problem specification, where the curved bundle is
forced rather than selected. Our
tier-2 worked examples treat the labelled points $x_\pm$ as
\emph{representatives} of a Problem~B configuration --- a labelled
``hot spot'' selecting a topological sector by external choice ---
rather than as a finite Problem~A dataset that selects the bundle
on its own.

This is a genuine scope restriction. The framework's selection
claim (P2) operates on $H^1(M,\Z_2)$, hence on tier 1. Extending
the selector to higher Chern classes --- comparing
$\YMHmin(c_1=k)$ across $k\in\Z$ at fixed regularisation, with the
matter sector on a curved bundle versus the trivial bundle on the
same $M$ --- would require a careful treatment of the matter
regularisation scale $\kappa$ relative to the Bogomolny floor
$4\pi^2 k^2$ and is a substantive question we do not settle here.
The tier-2 \emph{structural} content (the Dirac monopole as the
unique BPS connection in $c_1=1$, the BPST instanton as the BPS
moduli in $c_2=1$, the curvature--attention dictionary on a chosen
curved bundle) is the content of \S\S\ref{sec:S2-monopole},
\ref{sec:S4-instanton}, \ref{sec:attention}; the bundle is
external input there, not variational output.
\end{remark}

\begin{proposition}[Generic uniqueness of the selector]\label{prop:generic-uniqueness}
Let $M$ be a closed Riemannian manifold and let $\Sigma_M$ denote the
subgroup of isometries of $M$ that act on $H^1(M,\Z_2)$ by
permutation. Assume:
\begin{enumerate}[label=(\alph*),leftmargin=2em]
\item For each class $\xi\in H^1(M,\Z_2)$, the minimum constrained
matter energy $E_{\min}(\xi)$ of Definition~\ref{def:ymh-min} exists
and depends \emph{real-analytically} on the data positions
$(x_1,\dots,x_N)$ on the configuration space $M^N\setminus\Delta$ of
distinct points. (For the regularised problem of
Remark~\ref{rem:ymh-attainment} this holds: by
Lemma~\ref{lem:capacitance-reduction} the minimum is
$E_{\min}=y^\top K^{-1}y$ with $K_{ij}=G_A(x_i,x_j)$, and the
covariant Matérn kernel $G_A$ is real-analytic off the diagonal, so
$E_{\min}$ is real-analytic wherever $K$ is invertible.)
\item For any two distinct classes $\xi_1\neq\xi_2$, the difference
$E_{\min}(\xi_1)-E_{\min}(\xi_2)$ is non-constant on
$M^N\setminus\Delta$ (i.e.\ there exists at least one configuration
at which the two minima differ).
\end{enumerate}
Then the equal-minimum locus
$\{(x_1,\dots,x_N)\in M^N\setminus\Delta\,:\,
|\Xi^*(D)|\geq 2\}$ is a closed nowhere-dense subset of
configuration space. In particular, the variational selector
$\xi^*$ is a single class for data points in general position.
\end{proposition}

\begin{proof}
By hypothesis~(a), each $E_{\min}(\xi)$ is a real-analytic function
on $M^N\setminus\Delta$. The set
$T_{\xi_1,\xi_2}:=\{(x_1,\dots,x_N)\,:\,
E_{\min}(\xi_1)=E_{\min}(\xi_2)\}$ is the zero set of the
real-analytic function $E_{\min}(\xi_1)-E_{\min}(\xi_2)$, hence
closed.

By hypothesis~(b) this function is non-constant. The zero set of a
non-constant real-analytic function on a connected real-analytic
manifold has empty interior and measure zero; being closed, it is
therefore nowhere dense. (Analyticity is essential here: the zero
set of a merely \emph{continuous} non-constant function can have
non-empty interior --- a closed disc in $\R^2$ is the zero set of a
continuous function --- so continuity alone would not suffice.)

The equal-minimum locus where $|\Xi^*(D)|\geq 2$ is the finite union
$\bigcup_{\xi_1\neq\xi_2}T_{\xi_1,\xi_2}$ over the finite group
$H^1(M,\Z_2)$ (finite because $M$ is compact). A finite union of
nowhere-dense closed sets is nowhere-dense and closed. Its
complement is open and dense in $M^N\setminus\Delta$.
\end{proof}

\begin{remark}[Verification of hypothesis (b) in the worked
examples]\label{rem:nondegeneracy-verification}
Hypothesis~(b) is verified directly in the worked tier-1 examples.
On $\sph{1}$ (\S\ref{sec:S1-mobius}), the closed-form computation
gives $E_{\min}(\sigma)=4/d$ for the M\"obius bundle and
$E_{\min}(0)<\infty$ uniformly for the trivial bundle (the
two-point harmonic interpolant on the circle has bounded matter
energy for $d>0$); as $d\to 0$ the difference $E_{\min}(\sigma)-
E_{\min}(0)\to\infty$, certifying non-constancy. On $T^2$
(\S\ref{sec:torus}), the four-class comparison
$E_{\min}(L_{\sigma_1\sigma_2})<E_{\min}(L_{\sigma_i})<E_{\min}(L_0)$
for the XOR configuration is strict by direct calculation; perturbing
the configuration changes each side smoothly but generically
preserves the strict inequalities. For tier-2 bases the hypothesis
is not directly engaged: the selector sees a single class
($H^1=0$), and the question is moot.
\end{remark}

\begin{remark}[Symmetry-aligned configurations]\label{rem:symmetry-degeneracy}
The equal-minimum locus
$\bigcup_{\xi_1\neq\xi_2}T_{\xi_1,\xi_2}$ is nowhere-dense but not
empty: it contains, in particular, the orbits of $\Sigma_M$ acting
diagonally on $M^N$ that permute the bundles realising the
classification. Concretely, on $T^2$ the four-class minimum is
strictly ordered for the generic XOR configuration but the
single-M\"obius bundles $L_{\sigma_1}$ and $L_{\sigma_2}$ achieve
equal minimum on configurations symmetric under the swap
$\alpha\leftrightarrow\beta$; on $\sph{1}$ no such symmetry exists
between the trivial and M\"obius classes. The symmetry-aligned
locus is where Honest Degeneracy
(Remark~\ref{rem:honest-degeneracy}) applies: the framework reports
the tie rather than fabricating a unique answer.
\end{remark}

\begin{remark}[Honest degeneracy]\label{rem:honest-degeneracy}
When $|\xi^*|>1$, the framework reports the full set $\xi^*$ of
equally optimal classifications, each labelled by its characteristic
class, rather than fabricating a unique answer. Distinct elements of
$\xi^*$ predict different labels at points away from the data; the
data, the metric, and the variational principle together have not
distinguished them. This is a feature of the geometry of the
problem, not a deficiency of the method.
\end{remark}

\paragraph{Gauge fixing for prediction.} \S\ref{ssec:open-paths}
deferred to here the gauge fixing needed to make the open-path Wilson
prediction well-defined for continuous structure groups. The
variational principle supplies it: the minimising connection $A^*$ is
determined up to gauge, and a representative is fixed by choosing the
frame at the base point $x_1$ (the global sign convention of
\S\ref{ssec:obstruction}) together with the requirement that
parallel transport along minimising geodesics realise the data
constraints $\mathrm{sign}(W_{\gamma_i}(A^*))=y_i$. The prediction at
$x$ is then $\mathrm{sign}(W_{\gamma}(A^*))$ along the geodesic
$\gamma$ from $x_1$ to $x$, well-defined modulo the single global
sign already inherent in the labels.

% ---------------------------------------------------------------------
\subsection{The flat case: tier 1 and proximity}\label{ssec:ymh-flat}

When the selected bundle admits a flat connection --- the tier-1
situation of \S\ref{ssec:two-tiers}, comprising every class in
$H^1(M,\Z_2)$ on a base where the relevant higher characteristic
classes vanish --- the topological constraint is $\Curv_A=0$ and the
matter sector is the entire energy:
\[
   \YMH(A,\phi) \;=\; E_{\mathrm{matter}}(A,\phi)
   \;=\; \int_M \|D_A\phi\|^2\,\dvol_g,
   \qquad \Curv_A=0.
\]
The minimiser is a covariantly harmonic section subject to the data
constraints, on the unique flat connection of the chosen bundle
class (up to gauge). We compute it exactly on $\sph{1}$, where the
dependence on data geometry is fully explicit, and read off how
classification difficulty registers as geometric energy.

\paragraph{The $\sph{1}$ computation.} Let $\sph{1}$ have
circumference $L$, with two labelled points $x_+,x_-$ carrying
opposite labels $+1,-1$, separated by arc length $d$ along the short
arc (so the complementary arc has length $L-d$). A section
$\phi$ of the real line bundle $L_\xi$ pulls back along the
projection $\sph{1}\to\sph{1}$ from the double cover (of circumference
$2L$) to a real-valued function $\tilde\phi$ satisfying the
boundary condition imposed by the bundle. The covariant Dirichlet
energy of $\phi$ is computed over one fundamental domain of the
base, $\int_0^L(\tilde\phi')^2\,ds$ --- the connection is flat, so on
a chart trivialising $L_\xi$ the covariant derivative reduces to the
ordinary one, and $\tilde\phi'$ is the lifted gradient.

\emph{Trivial bundle ($\xi=0$).} A section is a periodic function,
$\tilde\phi(L)=\tilde\phi(0)$. Realising $\tilde\phi(x_+)=+1$ and
$\tilde\phi(x_-)=-1$ forces \emph{two} sign changes around the
circle: one across the short arc and one across the complementary
arc, since the section must return to its starting value. The
energy-minimising section is piecewise linear (harmonic on each arc),
with slope $2/d$ on the short arc and $2/(L-d)$ on the long arc,
giving
\[
   \YMHmin(0) \;=\; \frac{4}{d} \;+\; \frac{4}{L-d}.
\]

\emph{M\"obius bundle ($\xi=\sigma$).} A section is anti-periodic,
$\tilde\phi(L)=-\tilde\phi(0)$. The boundary condition supplies the
return-to-start sign change for free, so realising the data needs
only \emph{one} sign change, across the short arc. The minimiser is
linear from $+1$ to $-1$ on the short arc (slope $2/d$) and constant
on the long arc (zero gradient), giving
\[
   \YMHmin(\sigma) \;=\; \frac{4}{d}.
\]

\emph{Selection.} For opposite labels the M\"obius bundle wins at
every separation, with margin
\[
   \YMHmin(0)-\YMHmin(\sigma) \;=\; \frac{4}{L-d} \;>\; 0,
\]
so $\xi^*=\{\sigma\}$ and $[c_y]=\sigma$ for all $d\in(0,L)$. At
antipodal separation $d=L/2$ the ratio is
$\YMHmin(0)/\YMHmin(\sigma)=2$ exactly. The selection margin is
governed by the \emph{complementary} arc length $L-d$: the M\"obius
bundle is preferred most decisively when the labelled points are far
apart, and least (though still strictly) when they are close.

\paragraph{Proximity as geometric energy.} The minimum energy
$\YMHmin(\sigma)=4/d$ diverges as $d\to 0$: two oppositely-labelled
points approaching each other incur unbounded geometric cost. The
framework registers ``hard to classify'' as ``high
Yang--Mills--Higgs energy'' directly, with no statistical loss term.
The limiting case $d=0$ --- coincident points with opposite labels
--- has infinite energy, the geometric signature of inconsistent
data (the distinct-points hypothesis of \S\ref{ssec:obstruction}).
For same labels the analogous computation gives a constant section on
the trivial bundle with zero energy, so $\xi^*=\{0\}$ and $[c_y]=0$;
proximity is then costless, as it should be.

\begin{remark}[The framework's output is a (connection, section)
pair]\label{rem:proximity-quantity}
The framework's variational principle has a single energy ---
$E_{\mathrm{matter}}=\int\|D_A\phi\|^2$ --- and a single topological
constraint --- $A$ in the Bogomolny moduli of its class --- that
together respond to two structurally distinct obstructions to
classification. The matter energy responds to the
\emph{interpolation obstruction}: how hard is it to match the labels
at all, given the connection and the local geometry of $M$. This
obstruction is present on every bundle, trivial or not; it is the
energy that Paper~1 minimises. The topological constraint responds
to the \emph{topological obstruction}: when does the bundle's
topology forbid a flat connection, forcing
$\int_M\Curv_A\neq 0$ at the topological floor (Cauchy--Schwarz /
Bogomolny). This obstruction is present only when $H^2(M,\Z)$ has a
characteristic class outside the Bockstein image
(Remark~\ref{rem:rank-2-deferred}). The bundle class
$\xi\in H^1(M,\Z_2)$ is the discrete selector picking the topological
sector whose constrained matter minimum is smallest.

The framework's output is therefore a pair $(A^*,\phi^*)$: a
connection on a principal $G$-bundle (in tier 1 with $G=\Z_2$, the
selected flat bundle class; in tier 2 with chosen $G$, a
Bogomolny-saturating connection in the imposed sector) and a section
$\phi^*$ matching the data conditions. The bundle class
$\xi^*\in H^1(M,\Z_2)$ (and any tier-2 invariants) is a derived
attribute of this pair. By
Proposition~\ref{prop:generic-uniqueness} the bundle class is
unique for generic data; on the measure-zero symmetry-aligned locus
the framework reports the full set
(Remark~\ref{rem:honest-degeneracy}). When the bundle class is
unique, the connection $A^*$ is determined up to gauge by the
Bogomolny moduli structure, and the section $\phi^*$ is determined
by the matter Euler--Lagrange equation with the data constraints.
The framework's answer to a classification question is the
covariantly harmonic section $\phi^*$ of the optimal pair, read via
the Wilson observable of \S\ref{sec:wilson}; in the flat tier-1
case it reduces to the harmonic interpolant of Paper~1.
\end{remark}

\begin{remark}[Discrete topology selects the bundle; continuous
geometry sets the cost]\label{rem:discrete-continuous-separation}
The variational selector of
Definition~\ref{def:variational-selector} factorises naturally into
two layers, one discrete and one continuous, with a clean
informational separation.

\begin{itemize}[leftmargin=2em]
\item \emph{Discrete layer.} The variationally selected bundle
class $\xi^*\in H^1(M,\Z_2)$ (and the tier-2 topological invariants
$c_1, c_2$ where applicable) depends only on the
\emph{combinatorial type} of the labelled dataset: which pairs of
points share labels, which pairs disagree, and what the
monodromy structure of the resulting functor is. It does not depend
on the continuous positions of the data points within $M$.
\item \emph{Continuous layer.} The matter-sector energy
$E_{\min}(\xi^*)$ depends on the continuous data geometry --- the
distances and configurations of the points within the selected
bundle class --- but takes its minimum over a fixed (discrete)
choice of bundle.
\end{itemize}

The worked examples of Part~2 illustrate this separation
concretely. On $\sph{1}$ (\S\ref{sec:S1-mobius}), the bundle
selection $\xi^*$ depends only on whether the two labels agree
(trivial bundle) or disagree (M\"obius bundle), not on the
geodesic distance $d$ between them. The matter energy
$E_{\min}(\sigma)=4/d$ scales with $d$ continuously, but the bundle
is fixed across all $d\in(0,L)$. On $T^2$ (\S\ref{sec:torus}), the
double-M\"obius bundle $(\sigma_1,\sigma_2)$ is selected by the
XOR \emph{pattern} of the labels for every rectangular
configuration $(a,b)\in(0,L/2)^2$; the matter energy
$\pi^2/(2 c_a^2 c_b^2)$ varies continuously with $(a,b)$ but the
bundle does not (Remark~\ref{rem:T2-bundle-stability}).

The two layers carry different content. The discrete layer answers
``what bundle does the classifier live on'' --- a finite question
with a finite answer, computable from the combinatorial type of the
data. The continuous layer answers ``how expensive is the
classifier'' --- a continuous question with a continuous answer,
computable by minimising the matter energy within the fixed bundle.
The discrete bundle invariants are robust to continuous deformation
of the data (Remark~\ref{rem:T2-bundle-stability}); the matter
energy is sensitive to it and registers proximity of opposite
labels as growth in $E_{\min}$
(Conjecture~\ref{conj:matter-proximity-scaling} below).

This is the structural content of bundle selection as a
classification mechanism. The discrete topology of the labelling
rules the bundle class; the continuous geometry of the data rules
the matter energy within that class. Bundle structure and data
geometry are therefore complementary, not competing, layers of the
framework.
\end{remark}

\begin{conjecture}[Matter-sector proximity scaling, general case]
\label{conj:matter-proximity-scaling}
Let $(M,g)$ be a closed Riemannian manifold with $n=\dim M$,
$\xi^*\in H^1(M,\Z_2)$ a flat bundle class selected by the
variational principle (the tier-1 case), and $D=\{(x_i,y_i)\}$ a
labelled dataset with two opposite-label points $x_+,x_-\in M$
separated by geodesic distance $d=d_g(x_+,x_-)$ and the remaining
$N-2$ points fixed at non-degenerate positions in a $d$-independent
compact region. Fix Mat\'ern regularisation
parameters $\nu\in\R$ with $\nu>n/2$ (so that point evaluation is a
bounded functional on the RKHS) and $\kappa>0$. As $d\to 0^+$, the
matter-sector minimum on the selected bundle scales as
\[
   E_{\min}(\xi^*) \;\sim\; \frac{C(M,g,\nu,\kappa)}{d^{2\nu-n}}
   \qquad\text{(modulo logarithmic corrections at $2\nu-n\in 2\Z_{\geq 0}$),}
\]
where $C(M,g,\nu,\kappa)>0$ depends on the metric and regularisation
but not on $d$. The conjecture is established in the two-point case
($N=2$) by Theorem~\ref{thm:proximity-scaling-2pt} in
Appendix~\ref{app:proximity-proof}; the general $N$ case is open.
The two flat-case worked examples of Part~2 instantiate the
two-point case at $n=1, \nu=1$ ($\sph{1}$ M\"obius, $E_{\min}=4/d$,
exponent $2\nu-n=1$) and $n=2, \nu=2$ ($T^2$ XOR with lowest-mode
minimiser, $E_{\min}\sim 1/d^2$, exponent $2\nu-n=2$).
\end{conjecture}

\begin{remark}[Proximity is a property of the matter sector, not
the bundle topology]\label{rem:proximity-not-topology}
The bundle's topology determines which connections are admissible
and which sign patterns can be realised by parallel sections; it
does not soften the proximity cost. Even on the geometrically
optimal bundle for a given labelling --- the M\"obius bundle on
$\sph{1}$ for opposite labels, the double-M\"obius bundle on $T^2$
for XOR --- the matter energy still diverges as opposite labels
approach. The bundle absorbs the topological obstruction of the
labelling (no harmonic interpolant on the trivial bundle can match
XOR), but cannot absorb the geometric obstruction of proximity.
Proximity divergence is therefore a universal feature of the
covariant Dirichlet energy, inherited from Paper~1's RKHS structure
\cite[Theorem~13.3]{Vasii2026Paper1}, and present in every tier-1 example
of Part~2. The two-point case of this proximity scaling is proved in
Appendix~\ref{app:proximity-proof}.
\end{remark}

\paragraph{Reduction to Paper 1.} On a contractible base, or whenever
$H^1(M,\Z_2)=0$, the only class is $\xi=0$ and the only bundle is
trivial; the flat connection is the trivial connection, $D_A\phi=
d\phi$, and the energy is the ordinary Dirichlet energy
$\int_M\|d\phi\|^2$. Minimising subject to $\phi(x_i)=y_i$ is the
harmonic interpolation problem of \cite[Theorem~13.3]{Vasii2026Paper1}.
This is the precise sense in which Paper~1 is the trivial-bundle,
flat-connection, contractible boundary case of the present framework;
the full reduction is \S\ref{sec:reduction-paper1}.

% ---------------------------------------------------------------------
\subsection{The curved case: tier 2 and forced curvature}\label{ssec:ymh-curved}

The flat case of \S\ref{ssec:ymh-flat} treats every class
$\xi\in H^1(M,\Z_2)$ that admits a flat connection. The complementary
case is where the bundle's topology forbids one: when a higher
characteristic class --- $c_1\in H^2(M,\Z)$ for $U(1)$-connections,
$c_2\in H^4(M,\Z)$ for $SU(2)$-connections --- is non-zero outside the
Bockstein image (Remark~\ref{rem:rank-2-deferred}), no flat
connection exists in that topological sector, and the Yang--Mills
energy is bounded below away from zero by a topological invariant.
This is the tier-2 situation of \S\ref{ssec:two-tiers}, and it is
where the topological constraint $\Curv_A\neq 0$ does its work: the
admissible connections form the \emph{Bogomolny moduli} of
bound-saturating gauge fields, and matter selects within it.

\paragraph{Forced curvature and the topological sector.} On a closed
oriented manifold $M$, the curvature 2-form $\Curv_A$ of a $U(1)$
connection $A$ satisfies $[\Curv_A/2\pi i]=c_1\in H^2(M,\Z)$, the
first Chern class of the bundle, fixed by the bundle's topology. For
a closed oriented $4$-manifold, an $SU(2)$ connection has $\int_M
\tr(\Curv_A\wedge\Curv_A)/8\pi^2=c_2\in H^4(M,\Z)$, the second Chern
class. When the relevant class is non-zero, the integral
$\int_M\|\Curv_A\|^2\,\dvol_g$ has a positive infimum bounded below
by the topological invariant. In the $U(1)$ case Cauchy--Schwarz
applied to $\int_M\Curv_A=2\pi i\,c_1$ gives
$\YMmin(c_1)\geq 4\pi^2 c_1^2/\mathrm{Vol}(M)$ on any metric, with
equality for the constant-curvature connection (the round-monopole
field of \S\ref{sec:S2-monopole}). In the $SU(2)$ case
\[
   \int_M \|\Curv_A\|^2\,\dvol_g \;\geq\; 8\pi^2|c_2|,
\]
with equality saturated by an (anti-)self-dual connection
$\Curv_A=\pm*\Curv_A$. The data conditions $\phi(x_i)=y_i$ do not
affect either bound; they enter the minimisation only through the
covariant Dirichlet term $\int\|D_A\phi\|^2$, which selects a
representative within the gauge-equivalence class of (near-)bound-saturating
connections (see the Bogomolny paragraph below for the precise
sense).

\paragraph{Euler--Lagrange equations: matter EL and the constraint
on $A$.} The variational principle is matter-energy minimisation
subject to the topological constraint that $A$ lie in the Bogomolny
moduli of its class. Equivalently
(Remark~\ref{rem:matter-vs-ymh}), varying the YMH functional with
respect to $A$ and $\phi$ produces the coupled system
\begin{align}
   D_A^*\Curv_A &\;=\; J(\phi), \label{eq:ym-with-source}\\
   (D_A^*D_A+\kappa^2)^\nu\,\phi &\;=\; \sum_{i=1}^N c_i\,\delta_{x_i},
   \quad \phi(x_i)=y_i,
   \label{eq:matter}
\end{align}
where \eqref{eq:matter} is the matter EL --- the active equation
carrying the data, the gauge-covariant Mat\'ern interpolation of
Paper~1 --- and \eqref{eq:ym-with-source} is the Yang--Mills equation
sourced by the matter current. The Bogomolny constraint
$\Curv_A=\pm*\Curv_A$ (where applicable) implies
$D_A^*\Curv_A=0$ via Bianchi, so on the BPS moduli the
gauge-side equation \eqref{eq:ym-with-source} reduces to
$J(\phi)\approx 0$; the matter source pushes $A$ off the moduli by a
correction of size set by $\|\phi\|$, captured below by the moduli
approximation. The regularisation
parameter $\nu>n/2$ and $\kappa^2>0$ are fixed, with the
right-hand side of \eqref{eq:matter} a sum of delta sources at the
data points with coefficients $c_i$ determined by the interpolation
conditions $\phi(x_i)=y_i$ (the covariant version of Paper~1's
$K\alpha=v$ system, \cite[Theorem~13.3]{Vasii2026Paper1}). Away from
the data points \eqref{eq:matter} reduces to the homogeneous
regularised covariant Laplace equation. The operator
$(D_A^*D_A+\kappa^2)^\nu$ is the gauge-covariant version of the
Mat\'ern regularisation \cite{Vasii2026Paper1} that makes point
evaluation a bounded functional on the natural function space: for
$n=1$ point evaluation is bounded on $H^1$ and $\nu=1$ with
$\kappa=0$ suffices (the unregularised covariant Laplacian, as in
\S\ref{ssec:ymh-flat}'s $\sph{1}$ computation); for $n\geq 2$ points
have $H^1$-capacity zero and the higher-order regularisation is
necessary for the minimiser to exist. The matter current entering
\eqref{eq:ym-with-source} is
\[
   J(\phi) \;=\; -i\,\mathrm{Im}\bigl(\phi^*\,D_A\phi\bigr)
   \;\in\; \Omega^1(M,\,i\R)
\]
in the $U(1)$ case, where the $i$ matches the $i\R$-valued algebra
on the left-hand side, and
\[
   J^a(\phi) \;=\; -\mathrm{Im}\bigl\langle\phi,\,T^a\,D_A\phi
   \bigr\rangle, \qquad J(\phi)=J^a(\phi)\,T^a\in\Omega^1(M,\mathfrak{su}(2)),
\]
in the $SU(2)$ case with $\phi$ a section of the
$\C^2$-valued associated bundle (fundamental representation, matching
the doublet-valued classifier of \S\ref{sec:S4-instanton}) and
$T^a$ the Lie algebra generators. The matter-primary picture: matter
EL on the Bogomolny moduli + the topological constraint that pins
$A$ to that moduli. The YMH-two-term picture: simultaneously solve
both equations \eqref{eq:ym-with-source}--\eqref{eq:matter}. The two
pictures agree in the moduli-approximation regime
(\S\ref{ssec:ymh-curved}'s ``Bogomolny bound and BPS moduli
approximation'' paragraph below).

On the trivial topological sector ($c_1=0$ or $c_2=0$ as
applicable), the gauge sector contributes only non-negative cost
$\int\|\Curv_A\|^2\geq 0$, with equality at $A=0$. Taking $A=0$, the
matter sector reduces to the (unregularised, for $n=1$, or
regularised, for $n\geq 2$) ordinary Dirichlet integral, minimised
by the harmonic interpolant satisfying real boundary data. The
energy-minimising section is itself real-valued (the imaginary part
of $\phi$ is unconstrained and zero at its energy minimum); hence
$\phi^*D_A\phi=\phi\,d\phi$ is real and
$J(\phi)=-i\,\mathrm{Im}(\phi\,d\phi)=0$. Equation~\eqref{eq:matter}
reduces to the regularised Laplace equation
$(\Delta+\kappa^2)^\nu\phi=\sum_i c_i\delta_{x_i}$ --- the harmonic
interpolation problem of Paper~1, exactly.

The pair $(A=0,\phi_{\R})$ is thus a critical point of the joint
functional on the trivial sector, and it is the configuration this
paper works with. We do \emph{not} claim it is the global joint
minimiser. The tempting argument --- ``any $A\neq 0$ adds strictly
positive Yang--Mills cost that the matter sector cannot offset'' ---
is circular: it compares the matter energy at $A\neq 0$ with the
matter minimum computed \emph{at} $A=0$, whereas the relevant
question is whether some $A\neq 0$ lowers the matter term by more
than its own curvature cost. That question is genuinely open, and
the following heuristic --- which we present, and then immediately
qualify --- is what makes it interesting.

\emph{An unregularised heuristic.} Take two opposite labels at
separation $d$ on a trivial $U(1)$-bundle and carry the sign
reversal by holonomy rather than by a zero of the section: let
$\phi\equiv 1$ (nowhere vanishing) and let $a$ be supported in a
strip of width $w$ along the connecting geodesic with
$\int_\gamma a=\pi$, so $|a|\approx\pi/d$. Charging the matter term
the \emph{unregularised} Dirichlet energy $\int|D_A\phi|^2$ gives
\[
   E_{\text{matter}}\sim\frac{\pi^2 w}{d},\qquad
   E_{\mathrm{YM}}\sim\frac{\pi^2}{dw},\qquad
   E_{\text{total}}\sim\frac{\pi^2}{d}\Bigl(w+\frac1w\Bigr),
\]
minimised at $w\sim 1$ with value $\sim 1/d$, apparently below the
flat cost $\sim d^{-(2\nu-n)}$ of the $A=0$ configuration.

\emph{Why this comparison is not fair.} The two sides are charged
under different functionals: the curved configuration pays an
unregularised Dirichlet energy while the flat one pays the
regularised (Mat\'ern) cost of Lemma~\ref{lem:kernel-asymptotic}. In
$n=2$ the unregularised infimum with point constraints is zero
(\S\ref{sec:torus}), so the flat side must be regularised --- and
then so must the curved side. Charging the tube the same
$\nu$-th order cost, $\int|(D_A^*D_A+\kappa^2)^{\nu/2}\phi|^2
\gtrsim (\pi/d)^{2\nu}\,dw$ (Jensen along $\gamma$ forces
$|a|\gtrsim\pi/d$ on average), gives
\[
   E_{\text{tube}} \;\gtrsim\;
   \pi^{2\nu}\,w\,d^{\,1-2\nu}+\frac{\pi^2}{dw}
   \;\geq\; 2\pi^{\nu+1}d^{-\nu},
\]
against a flat cost $\sim d^{-(2\nu-2)}$. The tube therefore
\emph{loses} at $\nu=2$ (equal powers, and it forfeits the
logarithmic gain of the log branch of
Lemma~\ref{lem:kernel-asymptotic}), and wins only for $\nu>2$, where
$d^{-\nu}$ is the weaker divergence. The crossover sits exactly at
the threshold $\nu=2$ relevant to $n=2$. The apparent
``$1/d$ beats $1/d^2$'' of the heuristic is thus an artefact of the
mismatched functionals, not a mechanism.

We record this honestly: under the paper's own operator
regularisation, a uniform strip does not beat the flat configuration
at $\nu=2$. An affirmative answer to
Question~\ref{q:joint-minimiser} would need either a different
regularisation --- regularising the \emph{data} (discs of radius
$\rho$, $\nu=1$) rather than the operator is the natural candidate,
since both configurations are then charged alike and the flat cost
diverges only as $d\to 2\rho$ --- or a configuration cleverer than a
uniform strip. Note also that the estimate degenerates for
$n\geq 3$: a tube of cross-section $w^{n-1}$ gives
$E_{\mathrm{YM}}\sim w^{n-3}/d$, so thin flux tubes are free as
$w\to 0$ and the connection needs regularisation too.

\begin{question}[Does curvature reduce the cost of nearby opposite labels?]
\label{q:joint-minimiser}
Fix a topological sector (including the trivial one) on a closed
$(M,g)$, fix regularisation parameters $\nu>n/2$, $\kappa>0$, and
let
\[
   E^*(d)\;=\;\inf_{(A,\phi)}\Bigl\{\int_M\|\Curv_A\|^2
   +\bigl\|\phi\bigr\|^2_{\nu,A} \;:\; \phi(x_\pm)=y_\pm\Bigr\},
   \qquad
   \|\phi\|^2_{\nu,A}:=\int_M\bigl|(D_A^*D_A+\kappa^2)^{\nu/2}\phi\bigr|^2,
\]
be the \emph{joint} infimum over all connections in the sector ---
not merely over the Bogomolny moduli --- with two opposite-label
points at geodesic distance $d$, and let $E_{\mathrm{flat}}(d)$ be
the corresponding infimum at fixed $A=0$, \emph{with the same
regularised matter norm}. (Both sides must be charged under the same
functional: with the unregularised Dirichlet energy the flat
infimum vanishes identically in $n\geq 2$ by capacity, and the ratio
below is meaningless.) Is
$E^*(d)/E_{\mathrm{flat}}(d)\to 0$ as $d\to 0$, and do the
minimising connections have curvature concentrating in a
neighbourhood of the connecting geodesic?

As the estimate above shows, the answer is \emph{not} affirmative
for the uniform-strip configuration at $\nu=2$ in $n=2$; whether it
is affirmative for some other configuration, for $\nu>2$, or under a
data regularisation in place of the operator regularisation, is
open. An affirmative answer would give the framework's attention
reading a dynamical content it does not currently have: curvature
would not merely \emph{encode} feature interaction
(Definition~\ref{thm:attention}) but would \emph{respond} to the
data, concentrating precisely where opposite labels crowd and the
Higgs energy is largest. We emphasise that the present paper does
not establish this: throughout, the connection is either flat
(tier~1) or pinned to the Bogomolny moduli of an imposed sector
(tier~2), and in both regimes the curvature is data-decoupled by
construction (Remark~\ref{rem:proximity-quantity}).
Settling the question requires the joint problem above, and in
$n\geq 3$ a regularisation of the connection to exclude thin flux
tubes.
\end{question}

\paragraph{Bogomolny bound and BPS moduli approximation.} On the
non-trivial topological sector with $SU(2)$ structure group on
$\sph{4}$ (and the analogue for $U(1)$ on $\sph{2}$), decomposing
$\Curv_A=\Curv_A^++\Curv_A^-$ into self-dual and anti-self-dual
parts (with $*\Curv_A^\pm=\pm\Curv_A^\pm$) gives the
Bogomolny inequality \cite{Bogomolny1976,JaffeTaubes}
\[
   \int_M\|\Curv_A\|^2\,\dvol_g
   \;=\; \tfrac{1}{2}\int_M\|\Curv_A\mp *\Curv_A\|^2\,\dvol_g
   \;\pm\;\int_M\tr(\Curv_A\wedge\Curv_A),
\]
where the topological term is $\pm 8\pi^2 c_2$ and the deviation term
is non-negative. The Yang--Mills energy is thus bounded below by
$8\pi^2|c_2|$, with equality if and only if $\Curv_A=\pm*\Curv_A$
(self-dual or anti-self-dual); the saturating instanton solutions
on $\sph{4}$ were exhibited by Belavin--Polyakov--Schwarz--Tyupkin
\cite{BPST1975}. The (anti-)self-duality equation is
first-order; its solutions form a finite-dimensional moduli space
parametrised by centre, scale, and frame (for $\sph{4}$ with
$|c_2|=1$, the BPST family of \cite{BPST1975}; the general
moduli structure is the content of ADHM \cite{ADHM1978,Atiyah1979}).

In the pure Yang--Mills theory ($\phi=0$, no matter source), the
joint minimiser of \eqref{eq:ym-with-source} is exactly self-dual:
$D_A^*\Curv_A=0$ is automatic on a self-dual connection by the
Bianchi identity. With $J(\phi)\neq 0$ the right-hand side of
\eqref{eq:ym-with-source} is non-zero, and the true joint minimiser
is pushed off the self-dual locus by a correction of size set by
$\|\phi\|$. The standard adiabatic treatment is the \emph{moduli
approximation} of Manton \cite{Manton1982}: in the regime of small
Higgs amplitude (or large
gauge scale), the minimiser lies in a neighbourhood of the BPS
moduli, and the matter equation \eqref{eq:matter} acts on the moduli
parameters (centre, scale, frame) as a finite-dimensional
optimisation, picking a representative from the BPS family. The
worked $\sph{4}$ instance --- the moduli approximation made
concrete by the ADHM constraint equations --- is
Theorem~\ref{thm:instanton} of \S\ref{sec:S4-instanton}; the
$\sph{2}$ analogue with abelian Higgs is the monopole of
\S\ref{sec:S2-monopole}. The exact joint minimiser away from the
adiabatic regime is harder and is not treated here.

\paragraph{Proximity in the curved case.} The tier-1 proximity
divergence of \S\ref{ssec:ymh-flat} extends to tier 2 in a
dimension-dependent way: the matter-sector divergence persists in
every dimension (RKHS norm of the minimum-energy interpolant blows
up as $d\to 0$), but the curvature response is
structurally different in $n=4$ (BPST scale tracks data, peak
$\|\Curv\|\sim 1/d^2$) and $n=2$ (uniform monopole, no localising
scale at $\lambda=0$). The precise statement is
Conjecture~\ref{conj:proximity-curvature}, located at the end of
\S\ref{sec:S4-instanton} where the $n=4$ side is made concrete by
$\lambda\sim d$ from Remark~\ref{rem:lambda-status}; the $n=2$ side is
the uniform monopole of \S\ref{sec:S2-monopole}.

% ---------------------------------------------------------------------
\section{Reduction to Paper 1}
\label{sec:reduction-paper1}

\begin{proposition}[Reduction to Paper 1]\label{prop:reduction-paper1}
Let $M$ be contractible. Take the framework's structure group in
its real form $G=\R^*$ on a real line bundle (equivalently, the real
form of the complexified $U(1)$-setup of \S\ref{sec:yang-mills-higgs}
restricted to real-valued sections; on a contractible base
$H^2(M,\Z)=0$ trivialises any $U(1)$-bundle, and a flat $U(1)$
connection on the trivial bundle is gauge-equivalent to the real
form). Then the variational classification problem of
\S\ref{sec:yang-mills-higgs} reduces exactly to the harmonic
interpolation problem of \cite[Theorem~13.3]{Vasii2026Paper1}. The
unique minimum-matter-energy solution is the harmonic interpolant
\[
  f(x)\;=\;\sum_i \alpha_i\, y_i\, K(x,x_i),
  \qquad K\alpha = v,\quad K_{ij} = G_{\Delta_g}(x_i,x_j),
\]
where $K$ is the Green's function of the regularised Laplacian
$(\Delta_g+\kappa^2)^\nu$ on $(M,g)$, with $\nu>n/2+1$ as in
\cite[Theorem~13.3]{Vasii2026Paper1} (Paper~1 assumes the stronger
$C^1$ threshold; the present paper's standing hypothesis is
$\nu>n/2$, cf.\ Remark~\ref{rem:ymh-attainment}).
\end{proposition}

\begin{proof}
$M$ contractible $\Rightarrow$ $\pi_1(M)$ trivial, so every flat
connection has trivial holonomy and the functor
$\ClassFunc\colon\PathGpd(M)\to\BGroup{\Z_2}$ of
Theorem~\ref{thm:main} is the trivial functor. The Higgs section
therefore carries the entire classification content. With $G=\R^*$
(real form) abelian and the connection flat (since the trivial
bundle admits the trivial flat connection), the curvature vanishes
identically and the variational
problem reduces to bare matter-sector minimisation of the Dirichlet
energy $\int_M\|d\phi\|^2$ subject to the data conditions. The
minimum-Dirichlet-energy section satisfying
$f(x_i)=y_i\cdot r$ is the harmonic interpolant
$f(x)=\sum_i\alpha_i\,y_i K(x,x_i)$ by
\cite[Theorem~13.3]{Vasii2026Paper1}.
\end{proof}

The variational bundle of \S\ref{sec:yang-mills-higgs} thus
\emph{contains} Paper~1's as its flat abelian case: with the
connection flat and the structure group abelian, $D_A\to d$ and the
covariantly harmonic section reduces to Paper~1's horizontal section
\cite[Remark~5.3]{Vasii2026Paper1}, so the variational bundle is
Paper~1's $O(2)/\R^*$ classifier bundle. The present framework
extends Paper~1 by adjoining the curved, non-abelian regime absent
there; the functor $\ClassFunc$ of \S\ref{sec:functor} is the
label-level datum both share.

\begin{remark}[Local Paper~1 in charts vs.\ global framework]
\label{rem:local-vs-global}
A natural question, given Proposition~\ref{prop:reduction-paper1}:
since every smooth manifold is locally Euclidean and every bundle
locally trivial, why not solve the classification problem chart by
chart by applying Paper~1 on each $U_i$ in a trivialising cover
$\{U_i\}_{i\in I}$ of $L_{\xi^*}$, and patch the local harmonic
interpolants together? This question is worth answering precisely,
because the answer locates where the framework's contribution
actually sits.

\paragraph{The naive patchwork fails.} A trivialising cover gives
each $U_i$ a chart on which $L_{\xi^*}$ is trivial; Paper~1 on $U_i$
produces a local harmonic interpolant
$f_i\colon U_i\to\R$ matching the data points in $U_i$. The
collection $\{f_i\}$ is not a classifier on $M$ unless the local
solutions agree on overlaps in the bundle-structured way:
$f_j=g_{ij}\,f_i$ on $U_i\cap U_j$, where
$g_{ij}\in\{\pm 1\}$ is the transition cocycle of the bundle.
Paper~1 applied locally has no knowledge of the cocycle, so the
naive patchwork generically violates the gluing condition and
produces either a discontinuous section or a continuous one that
does not represent a section of $L_{\xi^*}$.

\paragraph{The honest local approach.} To make the chart-by-chart
method well-defined, one must (a) supply the cocycle $\{g_{ij}\}$
externally, and (b) impose the compatibility $f_j=g_{ij}f_i$ on
overlaps as a boundary condition coupling the local problems. The
local problems then become coupled by their boundary data, and
solving them jointly is no longer Paper~1: it is a global elliptic
problem on $M$ with bundle-twisted boundary compatibility, written
in chart-local variables. Under the identification $\phi|_{U_i}=f_i$
(with the local trivialisation of $L_{\xi^*}$), the global problem
of \S\ref{sec:yang-mills-higgs} reads in chart-local variables as
exactly the joint Paper~1-in-charts problem with cocycle
compatibility. The two approaches are computing the same thing.

\paragraph{Where the framework adds content.} Three places where
the global approach goes beyond chart-local Paper~1:

\begin{enumerate}[label=(\arabic*),leftmargin=2em]
\item \emph{Bundle selection.} The variational selector of
Definition~\ref{def:variational-selector} picks the bundle class
$\xi^*$ by comparing global YMH minima across $H^1(M,\Z_2)$. The
chart-local approach requires the bundle (and its cocycle) as input.
The framework discovers the bundle from the data; the patchwork
needs the user to know in advance which bundle to work on. On
$\sph{1}$ with two opposite labels, this is the difference between
the framework's prediction $4/d$ on the M\"obius bundle and the
patchwork's silence on which bundle to use; on $T^2$ with XOR, this
is the difference between the framework's selection of
$(\sigma_1,\sigma_2)$ and the patchwork's inability to evaluate the
four candidate bundles.

\item \emph{The connection is variational.} The chart-local Paper~1
problem assumes a flat connection in each chart (the connection
local form $a_i$ is absorbed into gauge). The global problem in
tier 2 has a variationally selected non-flat connection (the Dirac
monopole on $\sph{2}$, the BPST instanton on $\sph{4}$) whose
small-matter solutions emerge from the joint matter--Yang--Mills
minimisation (Remark~\ref{rem:matter-vs-ymh}). No chart of size
larger than the instanton scale $\lambda$ can flatten the
$\sph{4}$ BPST connection by gauge: the connection's curvature is
the global topological content $c_2=1$, which no local chart sees.

\item \emph{The curvature/attention dictionary is global.} The
§\ref{sec:attention} dictionary identifies $\Curv_A$ as the
antisymmetric feature-interaction form and the holonomy-priority
function as the geometric content of softmax attention. These are
global statements about a global object $\Curv_A\in\Omega^2(M,\adP)$.
The chart-local Paper~1 picture works with flat connections in each
chart and sees no curvature; the multi-head attention content of
\S\ref{ssec:abelian-nonabelian} requires the non-trivial commutator
$[\Connection\wedge\Connection]$, which is invisible to chart-local
Paper~1.
\end{enumerate}

\paragraph{The structural reading.} The framework of this paper is
to Paper~1 what de Rham cohomology is to local exact forms: locally
the same theory, globally an obstruction-theoretic enlargement. The
obstruction is the bundle class $\xi^*\in H^1(M,\Z_2)$ in tier 1
and the Chern class in tier 2; both vanish iff the framework
reduces to Paper~1 (Proposition~\ref{prop:reduction-paper1}). Where
the obstruction is non-trivial, the local-charts approach can still
compute the answer once supplied with the bundle as input, but it
loses the framework's structural content: the bundle is now
discovered, not assumed; the connection is variational, not chosen;
the curvature is a meaningful global object, not gauged away
chart-by-chart. The local picture is the framework's tier-1
contractible case, viewed pointwise.
\end{remark}

% =====================================================================
\part*{Part 2. Concrete Instances}
\addcontentsline{toc}{part}{Part 2. Concrete Instances}

% ---------------------------------------------------------------------
\section{Labelled points on $\sph{1}$: the M\"obius classifier}
\label{sec:S1-mobius}

This section is the framework's smallest non-trivial worked example.
The base $M=\sph{1}$ is one-dimensional, so curvature carries no
content (there are no 2-planes in $T_x\sph{1}$) and the entire
classification work is done by the matter sector on a topologically
non-trivial flat bundle. The example is pedagogically valuable for
exactly this reason: it isolates the bundle-selection mechanism of
\S\ref{sec:yang-mills-higgs} from the curvature content of
\S\ref{sec:attention}, and shows that tier-1 obstructions alone are
already enough to break the harmonic-interpolation paradigm of
Paper~1's flat-trivial-bundle setting. The closed-form computation of
the matter energy is also carried out in \cite[\S 9.2]{Vasii2026Paper1};
we revisit the example here to expose the variational selector
(which chooses between the trivial and M\"obius classes by
comparing constrained matter minima) and the Wilson observable as
the prediction mechanism on the bundle with no sign-function
classifier.

\paragraph{Setup.} Take $M=\sph{1}$ with circumference $L$, and a
labelled dataset $D=\{(x_+,+1),(x_-,-1)\}$ with the two points at
geodesic distance $d$ along the short arc (so the complementary arc
has length $L-d$). We parametrise $\sph{1}$ by arc length $s\in[0,L)$
with $x_-$ at $s=0$ and $x_+$ at $s=d$. The relevant cohomology is
$H^1(\sph{1},\Z_2)=\Z_2$, so there are exactly two bundle classes:
$\xi=0$ (trivial line bundle $L_0=\sph{1}\times\R$) and $\xi=\sigma$
(M\"obius line bundle $L_\sigma$). Both are flat: $\sph{1}$ admits
flat connections in every class, so the topological constraint of
\S\ref{ssec:ymh-functional} is $\Curv_A=0$ and the entire energy
is the matter sector.

\paragraph{The variational selector.} The matter-sector minima on
the two classes are computed in \S\ref{ssec:ymh-flat}:
\[
   E_{\min}(0) \;=\; \frac{4}{d}+\frac{4}{L-d},
   \qquad
   E_{\min}(\sigma) \;=\; \frac{4}{d}.
\]
The trivial bundle requires \emph{two} sign changes (one on each
arc, since periodicity forces return-to-start); the M\"obius bundle
absorbs one sign change into its anti-periodic boundary condition
on the double cover and requires only \emph{one} sign change across
the short arc. The selector picks the class with smaller constrained
matter energy:
\[
   \xi^* \;=\; \{\sigma\},
   \qquad
   E_{\min}(0) - E_{\min}(\sigma) \;=\; \frac{4}{L-d} \;>\; 0
   \quad\text{for all } d\in(0,L),
\]
so $[c_y]=\sigma$ for opposite labels at every separation. For same
labels the analogous computation gives $E_{\min}(0)=0$
(constant section) versus $E_{\min}(\sigma)>0$, so $\xi^*=\{0\}$
and the trivial bundle wins, recovering Paper~1 in the
$H^1(M,\Z_2)=0$-style behaviour even on a base with non-trivial
$H^1$.

\paragraph{The flat connection on the M\"obius bundle.} The
M\"obius bundle $L_\sigma$ has a unique flat connection up to gauge,
whose holonomy around the generator $[\ell]\in\pi_1(\sph{1},x_+)=\Z$
is the non-trivial element $-1\in\Z_2$. In the complexified line
bundle $L_\sigma^\C=L_\sigma\otimes_\R\C$ (the $U(1)$-bundle of
\S\ref{ssec:ymh-functional}, topologically trivial as $U(1)$-bundle
because $H^2(\sph{1},\Z)=0$), this flat connection has explicit
local form
\[
   \Connection \;=\; -\tfrac{1}{2}\,d\alpha,
   \qquad \alpha=2\pi s/L\in[0,2\pi),
\]
with curvature $\Curv_A=d\Connection=0$ (flat, as required by the
tier-1 constraint). The holonomy of $\Connection$ around the loop
$\ell$ generating $\pi_1(\sph{1})$ is
\[
   \Hol_\ell(\Connection) \;=\; \exp\!\left(-\int_\ell \Connection\right)
   \;=\; \exp(-i\pi) \;=\; -1 \;\in\;\Z_2\subset U(1).
\]
This is the topological constraint that $\xi=\sigma$ imposes: a
flat connection with non-trivial $\Z_2$-holonomy.

\paragraph{The covariantly harmonic section.} The matter EL on
$L_\sigma$ is, in the trivialisation where $\Connection$ is the
$1$-form above, the antiperiodic Laplace equation $\tilde\phi''=0$
on the double cover $\widetilde{\sph{1}}$ (circumference $2L$) with
$\tilde\phi(s+L)=-\tilde\phi(s)$ and the data conditions
$\tilde\phi(0)=-1$, $\tilde\phi(d)=+1$. The minimiser is piecewise
linear:
\[
   \tilde\phi(s) \;=\;
   \begin{cases}
      -1 + \dfrac{2s}{d} & 0 \leq s \leq d, \\[6pt]
      +1 & d \leq s \leq L,
   \end{cases}
\]
extended antiperiodically. The gradient is $2/d$ on the short arc
and $0$ on the long arc, giving covariant Dirichlet energy
$\int_0^L(\tilde\phi')^2\,ds = (2/d)^2\cdot d = 4/d$, matching
$E_{\min}(\sigma)$ above. The section is parallel ($D_A\phi=0$,
hence covariantly constant) on the long arc, where it costs nothing;
all the matter energy is concentrated on the short arc, where the
data forces a covariant gradient.

\paragraph{The forced zero and the decision ``boundary''.} The
section $\tilde\phi$ has a single zero at $s=d/2$ on the short arc;
this descends to a single point on $\sph{1}$, the midpoint of the
short arc between $x_-$ and $x_+$. This is the lifted decision
boundary: $\Gamma=\{s=d/2\}$ in the local trivialisation. However,
$\Gamma$ is not a decision boundary on $\sph{1}$ in the
sign-function sense of Theorem~\ref{thm:main}(ii): since
$[c_y]=\sigma\neq 0$, no continuous sign function on $\sph{1}$
realises the classifier. The ``boundary'' $\Gamma$ exists in the
trivialised picture but is an artefact of the local chart; on the
M\"obius bundle, the classifier is the section $\phi$, not a
function on $\sph{1}$, and there is no canonical assignment of
``positive side'' or ``negative side'' away from the data
(Remark~\ref{rem:boundary-structure}).

\paragraph{Prediction by holonomy parity.} The Wilson observable of
\S\ref{ssec:open-paths} prescribes the prediction at $x\in\sph{1}$
as the sign of $\phi(\tilde x)$ for a lift $\tilde x$ of $x$ via the
chosen reference path $\gamma\colon x_+\to x$. The path-dependence
reduces, on a flat bundle, to dependence on the homotopy class of
$\gamma$ in $\pi_1$; on $\sph{1}$ the class is the winding number
$n(\gamma)\in\Z$, and the $\Z_2$ functor only sees its parity:
\[
   \mathrm{prediction}(x;\gamma) \;=\;
   y_+\cdot\mathrm{Hol}_\gamma(\Connection)
   \;=\; (+1)\cdot(-1)^{n(\gamma)}.
\]
Concretely: a path that doesn't wind around $\sph{1}$ ($n=0$) gives
prediction $+1$ if $x$ is on the same side as $x_+$ (the long arc),
$-1$ if on the other side (the short arc between $x_-$ and the
midpoint $s=d/2$, where $\tilde\phi<0$). A path that winds once
($n=1$) reverses this. The framework reports the prediction along
with the chosen path; for the canonical (shortest-geodesic) path
the result is unambiguous and matches $\mathrm{sign}\,\tilde\phi$.

\begin{remark}[Adams-ladder placement of the worked
examples]\label{rem:adams-rungs}
The worked examples of this paper sit at the lowest three rungs of
the Adams ladder (introduction, Property P3): $\sph{1}$ M\"obius and
$T^2$ XOR at rung 0 (discrete $\Z_2$ structure group, $\Curv_A=0$);
$\sph{2}$ Dirac monopole at rung 1 ($G=U(1)$, complex, abelian
curvature, single-head attention); $\sph{4}$ instanton at rung 2
($G=SU(2)$, quaternionic, non-abelian curvature, multi-head
attention). The structural theorems that would justify this as a
\emph{ladder} --- rung monotonicity, minimal-architecture statements
at each rung, the rung-3 octonionic case --- belong to a forthcoming
paper; here we only use the rung assignments as labels for the
example sequence. Subsequent worked-example sections refer back to
this placement rather than restating it locally.
\end{remark}

\begin{remark}[Why \S\ref{sec:attention} is silent on $\sph{1}$]
\label{rem:S1-no-attention}
The curvature 2-form $\Curv_A\in\Omega^2(M,\adP)$ is evaluated on
2-planes in $T_xM$ (Theorem~\ref{thm:plaquette}). For $M=\sph{1}$,
$\dim T_x\sph{1}=1$ and the Grassmannian $\Gr_2(T_x\sph{1})$ is
empty: there are no 2-planes to evaluate $\Curv_A$ on. The
holonomy-priority function $\pi_A$ of
Definition~\ref{def:holonomy-priority} has empty domain. The
\S\ref{sec:attention} attention dictionary is therefore vacuous on
$\sph{1}$, consistent with the fact that the tier-1 example is
solved entirely by the matter sector via 1-dimensional parallel
transport. Attention as a geometric mechanism first appears on
$\sph{2}$ (\S\ref{sec:S2-monopole}), where $\dim M=2$ admits exactly
one 2-plane at each point and curvature is a scalar density.
\end{remark}

% ---------------------------------------------------------------------
\section{Labelled points on $\sph{2}$: the monopole classifier}
\label{sec:S2-monopole}

On $\sph{2}$ the framework first encounters forced curvature. The
base has trivial $H^1(\sph{2},\Z_2)=0$, so the variational selector
of \S\ref{sec:yang-mills-higgs} sees a single class and does not
choose between bundles; the tier-1 obstruction vanishes. The new
content is at $H^2(\sph{2},\Z)=\Z$, outside the selector's range
(see Remark~\ref{rem:problem-A-B}). We study the $c_1=1$ Hopf line
bundle as a Problem~B configuration in the sense of
Remark~\ref{rem:problem-A-B}: the bundle is imposed --- either by
continuous boundary data winding once around an equatorial loop, or
externally as the relevant non-trivial topological sector --- and
the framework determines the structure of the optimal classifier
\emph{given} the bundle. On finite labelled data alone, Paper~1's
Theorem~8.2 \cite{Vasii2026Paper1} guarantees a flat $O(2)$
classifier with finite matter energy on the trivial bundle, which
would beat the Hopf bundle's Bogomolny floor $\geq\pi$; the
monopole structure below is the structural content of the chosen
$c_1=1$ sector, not the unconditional Problem~A optimum.

Given this scoping, the variational problem on the Hopf bundle has,
by §\ref{ssec:ymh-curved}, the Bogomolny moduli of $U(1)$-bundles
with $c_1=1$ on $\sph{2}$ as its optimal connections, which
(because $U(1)$ is abelian) is a \emph{single point}: the Dirac
monopole. The matter sector then runs in this monopole background.

The Dirac monopole on $\sph{2}$ is also treated in
\cite[\S 9.3]{Vasii2026Paper1}; we revisit it here under the
present framework's variational and obstruction-theoretic reading,
reusing Paper~1's explicit computations where relevant.

\paragraph{Setup.} Take $M=\sph{2}$ with the round metric of unit
radius. Following the tier-2 framework of \S\ref{ssec:two-tiers},
fix the structure group at rung~1 of the Adams ladder: $G=U(1)$, the
unit circle in $\C$, with non-trivial sector $c_1=1$. The associated
complex line bundle $\Lcx=L\otimes\C$ is the tautological line bundle
of $\mathbb{CP}^1=\sph{2}$, and the principal $U(1)$-bundle $P\to\sph{2}$
with $c_1(P)=1$ is the Hopf bundle, with total space $\sph{3}$.
Place two labelled points $x_+=$ north pole, $x_-=$ south pole, with
labels $\pm 1$.

\paragraph{Why curvature is forced.} Chern--Weil applied to the
$U(1)$-bundle with $c_1=1$ gives
$\int_{\sph{2}}\Curv_A = 2\pi i\,c_1(P) = 2\pi i \neq 0$,
ruling out $\Curv_A\equiv 0$. The Yang--Mills energy is bounded
below in the class by the Bogomolny floor
$\int_{\sph{2}}\|\Curv_A\|^2\geq (2\pi)^2/\mathrm{Vol}(\sph{2})=\pi$
(Cauchy--Schwarz applied to the Chern--Weil integral), with equality
exactly when $\Curv_A$ is a constant multiple of the volume form:
$\Curv_A=\tfrac{1}{2}\,\omega_{\sph{2}}$. This is the unique
Bogomolny-saturating connection in $c_1=1$ up to gauge --- the Dirac
monopole --- and the Bogomolny moduli is the single point
$\{\Connection_{\mathrm{Dirac}}\}$.

\paragraph{The Dirac monopole connection.} In spherical coordinates
$(\theta,\psi)\in[0,\pi]\times[0,2\pi)$ with the north/south chart
covering of $\sph{2}$,
\[
   \Connection^{N} \;=\; \frac{i}{2}(1-\cos\theta)\,d\psi,
   \qquad
   \Connection^{S} \;=\; -\frac{i}{2}(1+\cos\theta)\,d\psi,
\]
with the transition function $g_{NS}=e^{-i\psi}$ relating them on
the equatorial overlap (winding number $1$, registering
$c_1=1$). The curvature is the same in both charts,
\[
   \Curv_A \;=\; d\Connection^{N} \;=\; \frac{i}{2}\sin\theta\,
   d\theta\wedge d\psi \;=\; \frac{i}{2}\,\omega_{\sph{2}},
\]
uniform over $\sph{2}$. The Yang--Mills equation $D_A^*\Curv_A=0$ is
satisfied trivially (the Hodge dual $*\Curv_A=i/2$ is constant). The
Dirac monopole is the unique $U(1)$-Yang--Mills connection with
$c_1=1$, the BPS minimiser of the Bogomolny inequality, and the only
admissible connection under the tier-2 constraint.

\paragraph{The forced zero of the section.} The associated complex
line bundle $\Lcx$ has $c_1=1$; equivalently, the Euler class of its
real rank-$2$ realification equals $1\in H^2(\sph{2},\Z)$. By the
Poincar\'e--Hopf theorem, every smooth section
$\phi\in\Gamma(\Lcx)$ has zeros whose total index is $1$. The
generic section has exactly one zero. In particular, no nowhere-zero
classifier exists --- the geometric obstruction is structural, not
data-induced. We call this zero the \emph{forced zero} of the
classifier; it plays the role of the decision boundary, with two
crucial differences from the Paper~1 sign-function picture: it is a
single \emph{point} (zero-dimensional, not codimension-1), and its
existence is forced by topology before any optimisation.

\paragraph{Symmetry pins the forced zero to the equator.} For two
labelled points $x_\pm$ at the poles, the labelled configuration is
invariant under $\Z_2\times\Z_2$:
\begin{itemize}[leftmargin=2em]
\item the equatorial rotation $\Z_2^{\mathrm{rot}}\subset SO(2)$
that swaps the two halves of the equator (any rotation by $\pi$
about a great circle through both poles); and
\item the $\Z_2^{\mathrm{label}}$ swap $(x_+,+1)\leftrightarrow
(x_-,-1)$ combined with reflection through the equator. As on
$\sph{4}$ below, this reflection is orientation-reversing and so
sends the $c_1=1$ sector to $c_1=-1$; it is a symmetry of the fixed
sector only when composed with complex conjugation of $L$ (the
antilinear map $\phi\mapsto\bar\phi$), which restores $c_1=1$ and
simultaneously implements the label flip. We understand the
$\Z_2^{\mathrm{label}}$ factor to include this conjugation.
\end{itemize}
Both symmetries fix the equator (as a set). The matter
sector $E_{\mathrm{matter}}(\Connection_{\mathrm{Dirac}},\phi)$ is
invariant under the labelled configuration's symmetries. Since the
regularised matter functional is strictly convex in $\phi$ and the
data conditions are affine, the constrained minimiser $\phi^*$ is
\emph{unique}; being unique, it is fixed by every symmetry of the
problem, hence $\Z_2\times\Z_2$-equivariant. (This order of
reasoning matters: uniqueness is what licenses the equivariance, not
the other way round.) The forced zero, as a topologically required
feature of $\phi^*$, therefore lies on the fixed-point set of the
symmetry --- the equator $\{\theta=\pi/2\}$. For the two-point
problem this can also be seen directly from the capacitance solve:
$|c_+|=|c_-|$ by the label symmetry, and the monopole kernel
$|K(\theta)|$ is monotone in the polar angle, so the modulus
$|\phi^*|$ vanishes exactly on $\theta=\pi/2$.

\begin{remark}[The azimuth is a frame convention, not a degeneracy]
\label{rem:S2-azimuth}
The azimuthal position $\psi_0$ of the forced zero on the equator is
\emph{not} an undetermined modulus of equally optimal classifiers,
and earlier drafts of this work misdescribed it as one. The
minimiser is unique, as just noted. What is conventional is the
\emph{frame}: the value of $\phi$ at a point is read in the local
trivialisation obtained by parallel-transporting a chosen frame at
$x_+$ along a chosen reference path
(\S\ref{ssec:open-paths}). Changing the meridian used for that
transport by $\Delta\psi$ rotates the frame at the far endpoint by
the lune holonomy $e^{-i\Delta\psi}$ and rotates the apparent
position of the zero correspondingly. This is the path-dependence of
\S\ref{ssec:open-paths} in the curved regime, not the genuine
tie between distinct optimal bundles of
Remark~\ref{rem:honest-degeneracy}: there is one classifier here,
described in a family of frames.
\end{remark}

\paragraph{The covariantly harmonic section.} The matter EL on
$\Lcx$ in the Dirac background is
\[
   (D_A^*D_A + \kappa^2)^\nu\,\phi \;=\; \sum_{i\in\{+,-\}} c_i\,
   \delta_{x_i}, \qquad \phi(x_\pm)=y_\pm,
\]
with $\nu>n/2=1$ for the Mat\'ern regularisation (Paper~1's
\cite[Theorem~13.3]{Vasii2026Paper1} adapted to the gauge-covariant
setting; cf.\ Remark~\ref{rem:ymh-attainment}). The unregularised
$D_A^*D_A$ on $\sph{2}$ with the monopole connection is the
covariant Laplacian whose spectrum is computed by the monopole
harmonics (Wu--Yang spinor harmonics), with lowest eigenvalue
$\lambda_0=|c_1|/2=1/2$ rather than zero --- the bundle is twisted,
and there is no constant covariantly harmonic section. The
minimum-matter-energy section is the
covariantly harmonic interpolant in the monopole background, taking
prescribed values $\pm 1$ at the poles, with single forced zero on
the equator at the moduli-determined position.

\paragraph{Holonomy and the Berry phase.} For a closed loop
$\partial R$ bounding a region $R\subset\sph{2}$ (with the chart
orientation), the Dirac connection's holonomy is
\[
   \Hol_{\partial R}(\Connection_{\mathrm{Dirac}})
   \;=\;\exp\!\left(-\int_R \Curv_A\right)
   \;=\;\exp\!\left(-\frac{i}{2}\,\Omega(R)\right),
\]
where $\Omega(R)=\int_R\omega_{\sph{2}}$ is the solid angle subtended
by $R$. The holonomy is the Berry phase --- the geometric phase
acquired by a parallel-transported state of a spin-$1/2$ system
around the loop --- and it serves as the Wilson observable of the
framework on $\sph{2}$. The prediction at $x\in\sph{2}$ along a
reference path $\gamma\colon x_+\to x$ is
$\mathrm{sign}\,\mathrm{Re}(\Hol_\gamma\cdot y_+)$, which depends on
the path's homotopy class only modulo the contribution from the
curvature flux through any disc bounding the path.

\paragraph{Path consistency: a numerical-experiment proposal.} The
framework's prediction --- the Dirac monopole is the optimal gauge
background, with curvature $\tfrac{i}{2}\,\omega_{\sph{2}}$ uniform
across $\sph{2}$ --- can be checked by a consistency test on
path-dependent holonomies. Take two paths $\gamma_1, \gamma_2\colon
x_+\to x$ between any pair of points on $\sph{2}$. The two paths
bound a 2-surface $\Sigma$ on $\sph{2}$ (with sign depending on
orientation). The holonomy discrepancy is
\[
   \Hol_{\gamma_2}(\Connection_{\mathrm{Dirac}})
   \,\Hol_{\gamma_1}(\Connection_{\mathrm{Dirac}})^{-1}
   \;=\;\exp\!\left(-\int_\Sigma\Curv_A\right)
   \;=\;\exp\!\left(-\frac{i}{2}\,\Omega(\Sigma)\right),
\]
the area-of-enclosed-region formula. The framework's prediction is
testable: compute both holonomies independently by integrating the
parallel-transport ODE along $\gamma_1$ and $\gamma_2$; verify the
quotient matches $\exp(-i\Omega/2)$ for the enclosed solid angle.
This is a consistency check on the framework's selected connection,
not a transformer experiment. We do not implement it here.

\begin{remark}[Uniform curvature as maximum-entropy attention]
\label{rem:S2-uniform-attention}
The Dirac monopole's curvature is uniform across $\sph{2}$. Through
the attention dictionary of \S\ref{ssec:transformer-dictionary}, the
holonomy-priority function $\pi_A(x)=\|\Curv_A(x)\|^2$ is constant
on $\sph{2}$: every 2-plane (and at $\dim M=2$ there is only one
2-plane per tangent space, the entire $T_x\sph{2}$) carries the
same priority. This is the geometric content of
``maximum-entropy attention given $c_1=1$'': the framework's
optimal connection distributes the topologically-forced curvature
flux as uniformly as possible, with no preferred direction. Local
concentration of attention --- non-uniform $\pi_A$ --- would
require either symmetry breaking by additional data points or a
confining potential (Remark~\ref{rem:higgs-potential}). The
two-point monopole is therefore the simplest non-trivial gauge
classifier and also the most featureless one; the next case
($\sph{4}$ instanton, \S\ref{sec:S4-instanton}) exhibits curvature
localised into a single core, the structurally richest example.
\end{remark}

\paragraph{Three labelled points: symmetry breaking.} With a third
labelled point $x_3$ breaking the $\Z_2\times\Z_2$ symmetry of the
two-point configuration, the equatorial moduli collapses: the
forced zero moves to a specific location determined by the matter
EL with three sources. The schematic prediction is that the zero
shifts toward the unpaired class --- the side of the equator with
only one labelled point of its sign --- with displacement scale set
by the matter current $J(\phi)$ generated by the symmetry-breaking
data. The full $N=3$ analysis is straightforward in principle (the
matter EL is the inhomogeneous monopole-Laplace equation with three
point sources) but the closed-form solution involves
Wu--Yang spherical harmonics in the monopole background and is not
worked here. The general $N$-point case reduces to inverting the
covariant Green's kernel against the data sources --- the abelian
analogue of the ADHM constraint of Theorem~\ref{thm:instanton}.

% ---------------------------------------------------------------------
\section{Labelled points on $\sph{4}$: the instanton classifier}
\label{sec:S4-instanton}

On $\sph{4}$ the framework reaches its richest single example. The
base has $H^1(\sph{4},\Z_2)=0$ and $H^2(\sph{4},\Z)=0$, so the
variational selector of \S\ref{sec:yang-mills-higgs} sees a single
class and does not choose between bundles. The new content is at
$H^4(\sph{4},\Z)=\Z$, where the second Chern class $c_2$ of an
$SU(2)$-bundle lives --- outside the selector's range (see
Remark~\ref{rem:problem-A-B}). As on $\sph{2}$, we study the
$c_2=1$ bundle as a Problem~B configuration in the sense of
Remark~\ref{rem:problem-A-B}: the bundle is imposed (by a chosen
topological sector or by continuous boundary data winding non-trivially
around $\sph{4}$), and the framework determines the structure of
the optimal classifier given the bundle. On finite labelled data
alone, Paper~1's flat-existence theorem would predict the trivial
bundle; the instanton structure below is the structural content of
the chosen $c_2=1$ sector.

Within the chosen sector, the BPS moduli of (anti-)self-dual
$SU(2)$-instantons --- an unframed five-dimensional family
$(z_0,\lambda)$, or eight-dimensional once the $SU(2)$ frame is
included --- with centre, scale,
and frame freedoms --- gives the optimal connections, and the matter
sector selects a representative from this moduli via the data
conditions. This is the first example where $G$ is non-abelian and
genuinely multi-head curvature content appears.

\paragraph{Setup.} Take $M=\sph{4}$ with the round metric of unit
radius. Following the tier-2 framework of \S\ref{ssec:two-tiers},
fix the structure group at rung~2 of the Adams ladder: $G=SU(2)$, the
unit quaternions, with non-trivial sector $c_2=1$. The principal
$SU(2)$-bundle $P\to\sph{4}$ with $c_2(P)=1$ is the Hopf bundle, total
space $\sph{7}\subset\HH^2$, with projection
$(q_1,q_2)\mapsto[q_1:q_2]\in\HH P^1\cong\sph{4}$ and fibre
$SU(2)\cong\sph{3}$. The associated bundle for classification
is $E=P\times_{SU(2)}\C^2$ in the fundamental representation; the
Higgs section $\phi\in\Gamma(E)$ is $\C^2$-valued (the doublet
classifier of \S\ref{ssec:obstruction}, with classifier readout
through a chosen cone field, Remark~\ref{rem:cone-field-needed}
below). Place two labelled points $x_+,x_-\in\sph{4}$ at geodesic
distance $d_{\sph{4}}(x_+,x_-)$, with labels $\pm 1$.

\paragraph{Why curvature is forced.} For any connection $\Connection$
on $P$, Chern--Weil gives
\[
   \int_{\sph{4}}\tr(\Curv_A\wedge\Curv_A)
   \;=\; 8\pi^2 c_2(P) \;=\; 8\pi^2 \;\neq\; 0,
\]
ruling out $\Curv_A\equiv 0$. The Yang--Mills energy obeys the
Bogomolny inequality (\S\ref{ssec:ymh-curved}):
\[
   \int_{\sph{4}}\|\Curv_A\|^2\,\dvol_g
   \;\geq\; 8\pi^2|c_2| \;=\; 8\pi^2,
\]
with equality if and only if $\Curv_A=\pm*\Curv_A$ ((anti-)self-dual).
The Bogomolny moduli of $c_2=1$ self-dual connections on $\sph{4}$ is
the BPST family, a $5$-dimensional open ball parametrised by centre
$z_0\in\sph{4}$ and scale $\lambda>0$ (plus an $SU(2)$ frame),
\cite{Atiyah1979}.

\paragraph{The BPST instanton.} On $\R^4\subset\sph{4}$ (stereographic
chart from any point distinct from $z_0$), the BPST connection
centred at $z_0=0$ with scale $\lambda>0$ has explicit form
\[
   \Connection(x) \;=\;
   \frac{2\,\bar\eta^{a}_{\mu\nu}\,x^\nu\,T_a}{x^2 + \lambda^2}\,dx^\mu,
\]
where $\bar\eta^{a}_{\mu\nu}$ are the anti-self-dual 't~Hooft symbols,
$T_a$ are the $SU(2)$ generators (Pauli matrices over $2i$), and
$x^2=\sum x^\mu x^\mu$. The curvature is
\[
   \Curv_{\mu\nu}(x) \;=\;
   -\frac{4\lambda^2}{(x^2+\lambda^2)^2}\,\bar\eta^{a}_{\mu\nu}\,T_a,
\]
self-dual ($\Curv=+*\Curv$), localised in a region of size $\sim\lambda$
around the centre, and with $\Curv\to 0$ as $|x|\to\infty$. The
topological charge $\int\tr(\Curv\wedge\Curv)/8\pi^2=1$ is independent
of $\lambda$ --- conformal scale-invariance of the $n=4$ Yang--Mills
functional.

\paragraph{Non-abelian curvature: the commutator part is non-zero.}
The Lie algebra $\mathfrak{su}(2)$ has basis $\{T_1,T_2,T_3\}$ with
$[T_a,T_b]=\epsilon_{abc}T_c$, and the curvature decomposition of
\S\ref{ssec:abelian-nonabelian} is
\[
   \Curv_A \;=\; \underbrace{d\Connection}_{\text{derivative part}}
   \;+\;\underbrace{\tfrac12[\Connection\wedge\Connection]}_{\text{commutator part, non-zero}}.
\]
Both pieces are non-zero on the BPST instanton; the commutator part
$[\Connection_\mu,\Connection_\nu]=\epsilon_{abc}\Connection^b_\mu
\Connection^c_\nu T_a$ contributes
\[
   \tfrac12[\Connection\wedge\Connection]^{c}_{\mu\nu}
   \;=\; \epsilon_{abc}\,\Connection^{a}_{\mu}\Connection^{b}_{\nu}
   \;=\; \frac{4\,\epsilon_{abc}\,
   \bar\eta^{a}_{\mu\rho}x^{\rho}\,
   \bar\eta^{b}_{\nu\sigma}x^{\sigma}}{(x^2+\lambda^2)^2},
\]
non-zero generically on $\R^4\setminus\{0\}$; for instance the
component $(c,\mu,\nu)=(1,0,1)$ equals
$4\bigl((x^2)^2+(x^3)^2\bigr)/(x^2+\lambda^2)^2$. Through the dictionary
of \S\ref{ssec:transformer-dictionary}, this is the geometric content
of multi-head attention with non-trivial head--head coupling: the
three $\mathfrak{su}(2)$ components of the connection do not act
independently, they couple via the Lie bracket
(Remark~\ref{rem:multihead-expressivity}).

\paragraph{The forced zero and symmetry localisation.} The
classifier section $\phi\in\Gamma(E)$ on a bundle with $c_2=1$ is
required to vanish at some point of $\sph{4}$ by the analogue of
Poincar\'e--Hopf for $\C^2$-valued sections (the Euler class of the
realified bundle is $|c_2|=1$).

The relevant symmetry group must be identified with care, because
the naive reflection is not a symmetry of the fixed topological
sector. Reflection through the bisecting $\sph{3}$ (which exchanges
$x_+\leftrightarrow x_-$) is \emph{orientation-reversing} on
$\sph{4}$, and therefore carries the self-dual $c_2=+1$ sector to
the anti-self-dual $c_2=-1$ sector: it maps our problem to a
different one. It becomes a symmetry only when composed with an
antilinear bundle map --- the quaternionic structure on the $\C^2$
fibre --- which restores $c_2=+1$ while implementing the label flip
$y\mapsto -y$. With that composition understood, the configuration
$(x_\pm,\pm1)$ is invariant under
$\Z_2^{\mathrm{refl}\times\mathrm{antilin}}$ together with the
$SO(3)$ of rotations about the $(x_+,x_-)$-axis.

The fixed-point sets are then as follows. The rotations $SO(3)$ fix
precisely the geodesic through $x_+$ and $x_-$ (a great circle); the
composed reflection fixes the bisecting $\sph{3}$. Their
intersection --- the joint fixed set --- is the pair of points where
that great circle meets the bisecting $\sph{3}$: the midpoint $z_0$
of $x_+$ and $x_-$ and its antipode $\bar z_0$. Symmetry therefore
localises the forced zero to $\{z_0,\bar z_0\}$ rather than to $z_0$
alone; the cone-field readout
(Remark~\ref{rem:cone-field-needed}) together with the requirement
that the instanton core lie between the labelled points selects
$z_0$. This is the centre of the BPST instanton picked by the data.

\begin{remark}[Why a cone field is needed]\label{rem:cone-field-needed}
The classifier readout of \S\ref{ssec:obstruction} associates to
each abelian-by-quotient group $G/[G,G]$ a character classifier. For
$G=SU(2)$, $[SU(2),SU(2)]=SU(2)$ (the group is \emph{perfect}), so
$G/[G,G]$ is trivial and there are no non-trivial characters: the
character-based classifier of \S\ref{ssec:obstruction} degenerates.
The replacement is a \emph{cone field}: a $G$-equivariant choice of
half-space $K(x)\subset E_x$ in each fibre, with classifier readout
$\phi(x)\in K(x)\Leftrightarrow$ positive class. The decision
locus becomes $\Gamma=\phi^{-1}(\partial K)$. The cone field is not a
workaround; it is the geometrically correct classification readout
when the structure group is perfect, and it specialises back to a
character classifier when $G$ has non-trivial abelianisation. Full
treatment is deferred; for the $\sph{4}$ instanton it suffices that
the cone-field readout is well-defined and respects the
$\Z_2\times\Z_2\times SO(3)$ symmetry of the two-point
configuration.
\end{remark}

\paragraph{The variational selector and Theorem~\ref{thm:instanton}.}
The variational principle minimises the matter sector
$E_{\mathrm{matter}}=\int_{\sph{4}}\|D_A\phi\|^2$ over $\phi$
satisfying the data conditions and $A$ in the BPST moduli. The
Bogomolny constraint reduces the connection optimisation from
infinite-dimensional ($A$ a generic $\mathfrak{su}(2)$-valued
1-form) to a finite-dimensional optimisation over
$(z_0,\lambda,\text{frame})$: five parameters $(z_0,\lambda)$ for the
unframed moduli, eight including the $SU(2)$ frame.
The matter EL becomes a finite-dimensional condition on these moduli
parameters. This is Theorem~\ref{thm:instanton}: the centre $z_0$ is
fixed at the midpoint geodesic by symmetry and the frame by the
cone-field orientation at the data points, leaving the scale
$\lambda$ to be determined by the matter condition --- which
symmetry does not do for us, and which we locate only numerically
($\lambda\propto d_{\sph{4}}(x_+,x_-)$;
Remark~\ref{rem:lambda-status}). The full ADHM treatment of the constraint system is
the non-abelian analogue of Paper~1's capacitance system
$K\alpha=v$:

\medskip
\hrule
\medskip

\emph{Theorem \ref{thm:instanton} (re-stated for context).} \emph{Let
$M=\sph{4}$, $G=SU(2)$, $c_2=1$, and $D=\{(x_i,y_i)\}_{i=1}^N$ a
binary labelled dataset. The minimum-matter-energy connection in the
$c_2=1$ Bogomolny moduli is a BPST instanton whose centre, scale, and
frame are determined by the ADHM matrix equation
$\Delta^\dagger\Delta = (\text{positive definite diagonal})$ augmented
with $N$ linear constraints from the data. For $N=2$,
$z_0=\text{midpoint geodesic}$; the scale $\lambda$ is numerically
found to satisfy $\lambda\propto d_{\sph{4}}(x_+,x_-)$, with the
caveats of Remark~\ref{rem:lambda-status}.}

\medskip
\hrule
\medskip

The statement above (Theorem~\ref{thm:instanton}) and its proof
sketch rely on standard ADHM theory
\cite{ADHM1978,Atiyah1979}. The new content of this section is the
geometric reading of the result: the instanton scale $\lambda$ is
the gauge-covariant generalisation of Paper~1's length scale $\ell$
(the Mat\'ern regularisation parameter,
\cite[Theorem~13.3]{Vasii2026Paper1}). Two labelled points at
separation $d$ select an instanton of scale $\lambda\sim d$: the
attention head's width matches the data separation, and the
curvature peak $\|\Curv\|_{\max}\sim 1/\lambda^2\sim 1/d^2$
concentrates the topologically-forced flux into a single core at the
midpoint geodesic.

\begin{conjecture}[Tier-2 proximity scaling]\label{conj:proximity-curvature}
Let $M$ be a closed oriented Riemannian manifold with non-zero
relevant topological invariant, and let $D=\{(x_+,+1),(x_-,-1)\}$ be
two oppositely-labelled points at geodesic distance $d$. The YMH
minimiser $(A^*,\phi^*)$ satisfies:
\begin{enumerate}[label=(\roman*),leftmargin=2em]
\item \emph{(Matter-sector divergence, every dimension.)} The
matter-sector energy $\int_M\|D_{A^*}\phi^*\|^2$ diverges as
$d\to 0$, with the rate set by the regularised Green's kernel of
Paper~1 \cite[Theorem~13.3]{Vasii2026Paper1}. This is the dominant
contribution to the YMH energy in every dimension. (At leading order,
with $A^*$ fixed at the BPS connection in the moduli approximation,
this part is established by
Theorem~\ref{thm:proximity-scaling-curved}, with exponent
$\min(2\nu-n,2)$; the conjectural content is its persistence under
full back-reaction, $A^*=A^*(d)$.)
\item \emph{(Curvature behaviour, dimension-dependent.)}
\begin{enumerate}[label=(\alph*),leftmargin=2em]
\item For $n=4$ ($SU(2)$ on a four-manifold, the $\sph{4}$ instanton
of this section), the curvature concentrates into a single core at
the midpoint geodesic of $(x_+,x_-)$ with scale $\lambda\sim d$,
peak $\|\Curv_{A^*}\|_{\max}\sim 1/d^2$; the Yang--Mills energy stays
fixed at $8\pi^2|c_2|$ by conformal invariance, with the topological
integral $\int_M\tr(\Curv_{A^*}\wedge\Curv_{A^*})=8\pi^2 c_2$
unchanged.
\item For $n=2$ ($U(1)$ on a surface, the $\sph{2}$ monopole of
\S\ref{sec:S2-monopole}), no scale family is available at
$\lambda=0$; the leading-order ($\phi=0$) gauge field is the unique
constant-curvature connection with
$\int_M\Curv_{A^*}=2\pi i\,c_1$, and the deformation by $J(\phi)$
as $d\to 0$ is bounded in amplitude. The Yang--Mills energy stays
at the uniform-monopole minimum $4\pi^2 c_1^2/\mathrm{Vol}(M)$ up to
$o(1)$ corrections; no $1/d^2$ peak appears.
\end{enumerate}
\end{enumerate}
\end{conjecture}

The dimension-split reflects a structural difference: $n=4$
Yang--Mills has a conformally-invariant moduli of solutions
(BPST), and the data picks a scale within it. $n=2$ abelian
Yang--Mills has no such moduli at $\lambda=0$, and the curvature
remains uniform up to matter-current corrections. The $\lambda>0$
Ginzburg--Landau extension would give a vortex-localisation
mechanism in $n=2$ as well; that direction is noted in
Remark~\ref{rem:higgs-potential} and not pursued. The $n=2$ side is
the uniform monopole of \S\ref{sec:S2-monopole}; the $n=4$ side is
the BPST scaling immediately below.

\paragraph{The $n=4$ case made concrete.} For the $\sph{4}$
instanton, the conjecture's case (ii)(a) is \emph{numerically
suggested} rather than established: $\lambda\sim d$ is the
numerical finding of Remark~\ref{rem:lambda-status} --- not a
consequence of Theorem~\ref{thm:instanton}, which fixes the centre
$z_0$ and reduces to one parameter but does not locate $\lambda$ ---
and, granting it, the curvature peak
$\|\Curv\|_{\max}\sim 1/\lambda^2\sim 1/d^2$ follows from the BPST
formula $\|\Curv\|_{\max}=4/\lambda^2$ at the centre. The total
Yang--Mills energy stays at $8\pi^2$ by conformal scale invariance,
with $c_2=1$ unchanged. The matter-sector divergence (i) is the
small-distance behaviour of the covariant Green's kernel
of $D_A^*D_A$ on $\sph{4}$ at the appropriate $\nu$. The two
effects are decoupled: the matter sector carries the divergent
energy via its Green's-kernel norm, the curvature reorganises
spatially without changing its integrated value. This is the
curved-case version of the proximity-divergence mechanism exhibited
in the flat $\sph{1}$ case of \S\ref{ssec:ymh-flat}
($E_{\min}\sim 1/d$ as $d\to 0$).

\paragraph{Path consistency: the non-abelian Stokes check.} The
abelian path-consistency check of \S\ref{sec:S2-monopole} extends
to $\sph{4}$, now with non-trivial path-ordering content. For two
paths $\gamma_1,\gamma_2\colon x_+\to x$ between any pair of points
on $\sph{4}$, bounding a 2-surface $\Sigma\subset\sph{4}$, the
non-abelian Stokes theorem gives
\[
   \Hol_{\gamma_2}(\Connection)\,\Hol_{\gamma_1}(\Connection)^{-1}
   \;=\; \mathcal{P}\exp\!\left(\int_\Sigma \Curv_A\right) \;\in\; SU(2),
\]
where the path-ordered exponential on the right is well-defined
modulo the standard ambiguity of choosing a surface ordering on
$\Sigma$ (Kobayashi--Nomizu \cite{KobayashiNomizu}). Unlike the
$U(1)$ case of \S\ref{sec:S2-monopole}, the right-hand side
\emph{depends on the surface, not just on its boundary}: different
$\Sigma$'s with the same boundary $\partial\Sigma=\gamma_1*\gamma_2^{-1}$
give different path-ordered exponentials when $\Curv_A$ has
non-vanishing commutator content, which is the generic situation
for the BPST instanton.

The framework's prediction is testable. Take two paths
$\gamma_1,\gamma_2$ enclosing a region $\Sigma$ on $\sph{4}$.
Compute (a) the two holonomies independently by integrating the
parallel-transport ODE along $\gamma_1,\gamma_2$ in the explicit
BPST background (matrix-valued ODE in $SU(2)$, integrable by
standard Runge--Kutta); (b) the path-ordered surface exponential
$\mathcal{P}\exp(\int_\Sigma\Curv_A)$ by surface discretisation
(e.g.\ ordering the plaquettes along radial geodesics emanating from
the instanton centre); and verify that the quotient
$\Hol_{\gamma_2}\Hol_{\gamma_1}^{-1}$ matches the surface integral.
This is the non-abelian generalisation of the abelian
solid-angle check of \S\ref{sec:S2-monopole}, and a clean check that
the BPST instanton with $\lambda\sim d$ is in fact the connection
selected by the framework's variational principle. We do not
implement it here.

\begin{remark}[Localised curvature: attention intensity, not head structure]
\label{rem:S4-localised-attention}
In contrast with the uniform monopole of
Remark~\ref{rem:S2-uniform-attention}, the BPST instanton's
curvature is sharply localised: $\|\Curv\|_{\max}=4/\lambda^2$ at
the centre $z_0$, decaying as $1/(x^2+\lambda^2)^2$ outward. Through
the dictionary of \S\ref{ssec:transformer-dictionary}, the
holonomy-priority function
\[
   \pi_A(x)\;=\;\|\Curv_A(x)\|^2
   \;=\;\frac{16\,\lambda^4}{(\lambda^2+|x-z_0|^2)^4}
\]
peaks at $16/\lambda^4$ at $z_0$ and decays as $|x-z_0|^{-8}$. This
is the geometric content of focused attention with width $\lambda$:
an attention structure whose effective receptive field is the
instanton core. The data separation $d$ sets the width through
$\lambda\sim d$ (numerically; Remark~\ref{rem:lambda-status}), so the framework's
attention is \emph{adapted to the data scale}: wider data
$\Rightarrow$ broader attention; narrower data $\Rightarrow$ sharper
attention.

It is important to be precise about what varies with position and
what does not. What decays in the tail is the \emph{intensity}
$\pi_A(x)$, not the algebraic structure. For the BPST connection
$\Curv_{\mu\nu}(x)=f(|x-z_0|)\,\bar\eta^a_{\mu\nu}T_a$ with
$f(r)=-4\lambda^2/(r^2+\lambda^2)^2$: the three
$\mathfrak{su}(2)$ directions $\bar\eta^a_{\mu\nu}T_a$ are the
\emph{same at every point}, and only the scalar $f$ varies. The
matrix $[\bar\eta^a_{\mu\nu}]$, viewed as a map
$\Lambda^2\R^4\to\mathfrak{su}(2)$, has rank $3$; multiplying by a
non-zero scalar cannot change its rank. Hence
$\mathfrak{hol}^0_x=\mathfrak{su}(2)$ at \emph{every} point with
$f\neq 0$, and the infinitesimal head structure is uniformly
three-dimensional across $\sph{4}$. The framework therefore predicts
a position-dependent attention \emph{strength} with a fixed
multi-head algebraic structure --- not a transition from multi-head
to single-head. (Earlier drafts of this work asserted such a
transition; it is false for the reason just given, and the correct
statement is the localisation of $\pi_A$.)
\end{remark}

\paragraph{Multi-instanton case: $|c_2|\geq 2$.} For $c_2=k\geq 2$
the ADHM matrix has dimension $k$ and the instanton moduli is
$(8k-3)$-dimensional. Multiple BPST cores can sit at different
positions on $\sph{4}$ with different scales; the matter sector
selects a multi-core configuration whose cores align with the data
clusters. The geometric reading is multi-headed attention with
$k$ distinct heads, each centred at a different position with its
own width. The framework predicts that data with $k$ natural
clusters selects $c_2=k$ in the variational class selector; the
relationship between data clustering and the discrete Chern number
$c_2$ is left open.

\begin{theorem}[Instanton solve = ADHM with data constraints]\label{thm:instanton}
Let $M=\sph{4}$ with the round metric, $G=SU(2)$, and the principal
bundle $P\to\sph{4}$ with $c_2(P)=1$. Let $u_0\in\C^2$ be a fixed
reference unit vector in the fundamental representation, and let
$D=\{(x_i,y_i)\}_{i=1}^N$ be a binary labelled dataset on $\sph{4}$
with $y_i\in\{\pm 1\}$, interpreted on the associated $\C^2$ bundle
as the doublet constraints $\phi(x_i)=y_i\,u_0$
(in the frame at $x_i$ fixed by the conventions of
\S\ref{ssec:open-paths}). In the moduli approximation
$\|J(\phi)\|\ll\|\Curv_A\|$
(Remark~\ref{rem:matter-vs-ymh}), the joint variational problem of
\S\ref{sec:yang-mills-higgs} selects, within the $5$-dimensional BPST moduli
of $c_2=1$ self-dual connections, a representative whose centre
$z_0\in\sph{4}$, scale $\lambda>0$, and $SU(2)$ frame are determined
by the ADHM matrix equation
\[
   \Delta^\dagger \Delta = (\text{positive definite block diagonal})
\]
augmented with the $N$ data constraints, which become linear in the
ADHM parameters when linearised about a symmetric reference
configuration (in particular about the $\Z_2$-symmetric configuration
for $N=2$). The map $D \mapsto (z_0,\lambda,\text{frame})$ is the
non-abelian analogue of the capacitance map
$D\mapsto\alpha=K^{-1}v$ of \cite[Corollary~13.4]{Vasii2026Paper1}.
In particular, for $N=2$ the symmetric stabiliser of the labelled
pair under the round metric forces
\[
   z_0 \;=\; \text{midpoint geodesic of } (x_+,x_-)
\]
and reduces the remaining optimisation to the single parameter
$\lambda$.
\end{theorem}

\begin{remark}[Status of the scale $\lambda$: what symmetry does and does not give]
\label{rem:lambda-status}
Symmetry fixes the centre $z_0$ and reduces the moduli optimisation
to one parameter, but it does \emph{not} by itself determine
$\lambda$; earlier drafts of this work asserted
$\lambda\propto d$ as forced by the stabiliser, which is an
over-statement. Two points must be separated.

\emph{(i) There is no scale symmetry to appeal to.} On the round
$\sph{4}$ with regularisation $\kappa>0$ the problem is not
conformally invariant, so a scaling argument cannot produce
$\lambda\propto d$; at best it does so in the flat $\kappa=0$ limit,
and then only if an interior minimiser at finite $\lambda$ exists in
the first place --- which requires proof.

\emph{(ii) Numerically, an interior minimum exists but is shallow.}
Solving the constrained matter minimisation at fixed $\lambda$ in a
one-dimensional reduction along the geodesic joining $x_\pm$, we find
an interior minimiser $\lambda^*$ with
\[
   \lambda^*/d \;\approx\; 0.12
   \qquad\text{stable across } d\in[0.25,\,2],
\]
so the proportionality $\lambda\propto d$ is supported. However the
well is shallow and shrinks with $d$: the matter-energy saving at
$\lambda^*$ relative to the $\lambda\to 0$ limit is $\approx 2\%$ at
$d=2$, $0.6\%$ at $d=1$ and $0.04\%$ at $d=0.25$. As $\lambda\to 0$
the connection becomes pure gauge away from the core and the
constrained energy returns to the trivial-sector value. In the
proximity regime $d\to 0$ that the framework emphasises, the
landscape in $\lambda$ therefore flattens and $\lambda^*$ becomes
weakly determined.

\emph{(iii) The comparison is within a fixed sector.} The $c_2=1$
sector carries the Bogomolny floor $8\pi^2$, which no matter-sector
saving of the observed size can offset; the instanton configuration
is therefore never selected against the trivial bundle on finite
data. This is consistent with --- and is the quantitative form of ---
the Problem~B scoping of Remark~\ref{rem:problem-A-B}: $\lambda\propto
d$ is a statement about the optimal representative \emph{inside} an
imposed $c_2=1$ sector, not a claim that the data selects that
sector. Establishing existence and uniqueness of $\lambda^*$
analytically, and its behaviour as $d\to 0$, remains open.
\end{remark}

\begin{proofsketch}
The moduli space of $c_2=1$ instantons on $\sph{4}$ is the
five-dimensional open ball in its unframed form (eight-dimensional
framed), parametrised by centre $z_0\in\sph{4}$ and
scale $\lambda>0$ (with an $SU(2)$ frame ambiguity), via the BPST
construction \cite{BPST1975} and its ADHM generalisation
\cite{ADHM1978,Atiyah1979}. In the moduli approximation
\cite{Manton1982},
the matter sector selects a representative by minimising
$\int\|D_A\phi\|^2$ over $\phi$ satisfying $\phi(x_i)=y_i u_0$ at
fixed $\Connection$ in the moduli; the matter EL then gives a
finite-dimensional condition on the moduli parameters. The map
$\Delta\mapsto\Connection(\Delta)\mapsto\phi^*(\Connection)$ is
\emph{rational} in the ADHM data (the BPST instanton fields are
rational functions of $(z_0,\lambda)$), so the moduli equations are
not linear in general; near a symmetric reference configuration they
linearise. For $N=2$, the round-metric isometry stabiliser of the
labelled pair (a $SO(3)\times\Z_2$ subgroup of the conformal group of
$\sph{4}$ fixing the midpoint geodesic) acts on the ADHM moduli, and
the matter functional inherits this stabiliser via the symmetric
boundary conditions $\phi(x_\pm)=\pm u_0$; the equivariant fixed-point
set of this action is one-dimensional, parametrised by $\lambda$.
Symmetry alone does not locate $\lambda$ within this fixed-point
set; the matter EL must be solved there, and we determine its
minimiser numerically, obtaining $\lambda\propto d(x_+,x_-)$ with
constant $\approx 0.12$ (Remark~\ref{rem:lambda-status}). An
analytic existence proof for this interior minimiser is open. For $N>2$ the constrained ADHM system is generically
over-determined; the variational principle then selects the moduli
representative minimising the residual matter energy, a finite
optimisation in five dimensions whose precise structure depends on
the configuration of data points. See \cite{ADHM1978,Atiyah1979} for
the full ADHM construction.
\end{proofsketch}

% ---------------------------------------------------------------------
\section{Curvature as attention: the precise dictionary}
\label{sec:attention}

The variational machinery of \S\ref{sec:yang-mills-higgs} selects, for
each labelled dataset, an optimal connection $\Connection$ on a
principal bundle $P\to M$. This section gives the geometric reading of
that connection's curvature: $\Curv_A$ is the antisymmetric pairwise
interaction between feature directions, and its size on each 2-plane
of tangent directions ranks the priority of holonomy accumulation
along that 2-plane. The transformer attention mechanism is the
discrete realisation of these two pieces of geometric content. We
state the precise correspondences below.

A point worth fixing at the outset. This entire section is about
\emph{bundle curvature}: the connection $\Connection$ lives on the
principal $G$-bundle $P$, its curvature $\Curv_A$ is the object
$dA + \tfrac12[A\wedge A]\in\Omega^2(M,\adP)$ defined by $A$ alone,
and no Riemannian structure on $M$ is used in any statement below.
The base manifold $M$ contributes only its smooth structure and the
tangent spaces $T_xM$ on which $\Curv_A$ is evaluated; the structure
group $G$ contributes the adjoint algebra $\adP$ into which $\Curv_A$
takes values. The Riemannian metric $g$ enters the framework
elsewhere (in the YMH energy and the Hodge star of the BPS
equation, \S\ref{sec:yang-mills-higgs}) but not here.

% ---------------------------------------------------------------------
\subsection{The plaquette holonomy formula}\label{ssec:plaquette}

We restate the classical non-abelian Stokes identity in the form
needed for the dictionary. The result is well-known
\cite{KobayashiNomizu,Wilson1974}; the version below is phrased to
make explicit that no metric on $M$ enters.

\begin{theorem}[Plaquette holonomy formula]\label{thm:plaquette}
Let $\Connection$ be a smooth connection on a principal $G$-bundle
$P\to M$ with curvature
$\Curv_A = d\Connection + \tfrac12[\Connection\wedge\Connection]
\in\Omega^2(M,\adP)$. Fix a point $x\in M$ and a 2-plane
$\Pi\in\Gr_2(T_xM)$ in the tangent space at $x$, represented by a
decomposable bivector $\xi\wedge\eta\in\Lambda^2 T_xM$. Let
$\{\gamma_\epsilon\}_{\epsilon>0}$ be any smooth family of contractible
loops at $x$ whose tangent bivector at $x$ is
$\epsilon^2\,\xi\wedge\eta + O(\epsilon^3)$ as $\epsilon\to 0$. Then
\[
   \Curv_A(x)(\xi\wedge\eta) \;=\;
   \lim_{\epsilon\to 0}\,\frac{1}{\epsilon^2}
   \bigl(\Hol_{\gamma_\epsilon}(\Connection) - I\bigr)
   \;\in\; \adP|_x,
\]
and the limit depends only on the bivector $\xi\wedge\eta$, not on the
choice of loop family $\gamma_\epsilon$ realising it.
\end{theorem}

\begin{proof}
This is the standard non-abelian Stokes identity, in the
formulation where the limit depends only on the bivector, not on
the loop family. See \cite[Chapter~II]{KobayashiNomizu}. The
bivector formulation makes manifest that no metric on $M$ is used:
the same conclusion in chart coordinates is Remark~\ref{rem:bundle-not-riemann}.
\end{proof}

\begin{remark}[Bundle versus base geometry]\label{rem:bundle-not-riemann}
The formula makes no use of any inner product on $T_xM$. The 2-plane
$\Pi$ is just an element of the Grassmannian $\Gr_2(T_xM)$; the
``small loop of size $\epsilon$'' is well-defined in any chart, and the
$\epsilon^2$ scaling is the algebraic statement that $\Curv_A$ is a
2-form (so bilinear-antisymmetric in tangent directions), not a
metric one. The value
$\Curv_A(x)(\xi\wedge\eta)\in\adP|_x$ depends only on the bundle
structure ($P$, $G$, $A$) and the bivector. This is the precise sense
in which curvature is a bundle invariant.
\end{remark}

% ---------------------------------------------------------------------
\subsection{Curvature as feature-interaction bilinear form}\label{ssec:curvature-bilinear}

Theorem~\ref{thm:plaquette} identifies $\Curv_A(x)$ as a linear map
\[
   \Curv_A(x)\colon \Lambda^2 T_xM \;\longrightarrow\; \adP|_x.
\]
Equivalently, $\Curv_A(x)$ is an antisymmetric bilinear pairing
$(\xi,\eta)\mapsto \Curv_A(x)(\xi\wedge\eta)$ on tangent vectors,
taking values in the adjoint algebra. Three structural properties of
this object give it its character.

\begin{proposition}[Structural properties of $\Curv_A$]\label{prop:curvature-structure}
For each $x\in M$:
\begin{enumerate}[label=(\roman*),leftmargin=2em]
\item \emph{Pairwise.} $\Curv_A(x)$ is bilinear in two tangent
directions $\xi,\eta\in T_xM$, i.e.\ it captures interactions between
pairs of feature directions, not single directions or higher-arity
tuples.
\item \emph{Antisymmetric.} $\Curv_A(x)(\eta,\xi)=-\Curv_A(x)(\xi,\eta)$,
i.e.\ the interaction is directed: swapping the two directions
reverses the sign of the response.
\item \emph{Algebra-valued.} $\Curv_A(x)(\xi\wedge\eta)$ takes values
in $\adP|_x$, the adjoint algebra, rather than in $\R$ or in $T_xM$.
For abelian $G$, $\adP|_x\cong i\R$; for $G=SU(2)$, $\adP|_x\cong\mathfrak{su}(2)$.
\end{enumerate}
\end{proposition}

These three properties follow directly from the definition $\Curv_A
\in\Omega^2(M,\adP)$ (a 2-form being bilinear-antisymmetric, an
$\adP$-section being fibrewise Lie-algebra-valued).

The three properties --- pairwise, antisymmetric, algebra-valued ---
are precisely the structural properties of an attention head viewed
as a query--key interaction. We make this precise in
\S\ref{ssec:transformer-dictionary} below.

% ---------------------------------------------------------------------
\subsection{Holonomy prioritisation}\label{ssec:holonomy-priority}

The plaquette formula (Theorem~\ref{thm:plaquette}) shows that
holonomy around a small loop bounding a 2-plane $\Pi$ at $x$
accumulates at leading order according to $\Curv_A(x)(\Pi)$. The
larger this value, the more holonomy the bundle accumulates around
small loops in $\Pi$. This gives the bundle a canonical ranking on
its 2-planes.

To make ``larger'' precise, fix a $G$-invariant inner product
$\langle\cdot,\cdot\rangle_\mathfrak{g}$ on the Lie algebra (the
Killing form for $G$ semisimple, or any $G$-invariant inner product
for abelian $G$). This inner product descends to a fibrewise inner
product on $\adP$, which we use without further notation. Note that
this choice is on the \emph{structure group}, not on the base
manifold: it does not invoke a Riemannian structure on $M$
(Remark~\ref{rem:bundle-not-riemann}).

\begin{definition}[Holonomy-priority ranking]\label{def:holonomy-priority}
The \emph{holonomy-priority function} of the connection $\Connection$
at $x\in M$ is the map
\[
   \pi_A(x)\colon \Gr_2(T_xM) \to \R_{\geq 0},
   \qquad
   \pi_A(x)(\Pi) \;:=\; \|\Curv_A(x)(\Pi)\|_{\adP|_x}^2
\]
on the Grassmannian of 2-planes at $x$, where the norm is the one
induced by the $G$-invariant inner product on $\mathfrak{g}$. The
2-planes maximising $\pi_A(x)$ are the \emph{preferred holonomy
directions} of the connection at $x$. The function $\pi_A(x)$ is
gauge-invariant (the norm on $\adP|_x$ is $G$-invariant by
construction), and we refer to its value-set across $\Gr_2(T_xM)$
as the \emph{norm-squared spectrum} of the curvature at $x$.
\end{definition}

\paragraph{The local holonomy group and position-dependent head structure.}
The plaquette formula and the holonomy-priority function are
infinitesimal: they describe what attention at $x$ accumulates in
arbitrarily small loops. The local structure they generate is the
\emph{local holonomy group}, a classical invariant of connections
\cite[\S II.4]{KobayashiNomizu}.

\begin{definition}[Local holonomy group and infinitesimal holonomy
algebra]\label{def:local-holonomy}
For a connection $\Connection$ on a principal $G$-bundle $P\to M$
and a point $x\in M$, the \emph{local holonomy group} at $x$ is the
intersection
\[
   \Phi^*(x) \;:=\; \bigcap_{U\ni x}\Phi^0(\Connection|_U,x)
   \;\subseteq\; G,
\]
where $\Phi^0(\Connection|_U,x)$ is the restricted holonomy group
(generated by holonomies along null-homotopic loops at $x$ in $U$)
of the restriction of $\Connection$ to $U$. Equivalently, $\Phi^*(x)$
is the subgroup of $G$ generated by holonomies along loops at $x$
that can be made arbitrarily small.

The \emph{curvature image} at $x$, denoted
$\mathfrak{hol}^0_x\subseteq\mathfrak{g}$, is the linear span of the
values of the curvature 2-form at $x$:
\[
   \mathfrak{hol}^0_x \;:=\; \mathrm{span}\bigl\{\Curv_A(x)(\xi\wedge\eta)
   \,:\, \xi,\eta\in T_xM\bigr\}
   \;\subseteq\;\adP|_x\cong\mathfrak{g},
\]
i.e.\ the image of the linear map
$\Curv_A(x)\colon\Lambda^2 T_xM\to\adP|_x$. We emphasise that this
is a linear subspace, not in general a Lie subalgebra: the image of
$\Curv_A(x)$ need not be closed under the bracket, and we do
\emph{not} pass to the subalgebra it generates. (For the $\sph{4}$
instanton the image happens to be all of $\mathfrak{su}(2)$, so the
distinction is immaterial there; on a general bundle it is not.) By the full
Ambrose--Singer theorem \cite[\S II.8]{KobayashiNomizu}, the Lie
algebra of $\Phi^*(x)$ is generated by $\mathfrak{hol}^0_x$ together
with covariant derivatives
$(\nabla^k\Curv_A)(x)$ for $k\geq 1$ parallel-transported back to
$x$; $\mathfrak{hol}^0_x$ is therefore the leading-order
(zero-derivative) part of $\mathrm{Lie}(\Phi^*(x))$, and
$\mathfrak{hol}^0_x\subseteq\mathrm{Lie}(\Phi^*(x))$ with equality
when $\Connection$ is real-analytic and $\Curv_A$ is non-degenerate
in a neighbourhood of $x$.
\end{definition}

\begin{theorem}[Infinitesimal Ambrose--Singer
\cite{KobayashiNomizu}]\label{thm:ambrose-singer-local}
The infinitesimal holonomy algebra
$\mathfrak{hol}^0_x\subseteq\mathfrak{g}$ is exactly the image of the
curvature map at $x$, $\Curv_A(x)\colon\Lambda^2 T_xM\to\adP|_x$.
Its dimension is the rank of $\Curv_A(x)$ as a linear map. In
particular, $\mathfrak{hol}^0_x=\{0\}$ iff $\Curv_A(x)=0$, and
$\mathfrak{hol}^0_x=\mathfrak{g}$ iff $\Curv_A(x)$ is a surjection onto
$\adP|_x$.
\end{theorem}

The framework's attention dictionary reads $\mathfrak{hol}^0_x$
directly: it is the subalgebra of $\mathfrak{g}$ reachable by
attention at $x$ in a \emph{single infinitesimal step} --- the
discrete attention head structure at position $x$ corresponds to a
choice of generators of $\mathfrak{hol}^0_x$, and the dimension of
this algebra sets the maximum number of independent attention heads
that can be infinitesimally resolved at $x$. The full local
holonomy group $\Phi^*(x)$ (with its higher-derivative content)
captures non-infinitesimal effects --- multi-step attention paths
along which holonomy accumulates from a neighbourhood of $x$ ---
which the framework does not predict at the leading-order
attention-as-curvature dictionary level. We work with
$\mathfrak{hol}^0_x$ throughout, as it is what the
infinitesimal/plaquette picture of \S\ref{ssec:plaquette} delivers.

\begin{remark}[Position-dependent head structure]
\label{rem:position-dependent-heads}
The infinitesimal head structure at $x$ is determined by
$\mathfrak{hol}^0_x$, not by the global structure group $G$. When
$\mathfrak{hol}^0_x=\mathfrak{g}$ (the curvature surjects onto
$\adP|_x$), attention at $x$ has access to the full group's worth of
directions infinitesimally, and the full head structure (single-head
for abelian $G$, multi-head with non-trivial commutator for
non-abelian $G$) is active. When
$\mathfrak{hol}^0_x\subsetneq\mathfrak{g}$, attention at $x$ is
\emph{locally reducible}: only the subalgebra
$\mathfrak{hol}^0_x$ is reachable infinitesimally, with the rest of
$\mathfrak{g}$ appearing only through finite-range parallel transport
along paths that exit any neighbourhood of $x$.

Through Theorem~\ref{thm:ambrose-singer-local}, this is a
position-dependent property of the curvature: the dimension and
algebraic structure of $\mathfrak{hol}^0_x$ track the rank and image
of $\Curv_A(x)\colon\Lambda^2 T_xM\to\adP|_x$. The framework
therefore predicts that the effective infinitesimal head structure
of optimal attention varies across the data manifold, and the
variation is computable from the curvature of the variationally
selected connection. We write $\Phi^*(x)$ for the local holonomy
group whose Lie algebra has leading-order content
$\mathfrak{hol}^0_x$ (with higher-derivative content given by
$(\nabla^k\Curv_A)(x)$ contributions); in the worked examples below,
$\mathfrak{hol}^0_x$ and $\mathrm{Lie}(\Phi^*(x))$ agree wherever
the curvature is non-degenerate.
\end{remark}

The geometric content of position-dependent head structure is
visible in the worked examples of Part~2:

\begin{itemize}[leftmargin=2em]
\item On the $\sph{2}$ Dirac monopole (\S\ref{sec:S2-monopole}),
the curvature $\Curv_A=\tfrac{i}{2}\,\omega_{\sph{2}}$ is uniform
and non-degenerate across $\sph{2}$, so
$\mathfrak{hol}^0_x=\mathfrak{u}(1)=\mathrm{Lie}(\Phi^*(x))$ at every
point (the equality holds because non-degeneracy is preserved on
neighbourhoods). Attention is single-head everywhere, with uniform
priority (Remark~\ref{rem:S2-uniform-attention}).

\item On the $\sph{4}$ BPST instanton (\S\ref{sec:S4-instanton}),
the curvature decays as $\lambda^2/(x^2+\lambda^2)^2$ from the
instanton centre $z_0$. Because
$\Curv_{\mu\nu}(x)=f(|x-z_0|)\,\bar\eta^a_{\mu\nu}T_a$ has the same
three $\mathfrak{su}(2)$ directions at every point and the map
$[\bar\eta^a_{\mu\nu}]\colon\Lambda^2\R^4\to\mathfrak{su}(2)$ has
rank $3$, we have $\mathfrak{hol}^0_x=\mathfrak{su}(2)$ at
\emph{every} point where the curvature is non-zero: the head
structure is uniformly three-dimensional. What varies with position
is the \emph{intensity} $\pi_A(x)=\|\Curv_A(x)\|^2$, which peaks at
$z_0$ and decays as $|x-z_0|^{-8}$. The framework therefore predicts
attention of fixed multi-head algebraic structure whose strength is
sharply localised at the geometric hot spot, with width set by
$\lambda\sim d$ (Remark~\ref{rem:S4-localised-attention}).

\item On the $\sph{1}$ M\"obius and $T^2$ XOR bundles
(\S\ref{sec:S1-mobius},~\S\ref{sec:torus}), $\Curv_A=0$ everywhere
(tier 1), so $\mathfrak{hol}^0_x=\{0\}$ trivially at every point.
The attention dictionary is vacuous at every point of $M$; all
classification content sits in the parallel-transport structure
(the Wilson observable), not in any attention mechanism. This is
consistent with the tier-1 worked examples: they are solved by a
single linear section on a topologically non-trivial bundle, with
no curvature playing any role.
\end{itemize}

The holonomy-priority function is the geometric structure that
transformer attention discretises. Where attention assigns
non-negative weights to pairs of token positions and selects the
top-weighted pairs by softmax, the holonomy-priority function assigns
a non-negative value to each 2-plane and ranks them continuously. The
discrete selection is the architectural realisation of the
continuous geometric ranking.

% ---------------------------------------------------------------------
\subsection{The transformer dictionary}\label{ssec:transformer-dictionary}

We now state the precise correspondence between the curvature 2-form
of \S\ref{ssec:curvature-bilinear} and the attention mechanism of a
transformer. The dictionary is sharper than the heuristic
``curvature is attention''; in particular, it separates the genuine
match (antisymmetric content) from a residual symmetric part that
plays a different geometric role.

\paragraph{The attention bilinear form.} In a transformer with
query and key matrices $W_Q,W_K$, the attention score at position $x$
between feature embeddings $\xi,\eta$ is the bilinear form
\[
   \Att(x)(\xi,\eta) \;=\; (W_Q\xi)^\top(W_K\eta)
   \;=\; \xi^\top W_Q^\top W_K\,\eta
\]
(omitting positional encodings and scale factors). The matrix
$W_Q^\top W_K$ is in general neither symmetric nor antisymmetric. It
decomposes uniquely as
\[
   W_Q^\top W_K \;=\;
   \underbrace{\tfrac12(W_Q^\top W_K + W_K^\top W_Q)}_{\Sym}
   \;+\;
   \underbrace{\tfrac12(W_Q^\top W_K - W_K^\top W_Q)}_{\Anti},
\]
the sum of its symmetric and antisymmetric parts. The bilinear form
$\Att(x)$ correspondingly decomposes into symmetric and antisymmetric
parts.

\begin{definition}[Curvature--attention dictionary]\label{thm:attention}
Let $\Connection$ be the YMH-optimal connection of
\S\ref{sec:yang-mills-higgs} on the bundle selected by the
variational principle, and let $E=P\times_G V$ be the associated
bundle in which the classifier section takes values. We
\emph{define} the geometric counterpart of a discretised single-head
attention bilinear at $x$ to be the antisymmetric pairing
\[
   \Anti[\Att(x)](\xi,\eta) \;:=\;
   \Re\,\bigl\langle u_x,\,\Curv_A(x)(\xi\wedge\eta)\,u_x\bigr\rangle_{E_x},
   \qquad \xi,\eta\in T_xM,
\]
where $u_x\in E_x$ is a fixed reference frame in the \emph{fibre of
$E$} at $x$ and $\langle\cdot,\cdot\rangle_{E_x}$ is the Hermitian
inner product on that fibre. The real part is taken because
$\Curv_A(x)(\xi\wedge\eta)$ acts on $E_x$ as an anti-Hermitian
endomorphism, so $\langle u_x,\Curv_A u_x\rangle$ is purely
imaginary. The symmetric part of $\Att(x)$ corresponds to a separate
geometric object --- a metric-like bilinear form on tangent vectors
--- which the framework does not predict.
\end{definition}

\begin{remark}[What is derived and what is stipulated]
\label{rem:attention-status}
It is worth being explicit about the logical status of
Definition~\ref{thm:attention}, since it carries the paper's
attention claim. The \emph{mathematical} content is
Theorem~\ref{thm:plaquette}: the curvature $\Curv_A(x)$ is, exactly
and without approximation, the leading-order holonomy accumulated
around an infinitesimal loop spanning $\xi\wedge\eta$, and by
Proposition~\ref{prop:curvature-structure} it is a pairwise,
antisymmetric, $\adP$-valued form on tangent directions --- the
structural signature of a query--key interaction. That much is
established geometry.

The identification of this object with an attention head is a
\emph{dictionary we stipulate}, not a theorem we prove: nothing here
is derived from the structure of $W_Q^\top W_K$, and we make no
claim that a trained transformer computes $\Curv_A$. What the
dictionary asserts is that curvature is the geometric object
occupying attention's structural role --- determining which pairs of
feature directions interact at $x$, with what strength and what
algebraic content. Earlier drafts stated this as a theorem; the
present formulation is the honest one.
\end{remark}

\begin{remark}[Lattice motivation for the dictionary]
\label{rem:lattice-motivation}
Theorem~\ref{thm:plaquette} expresses
$\Curv_A(x)(\xi\wedge\eta)$ as the leading-order correction to
parallel transport around a small loop tangent to $\xi\wedge\eta$.
Discretising this transport by single-step propagators
$P_\xi,P_\eta\in G$ at $x$ (the lattice plaquette formulation
\cite{Wilson1974}) produces the commutator-like object
$P_\xi P_\eta P_\xi^{-1}P_\eta^{-1}-I$, whose antisymmetrisation in
$\xi\leftrightarrow\eta$ at leading order is
$\partial_\xi\Connection_\eta-\partial_\eta\Connection_\xi+
[\Connection_\xi,\Connection_\eta]=\Curv_A(x)(\xi\wedge\eta)$ as an
$\adP$-valued 2-form. Acting with this endomorphism on a fixed
reference frame $u_x\in E_x$ and taking
$\Re\langle u_x,\Curv_A(x)(\xi\wedge\eta)u_x\rangle_{E_x}$ produces
a real antisymmetric bilinear in $(\xi,\eta)$ --- the scalar
quantity that Definition~\ref{thm:attention} matches to the
antisymmetric component of a single attention head. Any
normalisation constant is set by the discretisation scale of the
propagators and the normalisation of the attention scores. The frame
$u_x$ is the same open-path frame that fixes the Wilson observable's
gauge-covariant prediction (\S\ref{ssec:open-paths}); the dictionary
is therefore frame-dependent and gauge-covariant, not
gauge-invariant. The symmetric part of $\Att$ has no corresponding
bundle-curvature object: it is the symmetric metric content of the
query--key product, distinct from the antisymmetric curvature. This
is motivation for the dictionary, not a derivation of it
(Remark~\ref{rem:attention-status}).
\end{remark}

\begin{remark}[Why ``antisymmetric part of attention'' is the right
match]\label{rem:antisym-attention}
Standard scaled dot-product attention is not antisymmetric in its
two arguments: $\Att(x)(\xi,\eta)\neq -\Att(x)(\eta,\xi)$ in general.
Curvature, by Proposition~\ref{prop:curvature-structure}(ii), \emph{is}
antisymmetric. A literal identification ``$\Curv_A=\Att$'' would
therefore have a sign-error obstruction. The correct statement is
that $\Curv_A$ matches the antisymmetric component
$\Anti[\Att]=\tfrac12(\Att-\Att^\top)$; the symmetric component
$\Sym[\Att]$ is a separate piece of content, geometrically a
metric-like pairing on tangent directions rather than a curvature.
This decomposition is forced by the structure of bilinear forms and
is not a choice the framework makes; transformers compute both
parts, but only the antisymmetric one corresponds to the
classification curvature.
\end{remark}

\begin{remark}[Softmax as discrete holonomy prioritisation]\label{rem:softmax-priority}
The softmax operation $\Att(x)\mapsto\mathrm{softmax}(\Att(x))$
converts the bilinear scores into a probability distribution over
pairs of positions, sharply peaked at the maximising pairs. Through
the dictionary, this is the discrete realisation of the
holonomy-priority function $\pi_A(x)$ of
Definition~\ref{def:holonomy-priority}: continuous ranking of
2-planes by $\|\Curv_A(\Pi)\|^2$ becomes discrete selection of
top-scoring 2-planes by softmax. The temperature of the softmax
controls how sharply the discrete selection approximates the
argmax of $\pi_A$. The framework does not require softmax
specifically; any sharpening operator that picks the top of
$\pi_A$ realises the same geometric content.
\end{remark}

\begin{remark}[Empirical support: Park et al.]\label{rem:park-et-al}
\cite{ParkEtAl2024} provide complementary empirical evidence for the
matter-sector picture of \S\ref{sec:yang-mills-higgs}. They observe
that large language models reorganise their in-context
representations of token identities to minimise a discrete graph
Dirichlet energy
$E_G(H)=\sum_{i,j}A_{ij}\|h_i-h_j\|^2$ on the implicit
data-generating graph $G$, with principal components of the
representations recovering the Laplacian spectral embedding of $G$.
This is the discrete, trivial-bundle, unsupervised specialisation of
the variational principle of \S\ref{sec:yang-mills-higgs}:
substitute (a) a discrete graph for the smooth manifold $M$, (b) the
trivial bundle ($\Connection=0$, so $\Curv_A=0$ and the curvature
sector vanishes), and (c) unsupervised pair sampling for the
labelled data conditions $\phi(x_i)=y_i$, and the covariant Dirichlet
energy $\int_M\|D_A\phi\|^2\,\dvol_g$ reduces to the graph Dirichlet
energy $E_G$. The spectral embedding theorem of
\cite[Theorem~5.1]{ParkEtAl2024} is then the discrete analogue of
the Green's-kernel structure of the harmonic interpolant
\cite[Theorem~13.3]{Vasii2026Paper1}: in both cases the
energy-minimising configuration is built from eigenfunctions of the
relevant Laplacian. The matching of empirical observation to the
framework's matter-sector prediction is direct.
\end{remark}

% ---------------------------------------------------------------------
\subsection{Abelian versus non-abelian: single-head versus multi-head}
\label{ssec:abelian-nonabelian}

The curvature decomposes uniquely as
\[
   \Curv_A \;=\; \underbrace{d\Connection}_{\text{derivative part}}
   \;+\;\underbrace{\tfrac12[\Connection\wedge\Connection]}_{\text{commutator part}}.
\]
The derivative part $d\Connection$ depends only on how the connection
1-form $\Connection$ \emph{varies across $M$}; it is present for every
structure group, abelian or not. The commutator part
$[\Connection\wedge\Connection]$ vanishes identically when
$\mathfrak{g}$ is abelian (every Lie bracket is zero) and is the
irreducible non-abelian content of curvature otherwise.

\begin{remark}[Abelian/non-abelian decomposition of the curvature
2-form; standard]\label{prop:abelian-decomp}
The two pieces of $\Curv_A=d\Connection+\tfrac12[\Connection\wedge\Connection]$
behave differently under the abelian/non-abelian distinction.
\begin{enumerate}[label=(\roman*),leftmargin=2em]
\item If $G$ is abelian ($\mathfrak{g}$ commutative), then
$[\Connection\wedge\Connection]=0$ and $\Curv_A=d\Connection$.
\item If $G$ is non-abelian, then
$[\Connection\wedge\Connection]\neq 0$ in general, and only the
combination $\Curv_A=d\Connection+\tfrac12[\Connection\wedge\Connection]$
transforms gauge-covariantly --- individually, neither $d\Connection$
nor $[\Connection\wedge\Connection]$ is gauge-covariant in the
non-abelian case.
\end{enumerate}
Both facts are standard \cite[Chapter~II.5]{KobayashiNomizu}. The
abelian/non-abelian split of the curvature 2-form is textbook;
\emph{its translation into the single-head/multi-head split of
attention, the dictionary below, is the new content of this
section.}
\end{remark}

\begin{remark}[Multi-head expressivity from non-abelian curvature]
\label{rem:multihead-expressivity}
A consequence of Remark~\ref{prop:abelian-decomp}: single-head
attention can realise only the derivative part $d\Connection$ of a
curvature 2-form. Any classification problem whose YMH minimiser has
a non-trivial commutator $[\Connection\wedge\Connection]\neq 0$ ---
which by Remark~\ref{prop:abelian-decomp} requires non-abelian
$G$, which by the Adams-ladder framing of the introduction
requires a topological obstruction at rung 2 or higher --- is
fundamentally beyond single-head attention. The Adams rung of the
data thus determines a lower bound on the head structure of any
attention mechanism that can realise the optimal classifier
geometrically: rung 0--1 problems are abelian, single-head suffices;
rung 2 problems are $SU(2)$-valued, multi-head with genuine
$[\cdot,\cdot]$-coupling is necessary; rung 3 problems would require
$\mathrm{Spin}(7)$-valued attention, the geometric meaning of which
is open and is the subject of a future paper.
\end{remark}

% ---------------------------------------------------------------------
\subsection{Summary of the dictionary}\label{ssec:attention-summary}

The framework's dictionary between bundle curvature and transformer
attention is summarised by the following correspondences, established
in \S\S\ref{ssec:plaquette}--\ref{ssec:abelian-nonabelian}:

\begin{center}
\renewcommand{\arraystretch}{1.3}
\begin{tabular}{p{0.42\textwidth}p{0.50\textwidth}}
\toprule
\textbf{Bundle-curvature side} & \textbf{Attention side} \\
\midrule
$\Curv_A(x)\in\Omega^2(M,\adP)$, antisym.\ bilinear &
Antisym.\ part $\Anti[\Att(x)]$ of attention bilinear \\
Holonomy-priority $\pi_A(x)$ on $\Gr_2(T_xM)$ &
Softmax-selected top-pairs of attention \\
Infinitesimal holonomy algebra
$\mathfrak{hol}^0_x\subseteq\mathfrak{g}$ &
Effective infinitesimal head structure at position $x$ \\
Derivative part $d\Connection$ &
Single-head content of attention
(abelian $G$; $\sph{2}$ monopole, \S\ref{sec:S2-monopole}) \\
Commutator $[\Connection\wedge\Connection]$ &
Head--head interaction of multi-head attention
(non-abelian $G$; $\sph{4}$ instanton, \S\ref{sec:S4-instanton}) \\
Adams rung of $G$ (\S\ref{sec:adams-intro}) &
Maximum head-structure available globally \\
$G$-invariant norm on $\adP$ &
Magnitude scale of attention scores \\
\bottomrule
\end{tabular}
\end{center}

The dictionary is sharper than the often-quoted equivalence
``curvature is attention'': only the antisymmetric part of attention
corresponds to curvature (Definition~\ref{thm:attention}); the symmetric
part is a separate metric-like content not predicted by the
classification geometry. The dictionary is also entirely a statement
about bundle-curvature objects, in the sense of
Remark~\ref{rem:bundle-not-riemann}: no Riemannian structure on the
base manifold $M$ enters.

% ---------------------------------------------------------------------
\subsection{Synthesis: what the dictionary says about deep learning}
\label{ssec:synthesis}\label{sec:adams-intro}

With the dictionary in hand we can state precisely what the
framework says about transformer architecture. Two structural
claims and one open programme.

\paragraph{Attention is geometry, not statistics.}
Theorem~\ref{thm:plaquette} of \S\ref{ssec:plaquette} establishes
the geometric fact on which the reading rests: the curvature 2-form
$\Curv_A(x)\in\Omega^2(M,\adP)$ of the variationally selected
connection is the leading-order holonomy accumulated around an
infinitesimal loop, and is a pairwise, antisymmetric, algebra-valued
form on tangent directions. Definition~\ref{thm:attention} then
matches this object to the antisymmetric part of a single-head
attention bilinear --- a dictionary we stipulate rather than derive
(Remark~\ref{rem:attention-status}). The holonomy-priority function
$\pi_A(x)$ on the
Grassmannian of 2-planes at $x$ corresponds to the softmax-selected
top-pairs of attention. The \emph{structural role} of attention ---
which directions interact at $x$, how strongly, with what
algebraic content --- is fixed by the bundle's curvature, a
geometric object the framework predicts before any training. The
empirical observation of \cite{ParkEtAl2024} that large language
models reorganise their in-context representations to minimise a
graph Dirichlet energy is the trivial-bundle, unsupervised
specialisation of this picture (Remark~\ref{rem:park-et-al}).
What is established here is the dictionary; what is open --- and
empirically testable on trained transformers --- is whether the
discretisation is approximate or exact for any given
architecture.

\paragraph{The worked examples form rungs of an Adams ladder.}
The example sequence of Part~2 has a graded structure:
linear classifiers (no attention, contractible base; Paper~1)
sit at rung~0 ($\R$, abelian, trivial); single-head attention with
$G=U(1)$ sits at rung~1 ($\C$, abelian curvature, the $\sph{2}$
monopole of \S\ref{sec:S2-monopole}); multi-head attention with
$G=SU(2)$ sits at rung~2 ($\HH$, non-abelian curvature with
non-trivial commutator content
$[\Connection\wedge\Connection]\neq 0$, the $\sph{4}$ instanton
of \S\ref{sec:S4-instanton}). The rungs handle different topology:
rung~0 fails on $\sph{1}$ M\"obius and on $T^2$ XOR (both
require rung~1's $\Z_2$-bundles, the limit of rung~1 at the
abelian-discrete boundary); rung~1 fails on $\sph{2}$ monopole
classifications (no abelian connection carries the forced $c_1$);
rung~2 fails on $\sph{4}$ instanton classifications without the
full $SU(2)$ multi-head content
(Remarks~\ref{prop:abelian-decomp},
\ref{rem:multihead-expressivity}). The position-dependent
refinement of Remark~\ref{rem:position-dependent-heads} sharpens
this: the \emph{effective} head structure at $x$ is set by the
infinitesimal holonomy algebra
$\mathfrak{hol}^0_x\subseteq\mathfrak{g}$ of the optimal connection
at $x$. For the $\sph{4}$ instanton this algebra is
$\mathfrak{su}(2)$ at every point where the curvature is non-zero;
what varies across the data manifold is the attention
\emph{intensity} $\pi_A(x)=\|\Curv_A(x)\|^2$, sharply localised at
the instanton core with width $\lambda\sim d$
(Remark~\ref{rem:S4-localised-attention}).

What this paper establishes about the ladder: the three worked
examples ($\sph{1}$, $T^2$, $\sph{2}$, $\sph{4}$) instantiate
rungs~0--2 concretely, and the dictionary of \S\ref{ssec:attention-summary}
identifies their structural content with the corresponding
attention architecture. What this paper does \emph{not} establish:
that the rungs form a strict monotone hierarchy in general, that
rung~$k$ is the minimal architecture for problems with rung~$k$
topology, or what the rung-3 octonionic case looks like. These
are the subject of a separate forthcoming paper. The Adams-ladder
framing is invoked here as an organising heuristic for the
example sequence, not as an established theorem about
classification architectures.

% =====================================================================
\part*{Part 3. The Larger Picture}
\addcontentsline{toc}{part}{Part 3. The Larger Picture}

% ---------------------------------------------------------------------
\section{The torus: classification without a boundary}
\label{sec:torus}

The torus $T^2=\sph{1}\times\sph{1}$ is the canonical
two-dimensional example of holonomy-only classification: a tier-1
case where flat connections suffice ($H^2(T^2,\Z)=\Z$ admits flat
representatives via $c_1=0$ choices), but $H^1(T^2,\Z_2)=\Z_2\times\Z_2$
supports four distinct flat bundles, and the XOR labelling pattern
selects the most non-trivial one. The example completes the tier-1
section of the framework by showing how a two-dimensional base with
$H^1\neq 0$ produces a richer bundle moduli than $\sph{1}$, and how a
single non-trivial bundle realises the classical XOR pattern of
Minsky--Papert as a flat-bundle topological selection, not as a
nonlinearity in the classifier.

\paragraph{Setup.} Take $M=T^2$, parametrised by
$(\alpha,\beta)\in[0,L)\times[0,L)$ with the flat metric and total
side length $L$ in each direction. The relevant cohomology is
$\pi_1(T^2)=\Z\times\Z$ (generated by loops $e_1$ around the
$\alpha$-circle and $e_2$ around the $\beta$-circle), and
\[
   H^1(T^2,\Z_2) \;=\; \operatorname{Hom}(\Z\times\Z,\Z_2) \;=\; \Z_2\times\Z_2,
\]
with four flat $\Z_2$-bundles indexed by holonomy pairs
$(\rho(e_1),\rho(e_2))\in\Z_2\times\Z_2$:

\begin{center}
\renewcommand{\arraystretch}{1.2}
\begin{tabular}{lcl}
\toprule
\textbf{Bundle} & \textbf{$(\rho(e_1),\rho(e_2))$} & \textbf{Behaviour} \\
\midrule
Trivial $L_0$ & $(+1,+1)$ & periodic in both directions \\
M\"obius-$\alpha$, $L_{\sigma_1}$ & $(-1,+1)$ & antiperiodic in $\alpha$, periodic in $\beta$ \\
M\"obius-$\beta$, $L_{\sigma_2}$ & $(+1,-1)$ & periodic in $\alpha$, antiperiodic in $\beta$ \\
Double M\"obius, $L_{\sigma_1\sigma_2}$ & $(-1,-1)$ & antiperiodic in both \\
\bottomrule
\end{tabular}
\end{center}

All four bundles are flat (the topological constraint of
\S\ref{ssec:ymh-functional} is $\Curv_A=0$), so the entire energy
is the matter sector.

\paragraph{The XOR dataset.} Place four labelled points at
$(L/4,L/4)$, $(3L/4,L/4)$, $(L/4,3L/4)$, $(3L/4,3L/4)$ with labels
\[
   y\bigl(\tfrac{L}{4},\tfrac{L}{4}\bigr) = +1,\qquad
   y\bigl(\tfrac{3L}{4},\tfrac{L}{4}\bigr) = -1,\qquad
   y\bigl(\tfrac{L}{4},\tfrac{3L}{4}\bigr) = -1,\qquad
   y\bigl(\tfrac{3L}{4},\tfrac{3L}{4}\bigr) = +1.
\]
The pattern is XOR: diagonal pairs share labels, adjacent pairs
disagree. This is the classical Minsky--Papert dataset showing that
linear classifiers on $\R^2$ fail on XOR; the framework's analysis
on $T^2$ reveals the obstruction is topological, not architectural.

\paragraph{The variational selector picks the double M\"obius
bundle.} On each of the four flat bundles, the matter-sector
minimum subject to the four data conditions is bounded below by the
covariant Dirichlet energy of the unique covariantly harmonic
interpolant. We compare them.

\emph{Trivial bundle $L_0$.} A section is doubly periodic. Any
constant section fails to meet mixed-sign data, so the minimiser is
non-constant. The Fourier decomposition into eigenfunctions of
$\Delta_{T^2}=-(\partial_\alpha^2+\partial_\beta^2)$ has lowest
eigenvalue zero (the constant mode), so the unregularised problem is
degenerate in this sector and the regularisation $\kappa>0$ of
Remark~\ref{rem:ymh-attainment} is essential here. The minimiser has
no clean closed form, due to the mismatch between the periodic
boundary conditions and the antiperiodic structure XOR requires; it
is computed numerically below.

\begin{remark}[Scale invariance on $T^2$]
\label{rem:t2-scale}
In dimension $n=2$ the Dirichlet energy $\int|\nabla\phi|^2$ is
scale-invariant, so no unregularised matter energy on $T^2$ can carry
$L$-dependence: any claim of the form $E_{\min}\sim L^{-2}$ is
dimensionally impossible. The regularised energies below do depend
on $L$, but only through the dimensionless combination $\kappa L$.
\end{remark}

\emph{Single-M\"obius bundles $L_{\sigma_1}, L_{\sigma_2}$.} A
section antiperiodic in $\alpha$ (resp.\ $\beta$) absorbs half of
the XOR pattern: it produces an automatic sign flip across the
$\alpha=L/2$ line (resp.\ $\beta=L/2$ line), but not in the
orthogonal direction. The matter-sector minimum is strictly less
than $E_{\min}(L_0)$ but strictly greater than the double-M\"obius
minimum below: the antiperiodicity contributes only one of the two
required sign reversals.

\emph{Double M\"obius bundle $L_{\sigma_1\sigma_2}$.} A section is
antiperiodic in both directions, $\tilde\phi(\alpha+L,\beta)=
-\tilde\phi(\alpha,\beta)$ and $\tilde\phi(\alpha,\beta+L)=
-\tilde\phi(\alpha,\beta)$. The simplest covariantly harmonic
candidate is
\[
   \tilde\phi(\alpha,\beta) \;=\;
   \cos\!\bigl(\tfrac{\pi\alpha}{L}\bigr)\,
   \cos\!\bigl(\tfrac{\pi\beta}{L}\bigr),
\]
a doubly-antiperiodic eigenfunction of $\Delta$ on the cover with
eigenvalue $2\pi^2/L^2$ (the lowest antiperiodic-antiperiodic mode).
Evaluating at the four data points:
\begin{align*}
   \tilde\phi(\tfrac{L}{4},\tfrac{L}{4})
   &= \cos\tfrac{\pi}{4}\cdot\cos\tfrac{\pi}{4} = \tfrac{1}{2}, \\
   \tilde\phi(\tfrac{3L}{4},\tfrac{L}{4})
   &= \cos\tfrac{3\pi}{4}\cdot\cos\tfrac{\pi}{4} = -\tfrac{1}{2}, \\
   \tilde\phi(\tfrac{L}{4},\tfrac{3L}{4})
   &= \cos\tfrac{\pi}{4}\cdot\cos\tfrac{3\pi}{4} = -\tfrac{1}{2}, \\
   \tilde\phi(\tfrac{3L}{4},\tfrac{3L}{4})
   &= \cos\tfrac{3\pi}{4}\cdot\cos\tfrac{3\pi}{4} = \tfrac{1}{2}.
\end{align*}
The signs match the XOR labels exactly. Up to the overall
amplitude $r$ fixing $\tilde\phi(x_i)=y_i\cdot r$, the closed-form
section $\tilde\phi$ is the matter-energy minimiser on
$L_{\sigma_1\sigma_2}$ (it is the unique covariantly harmonic
section in the lowest antiperiodic-antiperiodic mode meeting the
data). Its matter energy is
\[
   E_{\min}(L_{\sigma_1\sigma_2}) \;=\;
   \int_{T^2}\|d\tilde\phi\|^2\,d\alpha\,d\beta
   \;=\; \frac{2\pi^2}{L^2}\,\cdot\,
   \int_{T^2}\tilde\phi^2\,d\alpha\,d\beta
   \;=\; \frac{2\pi^2}{L^2}\cdot\frac{L^2}{4}
   \;=\; \frac{\pi^2}{2},
\]
where we used $\int_{T^2}\cos^2(\pi\alpha/L)\cos^2(\pi\beta/L)
\,d\alpha\,d\beta = L^2/4$.

\paragraph{Selection: the computed comparison.} The energy
$\pi^2/2$ above is the Dirichlet energy \emph{of the lowest
antiperiodic mode}, not the minimum of the constrained variational
problem: in $n=2$ points have zero $H^1$-capacity, so the
unregularised infimum vanishes on \emph{every} one of the four
bundles (Remark~\ref{rem:ymh-attainment}), and the comparison must
be made in the regularised theory. We therefore solve the covariant
Matérn capacitance system of
Lemma~\ref{lem:capacitance-reduction} in each sector: the Green's
function of $(\Delta+\kappa^2)^\nu$ on the bundle
$(\rho(e_1),\rho(e_2))$ is the lattice sum
\[
   G_{(s_1,s_2)}(x)\;=\;\frac{1}{L^2}\sum_{m,n\in\Z}
   \frac{e^{\,2\pi i\left((m+\frac{s_1}{2})\alpha
   +(n+\frac{s_2}{2})\beta\right)/L}}
   {\bigl(\lambda_{mn}+\kappa^2\bigr)^{\nu}},
   \qquad
   \lambda_{mn}=\Bigl(\tfrac{2\pi}{L}\Bigr)^{2}
   \Bigl[\bigl(m+\tfrac{s_1}{2}\bigr)^2+\bigl(n+\tfrac{s_2}{2}\bigr)^2\Bigr],
\]
with $s_i=0$ for periodic and $s_i=1$ for antiperiodic, and evaluate
$E_{\min}=y^\top K^{-1}y$, $K_{ij}=G(x_i-x_j)$. For the square XOR
configuration on $L=2\pi$ the results are:

\begin{center}
\begin{tabular}{lrrrr}
\toprule
& \multicolumn{2}{c}{$\nu=2$} & \multicolumn{2}{c}{$\nu=3$}\\
\cmidrule(lr){2-3}\cmidrule(lr){4-5}
Bundle & $\kappa=0.5$ & $\kappa=1$ & $\kappa=0.5$ & $\kappa=1$\\
\midrule
Trivial $L_0$              & 43.12 & 70.80 & 109.63 & 253.80\\
M\"obius-$\alpha$ $L_{\sigma_1}$ & 33.27 & 63.09 & 60.73 & 187.69\\
M\"obius-$\beta$ $L_{\sigma_2}$  & 33.27 & 63.09 & 60.73 & 187.69\\
Double M\"obius $L_{\sigma_1\sigma_2}$ & \textbf{17.70} & \textbf{51.43}
& \textbf{15.86} & \textbf{109.96}\\
\bottomrule
\end{tabular}
\end{center}

The ordering
\[
   E_{\min}(L_{\sigma_1\sigma_2})
   \;<\; E_{\min}(L_{\sigma_1})=E_{\min}(L_{\sigma_2})
   \;<\; E_{\min}(L_0)
\]
holds at every $(\nu,\kappa)$ tested, and the double-M\"obius bundle
is strictly lowest for every rectangle
$(a,b)\in(0,L/2)^2$ we examined; the exact degeneracy
$E_{\min}(L_{\sigma_1})=E_{\min}(L_{\sigma_2})$ occurs precisely at
the $\alpha\leftrightarrow\beta$-symmetric configurations $a=b$ and
splits otherwise, as Remark~\ref{rem:T2-bundle-stability} predicts.
The variational selector picks the double-M\"obius bundle:
$\xi^*=\{(\sigma_1,\sigma_2)\}$ and $[c_y]=(\sigma_1,\sigma_2)\in
H^1(T^2,\Z_2)$. The absolute values depend on the regularisation
$(\nu,\kappa)$, as they must; the \emph{ordering}, and hence the
bundle selection, does not. The XOR pattern is thus a tier-1
classification problem with a non-trivial $\Z_2\times\Z_2$
selection on a base where the trivial bundle would fail by
periodicity.

\paragraph{Generalising the data geometry: asymmetric rectangles.}
The four-point XOR computation above used a symmetric rectangle
inscribed in the fundamental domain at $(L/4,L/4),(3L/4,L/4),
(L/4,3L/4),(3L/4,3L/4)$. The same analysis extends to any rectangle
of the form
\[
   x_1=(a,b),\quad x_2=(L-a,b),\quad x_3=(a,L-b),\quad x_4=(L-a,L-b),
\]
with $a,b\in(0,L/2)$, opposite-label pairs along horizontals and
verticals, XOR labels $(+,-,-,+)$ across the diagonal. The
$\Z_2\times\Z_2$ reflection symmetry of $T^2$ identifies any
configuration with $a>L/2$ to one with $L-a<L/2$ (with relabelled
points), so this range exhausts the essentially distinct
configurations.

The lowest doubly-antiperiodic mode
$\cos(\pi\alpha/L)\cos(\pi\beta/L)$ evaluated at the four data
points gives
\[
   \tilde\phi(a,b)=c_a c_b,\quad
   \tilde\phi(L-a,b)=-c_a c_b,\quad
   \tilde\phi(a,L-b)=-c_a c_b,\quad
   \tilde\phi(L-a,L-b)=c_a c_b,
\]
with $c_a:=\cos(\pi a/L)$ and $c_b:=\cos(\pi b/L)$. The signs
$(+,-,-,+)$ match XOR for every $(a,b)\in(0,L/2)\times(0,L/2)$ ---
\emph{the bundle selection is independent of the rectangle's
dimensions}. The amplitude matching $\tilde\phi(x_i)=y_i$ requires
rescaling by $1/(c_a c_b)$, giving the unique lowest-mode minimiser
\[
   \tilde\phi(\alpha,\beta) \;=\;
   \frac{\cos(\pi\alpha/L)\cos(\pi\beta/L)}{c_a\,c_b}.
\]
The matter-sector energy is
\[
   \boxed{\;\;E_{\min}(L_{\sigma_1\sigma_2})(a,b)
   \;=\; \frac{2\pi^2}{L^2}\cdot
   \int_{T^2}\tilde\phi^2\,d\alpha\,d\beta
   \;=\; \frac{\pi^2}{2\,\cos^2(\pi a/L)\,\cos^2(\pi b/L)}.\;\;}
\]
This is the closed-form energy as a function of the rectangle's
geometry, and the central object for the proximity analysis below.

\paragraph{The two regimes: far separation and proximity divergence.}
The factorised energy $E_{\min}\propto 1/(c_a^2 c_b^2)$ exhibits
two distinct limiting behaviours, corresponding to the two natural
geometric extremes.

\emph{Regime A: far separation $\Leftrightarrow$ minimum energy.}
As $a\to 0^+$ and $b\to 0^+$, the four points cluster toward the
corners of the fundamental domain: $x_1\to(0,0)$, $x_2\to(L,0)\sim
(0,0)$ through the torus identification, and so on. The
\emph{opposite-label} pairs sit at the maximum geodesic separation
on $T^2$ along the shortest path through the universal cover ---
i.e., the labels $\pm$ at $x_1, x_2$ are along the diagonal of a
fundamental domain. The cosines tend to $1$ and
\[
   E_{\min}(L_{\sigma_1\sigma_2}) \;\longrightarrow\; \frac{\pi^2}{2}
   \qquad \text{as } a,b\to 0^+.
\]
This is the framework's prediction at the easy end of the XOR
configuration: opposite labels far apart, the bundle absorbs the
sign change at no concentrated cost, energy at its global minimum.
The single mode $\cos(\pi\alpha/L)\cos(\pi\beta/L)$ realises XOR
across the entire fundamental domain with the slowest possible
gradient.

\emph{Regime B: proximity divergence.} As $a\to (L/2)^-$ (with
$b$ fixed in the interior), $c_a=\cos(\pi a/L)\to\sin(\pi(L/2-a)/L)
\to 0$ linearly in the small parameter $\epsilon:=L/2-a$. The data
points $x_1, x_2$ collapse toward $(L/2,b)$ from opposite sides
along $\alpha$, with horizontal opposite-label separation
$2\epsilon\to 0$. Since $c_a\sim\pi\epsilon/L$, the energy scales
as
\[
   E_{\min}(L_{\sigma_1\sigma_2}) \;\sim\;
   \frac{\pi^2}{2 c_b^2}\cdot\frac{L^2}{\pi^2\epsilon^2}
   \;=\; \frac{L^2}{2 c_b^2}\cdot\frac{1}{\epsilon^2}
   \;\longrightarrow\;\infty.
\]
The matter energy diverges as $1/\epsilon^2$, with $\epsilon$ the
half-separation of the collapsing opposite-label pair. The
analogous limit $b\to(L/2)^-$ gives a $1/(L/2-b)^2$ divergence; the
joint limit gives the product $1/(\epsilon_a^2\epsilon_b^2)$. The
two directional divergences are independent, reflecting the
separability of the lowest mode.

The structural reading: \emph{classification difficulty registers
as matter energy, with proximity of opposite labels the dominant
geometric cost.} On $T^2$ the cost is set by the closer of the two
opposite-label separations, with the precise scaling $1/d^2$ per
collapsing direction. This is the two-dimensional analogue of the
$S^1$ M\"obius scaling $E_{\min}\sim 4/d$ (\S\ref{ssec:ymh-flat}).
The exponent changes from $1/d$ on $S^1$ to $1/d^2$ on $T^2$
because the dominant mode is one-dimensional ($\cos(\pi\alpha/L)$
varies in one direction) but the energy integrates a squared
gradient over the full two-dimensional volume; explicitly,
$\int_{T^2}|d\tilde\phi|^2$ is proportional to the squared
amplitude $1/(c_a c_b)^2 \sim 1/\epsilon^2$ for fixed mode shape.

\begin{remark}[Closer opposite labels means higher energy: the
proximity principle on $T^2$]\label{rem:T2-proximity}
The formula $E_{\min}=\pi^2/(2 c_a^2 c_b^2)$ makes precise the
intuitive principle that \emph{opposite labels far apart are easy
to classify; opposite labels close together are hard}. Far apart
$\Leftrightarrow$ $a,b$ small $\Leftrightarrow$ $c_a,c_b$ near $1$
$\Leftrightarrow$ $E_{\min}\to \pi^2/2$ (the global minimum). Close
together $\Leftrightarrow$ $a$ or $b$ near $L/2$
$\Leftrightarrow$ $c_a$ or $c_b$ near $0$ $\Leftrightarrow$
$E_{\min}\to\infty$. The framework registers classification
difficulty as energy of the geometrically optimal classifier in the
geometrically selected bundle, with no statistical loss or
softness; proximity is the geometric obstruction, not a
hyperparameter. This is the canonical demonstration on $T^2$ of the
proximity-divergence principle: hard data does not produce a
``softer'' decision boundary, it produces a higher-energy
covariantly harmonic section, and as opposite labels approach the
section's amplitude (and hence its matter energy) diverges.
\end{remark}

\begin{remark}[Why the bundle does not change with $(a,b)$]
\label{rem:T2-bundle-stability}
A natural question is whether moving the data geometry can change
the variationally selected bundle. On the symmetric configurations
above the answer is no: $L_{\sigma_1\sigma_2}$ remains the selector
throughout $(a,b)\in(0,L/2)^2$, because the XOR sign pattern
$(+,-,-,+)$ is realised by the lowest doubly-antiperiodic mode for
\emph{every} such rectangle, and no other bundle class can match
the pattern with a lower-mode minimiser. The bundle selection
$[c_y]=(\sigma_1,\sigma_2)$ is therefore a discrete invariant of
the XOR labelling pattern, stable under continuous deformation of
the data positions; only the energy varies with $(a,b)$. This is
the structural content of bundle selection: the discrete topology
of the labelling rules the bundle class, the continuous geometry
of the data rules the matter energy within that class.
\end{remark}

\paragraph{The decision ``boundary''.} The zero set of $\tilde\phi$
on the universal cover is
\[
   \tilde\phi^{-1}(0)
   \;=\;
   \{\alpha=L/2\}\cup\{\beta=L/2\}\cup\{\text{translates by }L\},
\]
two orthogonal great circles on $T^2$ dividing it into four
quadrants of alternating signs --- exactly the XOR decision
geometry. The locus descends to $T^2$ as two crossing circles, but
\emph{this is not a sign-function decision boundary on $T^2$}: the
classifier is a section of $L_{\sigma_1\sigma_2}$, not a function
on $T^2$, and the ``sign'' of $\tilde\phi$ is well-defined only
relative to a chosen path from a reference point. The framework
predicts the classifier output via the Wilson observable of
\S\ref{ssec:open-paths} as the sign of
$\Hol_\gamma\cdot\tilde\phi(\tilde x)$ along a chosen path $\gamma$
to the query point $x$.

\paragraph{Prediction by holonomy arithmetic.} For a path $\gamma$
on $T^2$ with winding numbers $(n_1,n_2)\in\Z\times\Z$ around the
two generating circles, the double-M\"obius holonomy is
\[
   \Hol_\gamma(\Connection_{\sigma_1\sigma_2}) \;=\;
   (-1)^{n_1}\cdot(-1)^{n_2} \;=\; (-1)^{n_1+n_2} \;\in\;\Z_2.
\]
The prediction at $x\in T^2$ from a labelled point $x_+$ along
$\gamma$ is $y_+\cdot(-1)^{n_1(\gamma)+n_2(\gamma)}$. For the
canonical shortest-path choice (no winding), this matches the
lifted sign $\mathrm{sign}\,\tilde\phi$. The framework reports the
prediction along with the chosen path; ambiguities arise only at
$\tilde\phi^{-1}(0)$, the lifted boundary.

\paragraph{XOR is topological, not architectural.} The classical
reading of XOR is that it requires a nonlinear classifier on
$\R^2$. The framework's reading is structurally different: XOR on
$T^2$ is solved by a \emph{linear} section (a single Fourier mode
$\cos(\pi\alpha/L)\cos(\pi\beta/L)$) on a \emph{topologically
non-trivial} bundle. The ``nonlinearity'' is absorbed into the
bundle's antiperiodic structure; the classifier itself is the
simplest possible covariantly harmonic section. The contrast with
XOR on $\R^2$ is sharp:

\begin{center}
\renewcommand{\arraystretch}{1.2}
\begin{tabular}{lll}
\toprule
& \textbf{XOR on $\R^2$} & \textbf{XOR on $T^2$} \\
\midrule
$\pi_1$ & $0$ & $\Z\times\Z$ \\
$[c_y]\in H^1$ & $0$ & $(\sigma_1,\sigma_2)$ \\
Bundle & trivial & double M\"obius \\
Classifier & nonlinear function & linear section in cover \\
Decision boundary & in $M$ (the two diagonals) & lifted from cover \\
Prediction & evaluate $f(x)$ & section sign with holonomy \\
Adams rung & 0 (Paper~1) & 0 (this section) \\
\bottomrule
\end{tabular}
\end{center}

XOR is the obstruction to writing $\tilde\phi$ as a periodic
function on $T^2$. It is \emph{not} an obstruction on the universal
cover, where $\cos(\pi\alpha/L)\cos(\pi\beta/L)$ solves it
exactly. The framework recasts the apparent ``nonlinearity'' as the
topology of the bundle, not the complexity of the classifier.

\paragraph{Path consistency: the abelian flat check.} On the flat
double-M\"obius bundle, the holonomy of every closed loop is
determined by its homotopy class in $\pi_1(T^2)=\Z\times\Z$ alone:
\[
   \Hol_\gamma \;=\; (-1)^{n_1+n_2},
\]
with no curvature contribution. The path-consistency check of
\S\ref{sec:S2-monopole}'s monopole and \S\ref{sec:S4-instanton}'s
instanton reduces here to a topological check: for any two paths
$\gamma_1,\gamma_2$ from $x_+$ to $x$, the holonomy quotient is
$(-1)^{(n_1(\gamma_2)-n_1(\gamma_1))+(n_2(\gamma_2)-n_2(\gamma_1))}$,
where $n_i(\gamma_2)-n_i(\gamma_1)$ is the winding of the closed
loop $\gamma_2*\gamma_1^{-1}$ around the $i$-th generator. The check
is that the framework's prediction agrees with the labels along
every reference path consistent with this homotopy-only rule. This
is a discrete combinatorial check rather than the numerical path
integration of \S\ref{sec:S2-monopole} and \S\ref{sec:S4-instanton}.

\begin{remark}[Hypothetical secondary obstruction in $H^2(T^2,\Z_2)$]
\label{rem:T2-secondary}
A fifth data point inconsistent with the double-M\"obius prediction
would force a non-trivial component of $[c_y]$ in
$H^2(T^2,\Z_2)=\Z_2$, activating tier-2 obstructions and forcing
curvature. The torus then becomes a test case for the $n=2$ branch
of Conjecture~\ref{conj:proximity-curvature}: the matter current
$J(\phi)$ deforms the otherwise-zero curvature, and the conjecture
predicts the deformation stays bounded ---  no $1/d^2$ peak, no
$\lambda=0$ vortex localisation, since $n=2$ has no conformal
scale-invariant family analogous to BPST. The structural picture
parallels the three-point monopole case of
\S\ref{sec:S2-monopole}'s symmetry-breaking paragraph. The full
case is open and not worked here.
\end{remark}

\begin{remark}[XOR and architectural prejudices]
\label{rem:xor-architectural}
The Minsky--Papert critique of single-layer perceptrons rests on
the impossibility of realising XOR by a linear classifier on
$\R^2$. The framework's $T^2$ analysis shows the same XOR pattern
is realised by a \emph{linear} (single-Fourier-mode) classifier on
a topologically non-trivial bundle. The contrast is geometric: on
$\R^2$ the bundle is forced to be trivial by contractibility, and
the classifier must absorb the XOR ``twist'' as nonlinearity; on
$T^2$ the bundle's $\Z_2\times\Z_2$-twist absorbs the XOR pattern
and the classifier becomes the simplest possible covariantly
harmonic mode. The framework predicts that what appears as
nonlinearity in a function on a contractible domain may, on a
geometrically appropriate non-contractible domain, become linear
sectional content on a non-trivial bundle. The structural reading
is the ``XOR is topological, not architectural'' message of the
introduction (\S\ref{sec:intro}).
\end{remark}

% =====================================================================
\section{Numerical experiment: torus XOR and comparison with a trained MLP}
\label{sec:numerical-torus}

This section provides the paper's computational anchor. Paper~1
\cite[\S 14]{Vasii2026Paper1} closed with a numerical comparison
between the framework's harmonic interpolation and gradient-descent
training of a multilayer perceptron on the two-moons dataset,
demonstrating a $500\times$ speedup of the linear solve over backprop.
We perform the analogous comparison on the torus, where the
framework's prediction is a closed-form section on the
double-M\"obius bundle (\S\ref{sec:torus}) and the alternative is a
small MLP trained on the same four labelled points viewed in their
$\R^2$-Cartesian embedding.

\paragraph{Setup.} The labelled dataset is the XOR configuration of
\S\ref{sec:torus}:
\[
   D = \{(L/4,L/4,+1),\,(3L/4,L/4,-1),\,(L/4,3L/4,-1),\,(3L/4,3L/4,+1)\}
\]
with $L=2\pi$. The framework's predicted classifier is the section
\[
   \tilde\phi(\alpha,\beta) \;=\;
   2\cos\!\bigl(\tfrac{\pi\alpha}{L}\bigr)\,
   \cos\!\bigl(\tfrac{\pi\beta}{L}\bigr),
\]
covariantly harmonic on the double-M\"obius bundle, with amplitude
$c=2$ chosen so $|\tilde\phi(x_i)|=1$ at the data points. The
prediction is constructed in closed form from
\S\ref{sec:torus} with no iteration. The alternative classifier is
a multilayer perceptron with architecture $2\to 8\to 8\to 1$, tanh
activations, mean-squared-error loss, and plain gradient descent
with learning rate $0.05$ for $2\times 10^4$ epochs.
The MLP sees only the $(\alpha,\beta)$-coordinates of its training
points and their labels, with no information about the toroidal
topology of the domain.

\paragraph{Quantitative check on the closed-form mode.} The
Dirichlet energy of the lowest antiperiodic mode of
\S\ref{sec:torus} (with amplitude $c=2$) is
$c^2\cdot\pi^2/2=2\pi^2\approx 19.7392$. Numerical integration of
$\int_{T^2}\|D_A\tilde\phi\|^2\,\dvol$ on a $512\times 512$ grid of
$T^2$ yields $E_{\mathrm{num}}=19.7392$, matching to machine
precision (relative error $<10^{-15}$). We stress what this does and
does not establish: it confirms the \emph{integral} $\int\|D_A
\tilde\phi\|^2$ for the stated section, and nothing more. It is not
a verification of minimality --- the constrained minimum is the
regularised capacitance value computed in \S\ref{sec:torus}, which
depends on $(\nu,\kappa)$ and is a different number. The section
used in the experiments below is the closed-form lowest mode, chosen
because it is explicit and realises the XOR sign pattern exactly;
the bundle \emph{selection} rests on the capacitance comparison, not
on this integral.

\subsection{Experiment 1: four labelled points}
\label{ssec:experiment-1}

We first compare both classifiers on the minimal XOR dataset $D$.
Figure~\ref{fig:torus-xor} shows the result. Both fit the four
labelled points exactly (the MLP reaches MSE $<10^{-30}$ by epoch
$5000$). They differ structurally.

\begin{figure}[ht]
\centering
\includegraphics[width=\textwidth]{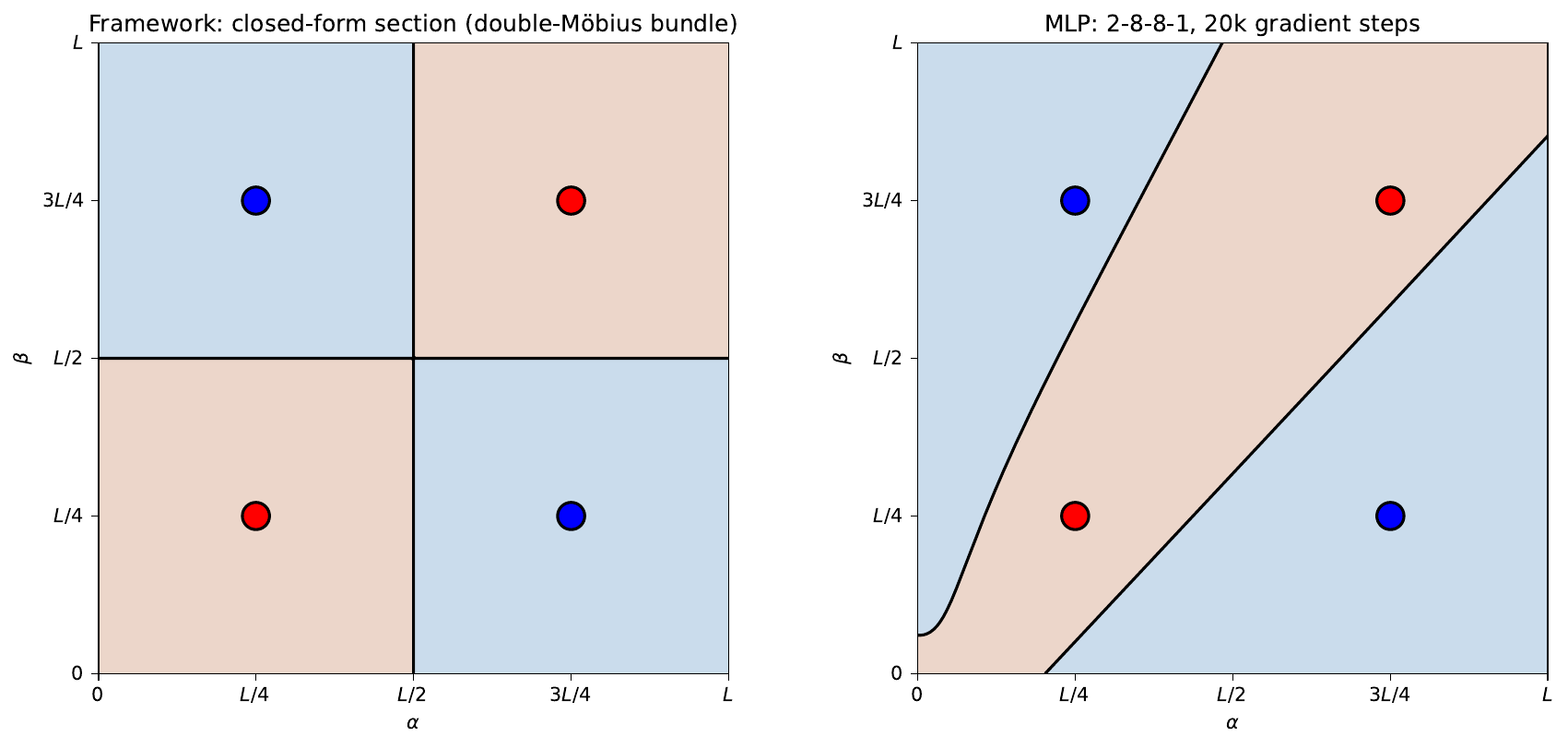}
\caption{Experiment 1, four labelled points.
Left: the framework's closed-form section on the
double-M\"obius bundle, with the predicted decision boundary
$\{\alpha=L/2\}\cup\{\beta=L/2\}$ as the exact perpendicular cross.
Right: an MLP $2$--$8$--$8$--$1$ with tanh activations, MSE loss,
trained to MSE $<10^{-30}$ on the same four points. Both fit the
data; only the framework respects the toroidal geometry.}
\label{fig:torus-xor}
\end{figure}

\begin{itemize}[leftmargin=2em]
\item \emph{Framework boundary (left).} The zero set of $\tilde\phi$
is the union of two great circles, $\{\alpha=L/2\}$ and
$\{\beta=L/2\}$, dividing $T^2$ into four quadrants of alternating
signs. The boundary is geometric: it comes from the lowest
antiperiodic eigenmode of the Laplacian on the double cover,
forced by the $\Z_2\times\Z_2$ topology of the bundle.

\item \emph{MLP boundary (right).} Two arcs connecting like-labelled
corners, with no geometric significance: the boundary is wherever
gradient descent happened to terminate from the chosen
initialisation. Sampling the MLP along the framework-predicted cross
$\{\alpha=L/2\}\cup\{\beta=L/2\}$ gives values with mean
$|\tilde\phi_{\mathrm{MLP}}|\approx 0.82$ and maximum
$\approx 1.39$ --- the MLP's outputs there are far from the framework's
exact $0$, confirming the MLP boundary is structurally elsewhere.
The MLP also does not respect the toroidal identifications: its
output is discontinuous across the seams $\alpha=0,L$ and
$\beta=0,L$ when viewed on $T^2$.
\end{itemize}

Four points underdetermine the boundary between them; Experiment~2
raises the data density to 1000 points to test whether the MLP then
converges to the framework's boundary --- it does not.

\subsection{Experiment 2: 1000 perturbed points, four random
initialisations}
\label{ssec:experiment-2}

We replace the four labelled centres by 1000 labelled points: 250
points sampled from a Gaussian of standard deviation
$\sigma=0.05L$ around each of the four XOR centres, with hard
labels inherited from the centre. The MLP architecture, learning
rate, and epoch count are unchanged. We train four MLPs from four
different random initialisations (seeds $1000$--$1003$) on the
\emph{same} 1000-point dataset, so any variation across the four
MLPs is due to initialisation only.

\begin{figure}[ht]
\centering
\includegraphics[width=\textwidth]{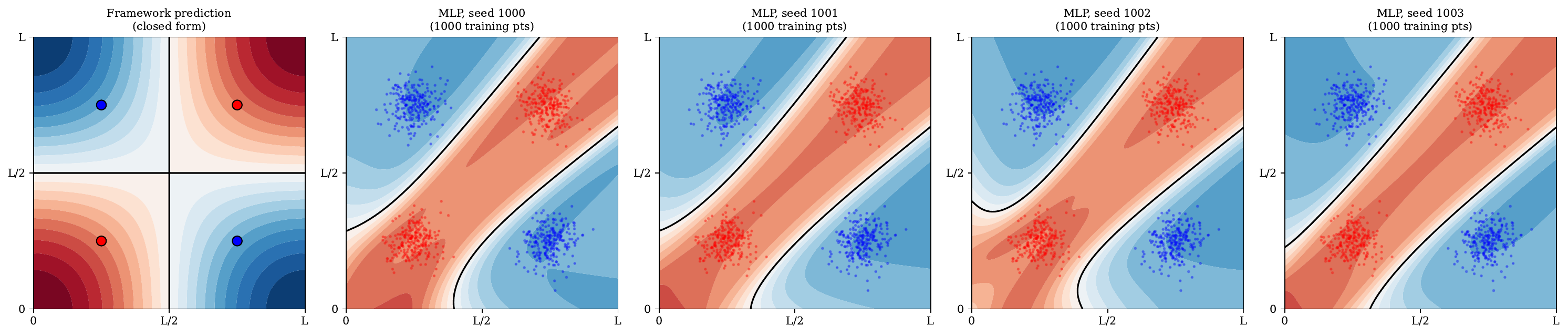}
\caption{Experiment 2, 1000 perturbed points (4 random
initialisations). Left: the framework's closed-form section,
identical to Figure~\ref{fig:torus-xor} (left) and computed without
training. Right: four MLPs trained to convergence on the same
1000-point dataset, varying only the initialisation. All four reach
MSE $\sim 10^{-4}$ on the training data and produce decision
boundaries that are roughly stable across seeds but structurally
distinct from the framework's cross.}
\label{fig:torus-perturbed}
\end{figure}

The four MLPs find broadly similar boundaries: two roughly diagonal
arcs separating the like-coloured corners from each other, with
minor differences across seeds (the arcs' exact curvature in the
central region varies). The boundary is more stable than in
Experiment 1, where the four-point dataset left the boundary
essentially unconstrained between data points; with $1000$ points
the data density is high enough that gradient descent converges to
a similar family of solutions from different starting weights.

But the family of solutions is not the framework's solution. Along
the framework-predicted cross $\{\alpha=L/2\}\cup\{\beta=L/2\}$,
where the framework's section vanishes identically by construction,
the four trained MLPs give mean $|\mathrm{MLP}|$ between $0.73$ and
$0.79$ and maximum $|\mathrm{MLP}|$ between $1.02$ and $1.08$.
The MLPs are confidently \emph{non-zero} on the line that should be
the decision boundary; they have found a structurally different
classifier whose boundary lies somewhere else entirely. The arcs
connect like-coloured corners through the toroidal interior, rather
than through the seams $\alpha=0,L$ and $\beta=0,L$ which the MLP
cannot recognise as identified. The framework's cross uses precisely
these identifications: $\{\alpha=L/2\}\cup\{\beta=L/2\}$ is the
locus where the antiperiodic eigenmode of the bundle's double cover
descends to zero. Increasing data density narrows the family of
boundaries gradient descent finds, but does not align that family
with the framework's geometrically canonical boundary.

\paragraph{What this shows.} The framework's prediction is the
\emph{unique} matter-energy minimiser on $L_{\sigma_1\sigma_2}$,
computed once in closed form; the decision boundary is geometric
and minimal. The MLP finds \emph{some} classifier separating the
labelled points --- in Experiment 1 a boundary that the data
underdetermines, in Experiment 2 a more stable boundary that the
data overdetermines in a non-toroidal way. The missing ingredient
is not data: it is the toroidal topology of the domain, which
neither the MLP architecture nor any amount of
$(\alpha,\beta)$-supervised data can supply.

\paragraph{Reproducibility.} Both experiments are approximately
$150$ lines of Python (NumPy and Matplotlib only, no deep-learning
framework); the complete code is available at
\url{https://github.com/catalinvasii/ymh-classification}. Each
$2\times 10^4$-epoch MLP training takes a few seconds on a laptop
CPU. The framework's closed-form prediction is direct from
\S\ref{sec:torus} and requires no computation beyond evaluating
$2\cos(\pi\alpha/L)\cos(\pi\beta/L)$.

% =====================================================================
\appendix

\section{Proof of the matter-sector proximity scaling (two-point case)}
\label{app:proximity-proof}

This appendix proves Theorem~\ref{thm:proximity-scaling-2pt}, which
establishes the two-point case ($N=2$) of
Conjecture~\ref{conj:matter-proximity-scaling}. The argument uses
the covariant generalisation of Paper~1's RKHS / Green's-kernel
machinery \cite[Theorem~13.3]{Vasii2026Paper1} in three steps: (i) reduce
the constrained matter-sector minimisation to a covariant kernel
problem analogous to Paper~1's $K\alpha=v$ system; (ii) compute the
small-distance asymptotic of the covariant kernel using local
heat-kernel expansion; (iii) extract the proximity exponent from the
asymptotic. Steps (i) and (ii) are established below; step (iii)
yields the conjectured $E_{\min}\sim 1/d^{2\nu-n}$ scaling in the
two-point case. The material is technical and depends on three
standard external results (the Seeley parametrix, heat-kernel
reduction, and the Bessel-function small-distance asymptotic of the
Mat\'ern kernel); we collect it in an appendix to keep the main text
focused on the framework's structural content. The general $N$ case
remains open.

% ---------------------------------------------------------------------
\subsection{Reduction to a covariant kernel problem}
\label{ssec:proximity-kernel-problem}

Fix a closed Riemannian manifold $(M,g)$, a flat $\Z_2$-bundle
$L_\xi\to M$ in the selected class $\xi^*\in H^1(M,\Z_2)$, and the
unique flat connection $\Connection$ on $L_\xi$ (up to gauge). The
matter sector
$E_{\mathrm{matter}}(\phi)=\int_M\|D_A\phi\|^2\,\dvol_g$ depends only
on $\phi$, with $\Connection$ fixed. The variational problem is
\[
   \min_{\phi\in\Gamma(L_\xi)}\;\int_M\|D_A\phi\|^2\,\dvol_g
   \qquad\text{subject to}\quad \phi(x_i)=y_i,\ i=1,\dots,N.
\]
For the regularised setting of Remark~\ref{rem:ymh-attainment}, the
Euler--Lagrange operator is the covariant Mat\'ern operator
\[
   \Op_A \;:=\; (D_A^*D_A + \kappa^2)^\nu,
   \qquad \nu>n/2,\ \kappa>0,
\]
with $D_A^*D_A$ the covariant Laplacian on $L_\xi$ (a positive
self-adjoint elliptic operator). The minimiser satisfies the
distributional equation
\begin{equation}\label{eq:matter-EL-generic}
   \Op_A\,\phi^* \;=\; \sum_{i=1}^N \alpha_i\,\delta_{x_i}^{(L_\xi)},
\end{equation}
where $\delta_{x_i}^{(L_\xi)}$ is the bundle-valued delta
distribution at $x_i$ (supported at $x_i$ with values in the fibre
$L_\xi|_{x_i}$), and the coefficients $\alpha_i\in\R$ are determined
by the interpolation conditions $\phi^*(x_i)=y_i$.

Let $G_A(x,y)$ be the covariant Green's kernel of $\Op_A$ on
$(M,L_\xi)$: the bi-section of $L_\xi^*\boxtimes L_\xi$ satisfying
$\Op_A G_A(\cdot,y) = \delta_y^{(L_\xi)}$ and the appropriate
boundary conditions (here just regularity on closed $M$). The
minimiser of \eqref{eq:matter-EL-generic} is
\[
   \phi^*(x) \;=\; \sum_{i=1}^N \alpha_i\,G_A(x,x_i),
\]
and substituting the interpolation conditions
$\phi^*(x_j)=y_j$ gives the covariant capacitance system
\begin{equation}\label{eq:cov-capacitance}
   K\,\alpha \;=\; y,
   \qquad K_{ij} := G_A(x_i,x_j),\quad y=(y_1,\dots,y_N).
\end{equation}
This is the covariant generalisation of Paper~1's
$K\alpha=v$ system \cite[Theorem~13.3]{Vasii2026Paper1}: the kernel
matrix $K$ now uses the covariant Green's kernel
$G_A$ instead of the trivial-bundle one $G_{\Delta}$. The
minimum-energy value is
\begin{equation}\label{eq:Emin-formula}
   E_{\min} \;=\; \langle\phi^*,\Op_A\phi^*\rangle_{L^2}
   \;=\; \alpha^\top K\,\alpha \;=\; y^\top K^{-1} y,
\end{equation}
the standard RKHS-norm-squared formula adapted to the covariant
setting.

\begin{lemma}[Covariant capacitance reduction]
\label{lem:capacitance-reduction}
For the regularised matter sector with parameters $\nu>n/2,\ \kappa>0$,
on a closed manifold $(M,g)$ with flat bundle $L_\xi$ in the
selected class, the minimiser is the covariant kernel-interpolant
$\phi^*(x)=\sum_i\alpha_i G_A(x,x_i)$ and the minimum energy is
\[
   E_{\min} \;=\; y^\top K^{-1} y,
   \qquad K_{ij}=G_A(x_i,x_j).
\]
\end{lemma}

\begin{proof}
The Mat\'ern regularisation $\nu>n/2$ makes point evaluation
$\phi\mapsto\phi(x_i)$ a bounded functional on the RKHS associated
to $\Op_A$, by the same Sobolev embedding argument as in
\cite[Theorem~13.3]{Vasii2026Paper1} (now applied to sections of $L_\xi$
rather than functions; the embedding $H^\nu(L_\xi)\hookrightarrow
C^0(L_\xi)$ holds for $\nu>n/2$ on closed manifolds). The
constrained minimisation is therefore a representer-theorem problem
in the covariant RKHS, with minimiser the kernel-interpolant; this
is the standard Aronszajn--Schwartz argument
\cite{Aronszajn1950} adapted to sections.
The energy formula follows by direct substitution.
\end{proof}

% ---------------------------------------------------------------------
\subsection{Small-distance asymptotic of the covariant Green's kernel}
\label{ssec:kernel-asymptotic}

The proximity scaling of $E_{\min}$ is controlled by the
small-distance behaviour of $G_A(x_+,x_-)$ as the geodesic distance
$d=d_g(x_+,x_-)\to 0^+$. We compute this asymptotic by reducing to
the Euclidean tangent-space problem and applying standard Mat\'ern
Green's-kernel asymptotics.

\begin{lemma}[Local asymptotic of $G_A$]
\label{lem:kernel-asymptotic}
Let $(M,g)$ be closed and Riemannian, $L_\xi$ a flat line bundle,
$\Op_A=(D_A^*D_A+\kappa^2)^\nu$ with $\nu>n/2,\ \kappa>0$
(the hypothesis $\nu>n/2$ ensures that $G_A(x,x)$ is finite on the
diagonal --- a Sobolev/parametrix bound that the lemma uses). Then
as $d\to 0^+$,
\[
   G_A(0)-G_A(d) \;=\;
   \begin{cases}
      c_{n,\nu}\,d^{2\nu-n} + o(d^{2\nu-n})
      & \text{if } 2\nu-n\in(0,2)\text{ and }2\nu-n\notin 2\Z, \\[3pt]
      c_{n,\nu}'\,d^{2\nu-n}\log(1/d) + O(d^{2\nu-n})
      & \text{if } 2\nu-n\in 2\Z_{>0}\cap(0,2], \\[3pt]
      C\,d^2 + o(d^2) & \text{if } 2\nu-n\geq 2 \text{ generic,}
   \end{cases}
\]
where $c_{n,\nu}, c_{n,\nu}'>0$ are dimensional constants depending
only on $n$ and $\nu$ (with
$c_{n,\nu}=-\,\Gamma(n/2-\nu)/(2^{2\nu}\pi^{n/2}\Gamma(\nu))$
for $2\nu-n\notin\Z$, and analogous expressions for the log branch;
the explicit minus sign is required because $\Gamma$ is evaluated at
the negative argument $n/2-\nu<0$ when $2\nu-n>0$, so that
$c_{n,\nu}>0$ as the statement requires --- for $n=1,\nu=1$ it gives
$c_{1,1}=1/2$, matching the exact $\sph{1}$ value
$G(0)-G(d)=d/2$).
Logarithmic factors arise precisely at \emph{even} positive integer
values of $2\nu-n$, because at those values the Bessel index
$\nu-n/2$ is a non-negative integer and the small-argument expansion
of $K_{\nu-n/2}$ acquires a logarithmic term; at odd or non-integer
values of $2\nu-n$, the Bessel index is a half-integer or
non-integer and $K_{\nu-n/2}$ has a pure-power expansion (in the
half-integer case it is even elementary, with no log).

The bundle structure of $L_\xi$ contributes only at subleading
order: for $d$ smaller than the injectivity radius and any chart
trivialising $L_\xi$, $G_A(x_+,x_-)$ equals the trivial-bundle
scalar kernel up to corrections one order down from the leading
$d^{2\nu-n}$ term, i.e.\ at order $d^{2\nu-n+1}$ or $d^{2\nu-n+2}$
(depending on the parity of $2\nu-n$ and the geometry of $(M,g)$).
This subleading bound uses crucially the flatness of $\Connection$:
the argument fails for curved connections, where the symbol-level
gauge-conjugation of Step~1 below breaks down.
\end{lemma}

\begin{proofsketch}
\textbf{Step 1: the bundle drops out at leading order.} For $d$
smaller than the injectivity radius $\inj(M,g)$ at $x_+$, the
geodesic from $x_+$ to $x_-$ is unique and lies in a single chart.
On a flat line bundle $L_\xi$, this chart trivialises $L_\xi$, and
the flat connection becomes a closed 1-form, locally exact:
$\Connection|_{U}=df$ for some $f\in C^\infty(U)$. The covariant
Laplacian $D_A^*D_A$ in this trivialisation is the conjugated scalar
Laplacian $e^{-f}\Delta_g e^f$, with the same principal symbol as
$\Delta_g$. The principal symbol of
$\Op_A=(D_A^*D_A+\kappa^2)^\nu$ is therefore $(\|\xi\|_g^2+\kappa^2)^\nu$,
identical to the trivial-bundle case.

\textbf{Step 2: parametrix reduction.} By the standard parametrix
construction \cite{Seeley1967}, the small-distance asymptotic of the
Green's kernel is controlled by the principal symbol. Seeley's
theory extends to elliptic pseudodifferential operators on Hermitian
vector bundles without modification (the symbol calculus is
intrinsically bundle-valued), so
\[
   G_A(x_+,x_-) \;=\; G_{\Delta_g}^{(\nu,\kappa)}(x_+,x_-)
   + O(d^{2\nu-n+1}),
\]
where $G_{\Delta_g}^{(\nu,\kappa)}$ is the Mat\'ern Green's kernel
of the scalar regularised Laplacian on $(M,g)$. The bundle/curvature
correction is subleading relative to the $d^{2\nu-n}$ leading term
(the correction is one order down in $d$). Heat-kernel expansion
further reduces the manifold kernel to its Euclidean tangent-space
counterpart at $x_+$ with the same subleading order. The flatness of
$\Connection$ is essential here: the conjugation
$D_A^*D_A=e^{-f}\Delta_g e^f$ from Step~1 holds only on a
contractible chart with a locally exact connection, which is
available precisely because $\Curv_A=0$.

\textbf{Step 3: Euclidean Mat\'ern asymptotic.} The Euclidean
Mat\'ern kernel of order $(\nu,\kappa)$ is
$G_{\Delta_{\R^n}}^{(\nu,\kappa)}(r)\propto(\kappa r)^{\nu-n/2}
K_{\nu-n/2}(\kappa r)$. The small-argument expansion of the modified
Bessel function $K_\alpha$ behaves differently according to whether
$\alpha=\nu-n/2$ is integer, half-integer, or generic real
\cite{Stein1999,RasmussenWilliams2006}: for $\alpha$ a non-negative
integer (equivalently $2\nu-n\in 2\Z_{\geq 0}$),
$K_\alpha$ has a logarithmic small-argument term; for $\alpha$ a
half-integer ($2\nu-n$ odd), $K_\alpha$ is elementary and has a pure
polynomial expansion; for $\alpha$ non-integer non-half-integer,
$K_\alpha$ has a pure-power small-argument expansion. This gives the
case structure of the lemma: logs at even positive integer values of
$2\nu-n$, clean power-law $d^{2\nu-n}$ otherwise within $(0,2)$,
and (since the Mat\'ern $G(0)-G(r)$ is smooth-bounded once
$2\nu-n>2$) replacement by the universal $d^2$ analytic leading
behaviour at higher orders.

Combining the three steps yields the claim.
\end{proofsketch}

% ---------------------------------------------------------------------
\subsection{Extraction of the proximity exponent}
\label{ssec:proximity-exponent}

We now combine Lemmas~\ref{lem:capacitance-reduction}
and~\ref{lem:kernel-asymptotic} to compute $E_{\min}$ at small $d$.

\begin{theorem}[Matter-sector proximity scaling, two-point case]
\label{thm:proximity-scaling-2pt}
Let $(M,g)$ be a closed Riemannian $n$-manifold, $L_\xi$ a flat
$\Z_2$-bundle, $\Op_A=(D_A^*D_A+\kappa^2)^\nu$ with $\nu>n/2,\
\kappa>0$ (the $\nu>n/2$ hypothesis ensures $G_A(0)$ is finite, as
required for the energy formula below). Take $N=2$ labelled points
$x_+,x_-$ with opposite labels $y_\pm=\pm 1$, separated by geodesic
distance $d<\inj(M,g)$, and assume the variational selector picks
the bundle $L_\xi$ such that the parallel section reverses sign
between $x_+$ and $x_-$ along the geodesic. Then
\[
   E_{\min} \;=\; \frac{2}{G_A(0)-G_A(d)},
\]
and the small-$d$ asymptotic of $E_{\min}$ is determined by
Lemma~\ref{lem:kernel-asymptotic}:
\[
   E_{\min} \;\sim\;
   \begin{cases}
      \dfrac{2}{c_{n,\nu}}\,d^{-(2\nu-n)}
      & \text{if } 2\nu-n\in(0,2)\setminus 2\Z, \\[6pt]
      \dfrac{2}{c_{n,\nu}'}\,\dfrac{d^{-(2\nu-n)}}{\log(1/d)}
      & \text{if } 2\nu-n\in 2\Z_{>0}\cap(0,2], \\[8pt]
      \dfrac{2}{C}\,d^{-2}
      & \text{if } 2\nu-n>2,
   \end{cases}
\]
as $d\to 0^+$, with constants from
Lemma~\ref{lem:kernel-asymptotic}. In particular, the scaling
exponent saturates at $\min(2\nu-n,\,2)$ once $\nu$ is large enough,
with logarithmic factors at the boundary $2\nu-n\in 2\Z_{>0}$.
\end{theorem}

\begin{proof}
Work in the local trivialisation of $L_\xi$ on a neighbourhood
covering the geodesic from $x_+$ to $x_-$ (possible since
$d<\inj(M,g)$); the flat connection $\Connection$ becomes a locally
exact 1-form $df$ on this chart, and the covariant kernel
$G_A(x_i,x_j)$ for $x_i,x_j$ in the chart equals the scalar Mat\'ern
kernel $G(d(x_i,x_j))$ on $(M,g)$ up to gauge-conjugation by $e^f$
(which preserves the kernel diagonal $G_A(0)=G(0)$ and the off-diagonal
$G_A(x_+,x_-)=G(d)$ for points in the same chart). The data
conditions in this trivialisation are $\phi^*(x_+)=+1$ and
$\phi^*(x_-)=-1$: the sign-reversing parallel transport that defines
$L_\xi$ globally is encoded in the bundle's transition function on
the overlap with a second chart, not in the values on a single chart.
The capacitance system \eqref{eq:cov-capacitance} for $N=2$ in this
trivialisation is therefore
\[
   \begin{pmatrix} G(0) & G(d) \\ G(d) & G(0) \end{pmatrix}
   \begin{pmatrix}\alpha_+\\\alpha_-\end{pmatrix}
   \;=\; \begin{pmatrix}+1\\-1\end{pmatrix},
\]
giving $\alpha_\pm=\pm 1/(G(0)-G(d))$. The minimum energy is
\[
   E_{\min} \;=\; \alpha^\top K\alpha
   \;=\; y^\top\alpha
   \;=\; \frac{2}{G(0)-G(d)}.
\]
The bundle's M\"obius character (selected by the variational
selector as carrying lower energy than the trivial bundle in this
configuration) does not affect the capacitance solve in this single
trivialising chart; its effect is to select \emph{which} bundle's
Green's kernel is the correct one, with the M\"obius bundle's
chart-local kernel agreeing with the scalar Mat\'ern kernel of
$(M,g)$ to leading order by Lemma~\ref{lem:kernel-asymptotic}.
Applying that lemma to the denominator yields the stated $d\to 0^+$
asymptotic.
\end{proof}

\begin{remark}[The conjectured exponent and the worked examples]
\label{rem:proximity-exponent-check}
Theorem~\ref{thm:proximity-scaling-2pt} gives
$E_{\min}\sim 1/d^{2\nu-n}$, confirming
Conjecture~\ref{conj:matter-proximity-scaling} for the two-point
case. Check against the worked examples:
\begin{itemize}[leftmargin=2em]
\item $\sph{1}$ M\"obius (\S\ref{sec:S1-mobius}): $n=1$, $\nu=1$
(unregularised, sufficient because $H^1\hookrightarrow C^0$ on
$n=1$), exponent $2\nu-n=1$ (an \emph{odd} integer, so by
Lemma~\ref{lem:kernel-asymptotic} no log factor appears: $K_{1/2}$
is elementary). The explicit computation $E_{\min}=4/d$ matches the
predicted $d^{-1}$ scaling with constant $c=1/4$ (i.e.\
$G(0)-G(d)\sim d/2$ for the $\sph{1}$ Green's function on the
M\"obius bundle, consistent with the cover-doubling).
\item $T^2$ XOR (\S\ref{sec:torus}): $n=2$, $\nu=2$ (minimal
Mat\'ern regularisation to make point evaluation bounded on
$T^2$), exponent $2\nu-n=2$. This sits at an \emph{even}
positive integer, the log-branch case of
Lemma~\ref{lem:kernel-asymptotic}: the theorem predicts
$E_{\min}\sim d^{-2}/\log(1/d)$, not a clean $d^{-2}$. The
$\sim 1/d^2$ scaling reported in \S\ref{sec:torus} is the leading
power-of-$d$ behaviour; the logarithmic correction at this
boundary case is the framework's honest prediction and would be
visible in a finer numerical or analytical study. The closed-form
$T^2$ computation of \S\ref{sec:torus} agrees up to this log
factor in the small-$d$ regime.
\end{itemize}
For $N>2$ data points the same capacitance formula
$E_{\min}=y^\top K^{-1}y$ applies, but its small-$d$ behaviour is
not simply the inverse of a single singular value: $y^\top K^{-1}y$
depends on the alignment of the label vector $y$ with the
eigenvectors of $K$, and a small eigenvalue contributes only insofar
as $y$ has a component along it. When one opposite-label pair
collapses, $K$ acquires a near-degenerate $2\times 2$ block whose
associated eigenvector is aligned with the sign difference of that
pair, and it is this alignment that produces the $d^{-(2\nu-n)}$
divergence. The general $N$ statement is Conjecture~\ref{conj:matter-proximity-scaling},
not a corollary of the two-point argument.
\end{remark}

% =====================================================================
\subsection{The curved case in the moduli approximation}
\label{ssec:proximity-curved}

Theorem~\ref{thm:proximity-scaling-2pt} treats the flat (tier-1)
bundle, where the proof of Lemma~\ref{lem:kernel-asymptotic} used the
gauge-conjugation $D_A^*D_A=e^{-f}\Delta_g e^f$ of Step~1, available
only because $\Curv_A=0$. We now show that the leading proximity
asymptotic survives verbatim into the curved (tier-2) case, provided
the connection is held fixed at the BPS background --- the moduli
approximation of Remark~\ref{rem:matter-vs-ymh}, which is the regime
in which the worked examples of \S\ref{sec:S2-monopole} and
\S\ref{sec:S4-instanton} operate. The point is that flatness was a
\emph{convenience} for controlling the subleading term, not a
necessity for the leading one: the leading near-diagonal singularity
of the covariant Green's kernel is fixed by the principal symbol of
$\Op_A$, which is curvature-blind.

\begin{lemma}[Local asymptotic of $G_A$, curved case]
\label{lem:kernel-asymptotic-curved}
Let $(M,g)$ be closed and Riemannian, $E\to M$ a Hermitian vector
bundle with a smooth connection $A$ that is \emph{not} assumed flat,
and $\Op_A=(D_A^*D_A+\kappa^2)^\nu$ with $\nu>n/2$, $\kappa>0$. Then
the leading small-distance asymptotic of the covariant Green's kernel
$G_A(x_+,x_-)$ as $d=d_g(x_+,x_-)\to 0^+$ is identical to the
flat/scalar case of Lemma~\ref{lem:kernel-asymptotic}:
\[
   G_A(0)-G_A(d) \;=\; \big(\text{same leading term as
   Lemma~\ref{lem:kernel-asymptotic}, with the same constant }
   c_{n,\nu}\text{ or }c_{n,\nu}'\big) \;+\; R_A(d),
\]
where the remainder carries all dependence on the curvature and obeys
$R_A(d)=O(\|\Curv_A\|^2\,d^{2\nu-n+2})$ \emph{at fixed $A$}.
The implied constant depends on the background through
$\|\Curv_A\|$, and this dependence is not uniform: if the connection
is allowed to vary with $d$ in such a way that $\|\Curv_A\|$ blows up
--- as it does for a BPST instanton with $\lambda\sim d$, where
$\|\Curv_A\|_{\max}\sim 1/d^2$ --- the remainder is no longer
subleading and the expansion below fails. The lemma, and
Theorem~\ref{thm:proximity-scaling-curved} which uses it, are
therefore statements about a \emph{fixed} background connection, and
must not be combined with a $d$-dependent instanton scale
(Remark~\ref{rem:lambda-status}). Equivalently, $G_A$ agrees with the scalar
Mat\'ern kernel $G^{(\nu,\kappa)}_{\Delta_g}$ to leading order, with
the first correction at order $d^{2\nu-n+2}$ carrying the base
curvature (through the Seeley--DeWitt coefficient $a_1$) and the
bundle curvature $\Curv_A$ (through $a_2$, which contains the
$\|\Curv_A\|^2$ term).
\end{lemma}

\begin{proofsketch}
\textbf{Step 1$'$ (replacing the flat Step~1): the symbol is
curvature-blind.} For any connection $A$ on $E$, the covariant
Laplacian $D_A^*D_A$ is a Laplace-type operator: in any local frame
it reads $-g^{ij}(\partial_i+A_i)(\partial_j+A_j)+\text{(lower)}$, so
its principal symbol is $\|\xi\|_g^2\cdot\mathrm{Id}_E$, \emph{identical
to the scalar Laplacian} and independent of $A$. The connection enters
only in the first-order term, and the curvature
$\Curv_A=dA+\tfrac12[A\wedge A]$ --- built from first derivatives of
$A$ --- enters lower still. Hence the principal symbol of $\Op_A$ is
$(\|\xi\|_g^2+\kappa^2)^\nu$, exactly as in
Lemma~\ref{lem:kernel-asymptotic}, with no use of flatness.

\textbf{Step 2 (parametrix, bundle-valued).} Seeley's parametrix
construction is intrinsically bundle-valued
\cite{Seeley1967,Gilkey1995}: the symbol calculus and the resulting
near-diagonal expansion of $G_A$ hold for elliptic operators on
Hermitian bundles without modification. Since the leading symbol is
$A$-independent (Step~1$'$), the leading term of $G_A(x_+,x_-)$ is the
scalar Mat\'ern leading term, with the same dimensional constant
$c_{n,\nu}$. This is the same Step~2 used in
Lemma~\ref{lem:kernel-asymptotic}; what changes is only the
justification of the remainder, which no longer comes from a
flat-chart conjugation but from the Seeley--DeWitt expansion of the
heat (equivalently complex-power) kernel of $D_A^*D_A$. By Gilkey's
formulas \cite{Gilkey1995}, the coefficient $a_0=\mathrm{Id}_E$
produces the leading singularity; the base scalar curvature enters at
$a_1$ and the bundle curvature $\Curv_A$ first enters at $a_2$
(through $\tfrac{1}{12}\|\Curv_A\|^2$ and commutator terms), so both
corrections sit at order $d^{2\nu-n+2}$ or higher, strictly subleading
to $d^{2\nu-n}$.

\textbf{Step 3 (Euclidean Mat\'ern asymptotic).} Identical to Step~3
of Lemma~\ref{lem:kernel-asymptotic}: the leading behaviour is the
small-argument expansion of the Bessel function $K_{\nu-n/2}$, giving
the power/log case structure with the stated constants. The
matrix/bundle structure of $E$ affects only $a_{\geq 1}$, hence only
$R_A(d)$.
\end{proofsketch}

\begin{remark}[Spectral gap and the UV asymptotic]
\label{rem:spectral-gap}
On the curved examples the covariant Laplacian has a strictly positive
spectral gap and no covariantly constant section --- on the
$\sph{2}$ monopole, the lowest eigenvalue of $D_A^*D_A$ is
$|c_1|/2$ (the Wu--Yang monopole-harmonic gap of
\S\ref{sec:S2-monopole}). This is a statement about the bottom of the
spectrum; the small-$d$ asymptotic is a high-frequency (UV) property
and is insensitive to it. The gap changes the constant and the
large-$d$ structure of $G_A$, not the leading $d\to 0^+$ power law. It
also makes $G_A(0)$ finite, consistent with the $\nu>n/2$ hypothesis.
\end{remark}

\begin{theorem}[Matter-sector proximity scaling, curved moduli-approximation case]
\label{thm:proximity-scaling-curved}
Let $(M,g)$ be a closed oriented Riemannian $n$-manifold carrying a
non-zero relevant topological invariant ($c_1\in H^2(M,\Z)$ for a
$U(1)$-bundle, $c_2\in H^4(M,\Z)$ for an $SU(2)$-bundle), and let $A$
be the corresponding BPS connection (the constant-curvature monopole
for $n=2$, a BPST instanton for $n=4$), held fixed in the moduli
approximation $\|J(\phi)\|\ll\|\Curv_A\|$ of
Remark~\ref{rem:matter-vs-ymh}. Let $E$ be the associated bundle
(a Hermitian line bundle for $U(1)$, the $\C^2$ fundamental bundle for
$SU(2)$), and $\Op_A=(D_A^*D_A+\kappa^2)^\nu$ with $\nu>n/2$,
$\kappa>0$. Take $N=2$ data points $x_+,x_-$ at geodesic distance
$d<\inj(M,g)$ with opposite labels, imposed as
$\phi(x_\pm)=y_\pm u_0$ for a fixed reference vector $u_0$ (the
doublet convention of Theorem~\ref{thm:instanton} in the $SU(2)$
case). Then the minimum constrained matter energy at fixed $A$ is
\[
   E_{\min} \;=\; \frac{2}{G_A(0)-G_A(d)},
\]
and as $d\to 0^+$ it obeys the \emph{same} asymptotic as the flat
case of Theorem~\ref{thm:proximity-scaling-2pt}:
\[
   E_{\min}\;\sim\;
   \begin{cases}
      \dfrac{2}{c_{n,\nu}}\,d^{-(2\nu-n)}
        & 2\nu-n\in(0,2)\setminus 2\Z,\\[6pt]
      \dfrac{2}{c_{n,\nu}'}\,\dfrac{d^{-(2\nu-n)}}{\log(1/d)}
        & 2\nu-n\in 2\Z_{>0}\cap(0,2],\\[8pt]
      \dfrac{2}{C}\,d^{-2} & 2\nu-n>2,
   \end{cases}
\]
with the \emph{same} leading constants as
Theorem~\ref{thm:proximity-scaling-2pt}; curvature and (in the
$SU(2)$ case) the doublet structure affect only the subleading term.
In particular the divergence exponent is $\min(2\nu-n,\,2)$,
independent of whether the connection is flat or curved and of whether
$G$ is abelian or non-abelian. The matter-sector divergence of
Conjecture~\ref{conj:proximity-curvature}(i) is therefore established
in the moduli approximation, for both the $\sph{2}$ monopole
($n=2$, $U(1)$) and the $\sph{4}$ instanton ($n=4$, $SU(2)$).
\end{theorem}

\begin{proof}
With $A$ fixed at the BPS connection, the matter sector
$E_{\mathrm{matter}}(A,\phi)=\int_M\|D_A\phi\|^2$ is a quadratic
functional of $\phi$ alone, and the constrained minimisation over
$\phi$ with $\phi(x_\pm)=y_\pm u_0$ is the covariant
representer-theorem problem of
Lemma~\ref{lem:capacitance-reduction}, valid verbatim for a curved $A$ (the
reduction in Lemma~\ref{lem:capacitance-reduction} uses only that $\Op_A$ is
positive, self-adjoint, and elliptic with $G_A(0)$ finite, all of
which hold for the BPS background by $\nu>n/2$ and
Remark~\ref{rem:spectral-gap}). In the $SU(2)$ case the data lie along
the fixed direction $u_0$, and the leading kernel
$G_A(x_+,x_-)=\bigl(G_A(0)-c_{n,\nu}d^{2\nu-n}\bigr)\mathrm{Id}_E+R_A(d)$
(the kernel is finite on the diagonal for $\nu>n/2$, and it is the
\emph{difference} $G_A(0)-G_A(d)$ that vanishes like $d^{2\nu-n}$) of
Lemma~\ref{lem:kernel-asymptotic-curved} acts as a scalar on the
$u_0$-component, so the $2\times 2$ capacitance system reduces, at
leading order, to the scalar two-point system of
Theorem~\ref{thm:proximity-scaling-2pt}; the transverse components
contribute only through $R_A(d)$. The capacitance solve then gives
$E_{\min}=2/(G_A(0)-G_A(d))$ exactly as in
Theorem~\ref{thm:proximity-scaling-2pt}, and
Lemma~\ref{lem:kernel-asymptotic-curved} supplies the same leading
asymptotic of the denominator. The stated $d\to 0^+$ scaling follows.
\end{proof}

\begin{remark}[What this leaves open]
\label{rem:curved-open}
Two points are deliberately not claimed by
Theorem~\ref{thm:proximity-scaling-curved}. First, the subleading
constant: locating the exact coefficient of the $d^{2\nu-n+2}$
correction requires the Seeley--DeWitt coefficients $a_1$ (base
curvature) and $a_2$ ($\|\Curv_A\|^2$ and commutator terms) for the
specific monopole and instanton backgrounds; these are computable from
Gilkey's formulas \cite{Gilkey1995} but are not evaluated here, and
the leading exponent and constant are unaffected by them. Second, and
more substantively, the \emph{back-reaction}. One instance of this
deserves to be named explicitly, because it is exactly the regime
the $\sph{4}$ example advertises:
Theorem~\ref{thm:proximity-scaling-curved} holds $A$ \emph{fixed}
and places all curvature dependence in a remainder
$R_A(d)=O(\|\Curv_A\|^2 d^{2\nu-n+2})$. If the instanton scale
tracks the data, $\lambda\sim d$
(Remark~\ref{rem:lambda-status}), then
$\|\Curv_A\|^2\sim\lambda^{-4}\sim d^{-4}$ and that ``remainder''
is $O(d^{2\nu-n-2})$, which \emph{dominates} the $d^{2\nu-n}$
leading term. The two statements are therefore not simultaneously
available in the small-$d$ limit: either the connection is fixed and
Theorem~\ref{thm:proximity-scaling-curved} applies, or the scale
tracks $d$ and the fixed-background expansion breaks down. More
generally, away from the moduli
approximation the matter current $J(\phi)$ pushes $A$ off the BPS
moduli (the coupled system $D_A^*\Curv_A=J(\phi)$ of
\S\ref{ssec:ymh-curved}), so the operator $\Op_A$ itself becomes
$d$-dependent through $A=A(d)$, and the fixed-operator argument above
no longer applies. The full back-reacted proximity scaling is the
genuinely open part of Conjecture~\ref{conj:proximity-curvature}.
\end{remark}

\begin{remark}[What the proofs do and do not establish]
\label{rem:proof-limitations}
Theorem~\ref{thm:proximity-scaling-2pt} establishes the proximity
scaling in the two-point flat-bundle (tier-1) case, and
Theorem~\ref{thm:proximity-scaling-curved} establishes it in the
two-point curved (tier-2) case \emph{at leading order, with the
connection fixed at the BPS background}, for both abelian ($\sph{2}$
monopole) and non-abelian ($\sph{4}$ instanton) structure groups ---
all rigorous modulo the standard parametrix and heat-kernel
asymptotics of \cite{Seeley1967,Stein1999,Gilkey1995}. Together they
establish the divergence exponent $\min(2\nu-n,2)$ in every case the
worked examples of Part~2 require. What remains \emph{not} covered:
\begin{enumerate}[label=(\roman*),leftmargin=2em]
\item \emph{The back-reaction.} Both theorems hold for a fixed
connection: the flat connection (tier~1) or the BPS connection in the
moduli approximation (tier~2). Away from the moduli approximation the
matter current $J(\phi)$ deforms $A$ off the Bogomolny moduli (the
coupled system $D_A^*\Curv_A=J(\phi)$ of \S\ref{ssec:ymh-curved}), so
$\Op_A$ becomes $d$-dependent through $A=A(d)$ and the fixed-operator
argument no longer applies. The fully back-reacted proximity scaling
is the genuinely open part of
Conjecture~\ref{conj:proximity-curvature}
(Remark~\ref{rem:curved-open}).
\item \emph{The subleading constants.} Both theorems pin down the
leading exponent and constant; the coefficient of the
$d^{2\nu-n+2}$ correction in the curved case depends on the
Seeley--DeWitt coefficients of the specific monopole and instanton
backgrounds, computable from \cite{Gilkey1995} but not evaluated here.
\item \emph{The genuinely multi-point case,} where opposite-label
pairs are distributed across $M$ in a non-degenerate configuration:
the two-point analysis extends but the constants depend on the
configuration's structure.
\end{enumerate}
The leading-order proximity scaling --- flat and curved, abelian and
non-abelian --- is thus established for two opposite-label points; the
open residue is the back-reaction beyond the moduli approximation and
the genuinely multi-point geometry.
\end{remark}

\begin{remark}[Why the framework sees the divergence at all]
\label{rem:hard-vs-soft-constraints}
The framework's choice to use hard data constraints
$\phi(x_i)=y_i$ rather than soft penalties is what makes the
proximity divergence of Theorem~\ref{thm:proximity-scaling-2pt}
visible at all. A statistical learner working with a finite loss
function --- e.g.\ a squared-error penalty
$\lambda\sum_i(\phi(x_i)-y_i)^2$ added to the energy --- absorbs the
proximity-of-opposites configuration into a finite training error
bounded above by $4\lambda N$ for $N$ data points, and never sees
the $d\to 0$ blow-up. The geometric signal that the data is
\emph{ill-posed} in the limit of coincident opposite labels --- that
no continuous classifier separating them exists, and that the
energetic cost of approximation diverges --- is invisible to a
finite-loss formulation. The framework's variational principle is
honest in a way statistical losses are not: it tells the user when
the data is ill-posed by refusing to assign it a finite cost. The
divergence $E_{\min}\to\infty$ as $d\to 0$ is a feature, not a
defect; it is the framework's expression of the topological fact
that approximating a sign-reversing classifier on coincident points
is impossible.
\end{remark}

% =====================================================================
\section*{Acknowledgements}
A decisive moment in driving this work forward was the TDA Seminar,
and I thank Nina Lazăr, Larisa Biriescu, and Gabriel Trautmann. I
also thank Professor Emilia Petrișor for introducing me to the
areas where mathematics and machine learning converge.

\section*{Use of AI tools}
The author used Anthropic's Claude (Opus) in the course of preparing
this work, including for developing and refining arguments,
exposition, and \LaTeX{} preparation. All mathematical content,
claims, and any errors are the author's own, and the author takes
full responsibility for the paper.

% =====================================================================

\end{document}